\documentclass[twoside,english]{amsart}
\usepackage{color}
\usepackage{geometry} 
\usepackage{babel}
\usepackage{amsbsy}
\usepackage{amsbsy}
\usepackage{amssymb,latexsym}
\usepackage{amsmath,amsthm}
\usepackage{pdfpages}
\usepackage{mathtools}
\usepackage{rotating}
\usepackage[emtex]{t-angles}
\usepackage{graphicx}
\usepackage{epstopdf}
\usepackage{tikz}
\usepackage{pgffor}
\usetikzlibrary{matrix,arrows}
\usetikzlibrary{intersections}
\usetikzlibrary{decorations.markings,positioning}

\newcommand{\codim}[1]{\mathrm{codim}(#1)}

\newtheorem{definition}{Definition}

\newtheorem{theorem}{Theorem}
\newtheorem{proposition}[theorem]{Proposition}
\newtheorem{lemma}[theorem]{Lemma}
\newtheorem{corollary}[theorem]{Corollary}

\theoremstyle{remark}
\newtheorem{remark}{\textsc{Remark}}

\newtheorem{example}{\textsc{Example}}

 \def\dpol{\mathop{{}^{\textsc{D}}\text{Pol}}}

\newcommand{\CP}{\mathbb{P}}
\newcommand{\C}{\mathbb{C}}
\newcommand{\RR}{\mathbb{R}}
\newcommand{\M}{\overline{\mathcal{M}}}

 \title[Geometric invariants of the configuration space]{Geometric invariants 
 of the  configuration space of $d$ marked points on  the complex plane}

\author{N. Combe}
\address{Max Planck Institute for Mathematics \\ Vivatsgasse 7\\ 53111 Bonn}
\email{combemaths@gmail.com}
\date{}
\thanks{
I wish to thank Norbert A' Campo for many interesting discussions, Leila Schneps and Pierre Lochak for helpful suggestions during my stay at Paris 6 and as well Hans Henrik Rugh for valuable discussions. I thank the Max Planck Institute for Mathematics in Bonn for the hospitality. In particular, I am grateful to Yuri I. Manin for giving me many advice. I thank the referee for constructive comments and suggestions.}

\begin{document}
\begin{abstract}
Interest in Conformal Field Theories and Quantum Field Theory lead physicists to consider configuration spaces of marked points on the complex projective line, $Conf_{0,d}(\CP)$. 

 In this paper, a real semi-algebraic  stratification of $Conf_{0,d}(\C)$, invariant under Coxeter-Weyl  group is constructed, using the natural relation of this configuration space with the space $\dpol_d$ of complex monic degree $d>0$ polynomials in one variable with simple roots. This decomposition relies on subsets of $\dpol_d$ forming a good cover in the sense of \v Cech of $\dpol_{d}$ and such that each piece of the decomposition is a set of polynomials, indexed by a decorated graph reminiscent of {\it Grothendieck's  dessins d'enfant}.
This decomposition in Coxeter-Weyl chambers brings into light a very deep interaction between the {\it real locus} of the moduli space $\M_{0,d}(\RR)$ and {\it the complex} one $\M_{0,d}(\C)$. Using this decomposition, the existence of geometric invariants of those configuration spaces has been shown. Many examples are provided. Applications of these results in braid theory are discussed, namely for the braid operad.
\end{abstract}
\maketitle
\tableofcontents

 \section{Introduction}
Interest in Conformal Field Theories and Quantum Field Theory lead physicists to consider configuration spaces of marked points on the complex projective line, $Conf_{0,d}(\CP)$ and the moduli spaces of marked points on genus $0$ curves. 
The properties of moduli spaces of genus 0 curves with $d$ unordered marked points $\mathcal{M}_{0,d}(\mathbb{C})$ are still an important subject of investigations. Those spaces lie at the heart of challenging problems in relation with the calculation of Gromov-Witten invariants, a particular importance is given to them.

 
In this paper,  a Coxeter-Weyl semi-algebraic stratification of this moduli space is given, highlighting the existence of new topological invariants of the space of configurations of $n$ marked points on the complex plane. We use the natural relation between the moduli space $\mathcal{M}_{0,d}(\mathbb{C})$ and the space $\dpol_d$ of complex, monic, degree $d>0$ polynomials in one variable, having simple roots with sum equal to zero (i.e. Tschirnhausen polynomials)~\cite{C0}:  $\mathcal{M}_{0,[d]}$ is the quotient of the $d$-th {\it unordered} configuration space on the complex plane modulo the group $PSL_{2}(\mathbb{C})$ and the configuration space can be considered as the space $\dpol_d$, due to the fundamental theorem of algebra. 



This approach allows a new insight on $\mathcal{M}_{0,[d]}(\mathbb{C})$ and $\mathcal{M}_{0,[d]}(\mathbb{R})$, since new geometric invariants of configuration of points on $\mathbb{C}$ are given. Moreover, this decomposition being invariant under a Coxeter-Weyl group it brings into light a very deep relation between the approach used to consider the real locus of the moduli space of marked points on the sphere $\mathcal{M}_{0,d}(\mathbb{R})$  and the classical complex one $\mathcal{M}_{0,d}(\mathbb{C})$. 

As an application to braid theory, it gives an alternative way to describe braid generators and to construct the braid operad.   
To prove the existence of new geometric invariants, we introduce a decomposition of $\dpol_{d}$, which leads to the construction of a good cover in the sens of \v Cech. The nerve of this covering provides the new {\it geometric invariant}. It turns out, that those geometric invariants have many symmetries, in particular polyhedral ones. This decomposition is based  on the Tits-Bruhat-Deligne theory of chambers and galleries. In the following, by {\it Weyl-Coxeter chamber}~\cite{De},\cite{Bki}(Chap IX, Sect. 5.2) we mean  a fundamental domain, along with reflections hyperplanes. The method we present has the following advantageous property: namely, it allows to define any braid relation in $B_{d}$, as a path in a {\it Weyl-Coxeter gallery}. 

The study of the configuration spaces with $d$ marked points in the complex plane was initiated in the early times, 1960-1970, by V. Arnold and, further developed by the Russian school: V. Arnold, V. Goryunov, O. Lyashko, A. Vassiliev, D. Fuchs, S. Chmutov, S. Duzhin, J. Mostovoy\quad~~\cite{{Ar70},{AVG1},{AVG2},{CH},{Mo}}. 

During the last 50 years, the compactified moduli space $\overline{\mathcal{M}}_{0,[n]}(\mathbb{C})$ has been essentially  studied from a complex geometry point of view. This point of view was introduced by Deligne-Mumford~\cite{DeM}, developed then by Knudsen~\cite{Knu} and Keel~\cite{Kee}. Other versions were provided: for instance by Kapranov~\cite{Ka93v1}, using Tits-Bruhat buildings and Fulton-MacPherson~\cite{FM}. 
Approaches using real geometry, have only been considered until now in the case of real {\it ordered} points, i.e. for $\overline{\mathcal{M}}_{0,n}(\mathbb{R})$. Namely, in~\cite{Dev}, Devadoss introduces his mosaic tesselation in order to construct the mosaic operad. In~\cite{EHKR}, the real cell decomposition is used to consider the cohomology of the real locus of those moduli spaces. Some other results concerning this real locus are due to Khovanov~\cite{Kho96} and Davis, Januszkiewicz and Scott~\cite{DJS}. 


The decomposition we provide, using the real-algebraic geometry point of view, allows a rich description of the open algebraic variety $\mathcal{M}_{0,[n]}(\mathbb{C})$, which is {\it not} visible in the framework of the previous decompositions.   
However, interesting connections exist with the mosaic tesselation of $\overline{\mathcal{M}}_{0,n}(\mathbb{R})$ and with the one given by Kapranov~\cite{Ka93v1,Ka93}. Namely, the set of generic cells of our decomposition are in bijective correspondence with pairs of all $d$ bracketings and, therefore, related to the associahedron and to the Stasheff polytope~\cite{Sta}. Note that, somehow, the main difference in the combinatorial structure is due to the fact the letters in the bracketings do not matter in our new decomposition, since we consider the unordered points case.
 
The present paper grew out of the previous works~\cite{NAC,C0,C1,Co1}, where we bring into light the existence of a topological real algebraic stratification, obtained through the notion of {\sl drawings} $\mathcal{C}_{P}$ of a polynomial $P$. Such an object is  reminiscent of Grothendieck's {\it dessin's d'enfant}~\cite{Gro84} in the sense that we consider the inverse image of the real and imaginary axes under a complex polynomial. The {\sl drawing} associated to a complex polynomial is, by convention, a system of blue and red curves properly embedded in the complex plane, being the inverse image under a polynomial $P$ of the union of the real axis (colored in blue) and the imaginary axis (colored  in red)~\cite{C1}, $\mathcal{C}_{P}=P^{-1}(\mathbb{R}\cup \imath \mathbb{R})$. For a polynomial $P\in \dpol_{d}$, the drawing contains $d$ blue and $d$ red curves, each blue curve intersecting exactly one red curve.  The entire drawing forms a forest (in terms of graphs), whose leaves (terminal vertices) go to infinity in the asymptotic directions of the angle $\pi/2d$. 

The origins of the idea of considering the configuration space as a space of $\mathbb{C}$-polynomials, was introduced by V. Arnold~\cite{Ar70}. 
In 1992, S. Barannikov~\cite{Ba} restudied problems concerning this space of polynomials, in the light of works of Gauss~\cite{Gauss}. More precisely, he used Gauss's approach to the study of complex polynomials (a polynomial is uniquely determined by taking the inverse image under this polynomial of the real and imaginary axis). This approach was recently used by E. Ghys, see~\cite{{GhysA},{Ghys}} (p. 72) to discuss a question of M. Kontsevitch, in 2009, on the intersection of polynomials, which lead to the following theorem: 

four polynomials, $P_1,P_2,P_3,P_4$ of a real variable $x$ cannot satisfy: 
\begin{itemize} \item $P_1(x)<P_2(x)<P_3(x)<P_4(x)$ for small $x<0$, \item $P_2(x)<P_4(x)<P_1(x)<P_3(x)$ for small $x>0$ .
\end{itemize}

N. A'Campo, introduces in~\cite{NAC}, the idea to use combinatorics of bi-colored forest to construct  a real semi-algebric decomposition of the space of complex polynomials with distinct roots. A'Campo's construction is based on equivalence classes  (in the sens of Barannikov) of Gauss' drawings.  As well,  he shows that a representative of an equivalence class is a polynomial and vice versa and  proves that such an equivalence class forms a contractible set in $\dpol_d$.

In this paper, the  isotopy classes of polynomial's drawings, relatively to the $4d$ asymptotic directions generate our construction. In order to avoid any confusion, we call   {\it elementa}, this isotopy classes. This decomposition in elementa allows us, not only to define a {\it topological stratification} of $\dpol_{d}$ but as well, generates the construction of a good \v Cech cover of $\dpol_{d}$~\cite{C1,Co1}. 

This approach can be applied to braid theory. It is known that the fundamental group of the configuration space of $d$ unordered marked points on the complex plane  is a $d$ strand braid group.  Therefore, any braid relation can be investigated in a new manner, using the decomposition in Weyl-Coxeter chambers. In particular, we bring a new insight on braid towers, which is obtained using the natural inclusions: $B_{d}\to B_{d+1}$. 
In addition, this gives a new insight on the braid operad. 

An operad $O_{*} = \{O_{k},k \geq 1\}$ is a collection of spaces together with some composition maps
$ O_{n} \times O_{i_{1}} \times \dots \times O_{i_{n}} \to O_{i} $(where $i = \sum_{k=1}^{n} i_{k}$) satisfying some axioms, with cabling
\[ B_d \times B_{i_{1}} \times...B_{i_{d}}\to B_{i},\] as composition, with $i=\sum_{k=1}^{d} i_k$. 

Moreover, this decomposition is also good in the sense of \v Cech. So, as a second application, one can explicitly calculate the cohomology of the braid group with values in a sheaf.  
This information turns out to be important, since by~\cite{KSV} it is known that monoidal functors preserve operads; hence the homology of an operad (in spaces) as composition is an operad in graded modules.
The paper is composed of four main points.

In section 2, we introduce the notion of elementa, signatures, strata and as well as their properties. We discuss as well the Whitehead moves on the signatures. 

In section 3, we investigate the classification of generic signatures. We show that there exist four classes of generic signatures: $M,F,S$ and $FS$. 
Two elementa are incident if the signature of one can be obtained from the other one by a so-called {\it half-Whitehead move} (this is a topological operation on the red (or blue) edges of a given signature modifying one signature into the other one). This incidence relation on elementa is deeply connected to the topological closure of each topological stratum: 

In the section 4, we prove that the elementa are endowed with an incidence relation, which forms a partially strictly ordered set.  We introduce the notion of inclusion diagram of the partially strictly ordered set and show that this inclusion diagram is the \v Cech nerve~\cite{Ce,C1}. This gives a geometric invariant.  
The construction is based on properties of generic elementa in such a way that each vertex corresponds to a generic elementum, each edge corresponds to a elementum of codimension 1 and  each 2-face to a elementum  of codimension 2  (etc). The incidence relations between the elementa are preserved on such a diagram. 

The construction of the inclusion diagram is illustrated  for low degrees: $d=2,3,4$. In particular, we pay attention to
the case $d=4$, we have a very rich geometrical structure. 

In the section 5, we prove the main statement: the decomposition is invariant under a Coxeter group and present a method of construction. 
We prove the main theorem:
\begin{theorem}[Main theorem]
Let $d>3$. The decomposition of $\dpol_{d}$ in $4d$ Weyl-Coxeter-chambers, induced by the topological stratification of $\dpol_d$ in elementa, is invariant under the Coxeter group given by the presentation: \[W=\langle r, s_1,...,s_{2d+1} | r^2= Id, (rs_j)^{2p}=Id, p|d\rangle, p\in \mathbb{N}^{*}\]
\end{theorem}
Note, that this structure is reminiscent of translation surfaces. 

As an application to braid theory, we show that: 
\begin{theorem}
Any braid relations following from the natural inclusion of braids $i^{\star}: B_d  \hookrightarrow B_{d+1}$ can be described using the decomposition in $Q$-pieces. 
\end{theorem}

Braids and the natural inclusions $i^{\star}: B_d  \hookrightarrow B_{d+1}$ can be read from this geometric decomposition in chambers and galleries. More over, this method can be used as an alternative description of the braid operad.

\section{Elementa, signatures and classes of polynomials}

The isotopy classes of drawings of polynomials of $\dpol_{d}$, relatively to their $4d$ asymptotic directions are the basic objects of our decomposition of this space. So, in the following, we call them {\it elementa}. Since to each drawing corresponds a unique polynomial $P\in \dpol_{d}$ to an elementa $A$ correspond a subset of $\dpol_{d}$ which will be denoted by the same symbol.

\begin{definition}
A chord diagram of degree $d$ consists in an oriented circle with  $ 2d$ distinct points   in counterclockwise order and  a distinguished set of $d$ disjoint pairs of points. Every two points belonging to the same pair are joined by a chord (diagonal).

A forest-chord diagram is a chord diagram such that  the intersection of chords cannot form any $k$-gon in the disc, but only a vertex of valency $2k$.
    
\end{definition}
\begin{definition}[Signature]~\label{De:1}
A degree $d$ signature is a forest embedded in a unit disc $\mathbb{D}$, with points (the terminal vertices) $S=\{0,1,2,.., 4d-1\}$ on $\partial\mathbb{D}$  lying on the $4d$ roots of the unity, which is obtained by a superimposition  of two degree $d$ forest-chord diagrams: one on the even points $B=\{0,2,\dots,4d-2\}$  called the blue forest-chord diagram and the other one on the odd points $R=\{1,3,\dots,4d-1\}$ the red forest-chord diagram,
with the constraint that red and blue forest chords intersect at $d$ points, such that one blue chord intersects only one red chord. 
\end{definition}

 Let us notice that a signature of degree $d$ has three types of vertices:
  \[\mathcal{V}=\{v, \bar{v},\bar{\bar{v}}\mid\, v\in\partial{\mathbb{D}},\,\bar{v},\bar{\bar{v}}\in int(\mathbb{D}),\, |v|=4d,\, |\bar{v}|=d,\, 0\leq |\bar{\bar{v}}|\leq d-1\},\]
with $ deg(v)=1,\, deg(\bar v)=4,\, 0\leq deg(\bar{\bar{v}})\leq 2d.$
 
\begin{figure}[h]
\begin{center}
\begin{tikzpicture}[scale=1.1]
\newcommand{\degree}[0]{12}
\newcommand{\last}[0]{23}
\newcommand{\XDOT}[0]{[black](i-1) circle (0.045)}
\newcommand{\B}[1]{#1*180/\degree}
\newcommand{\BDOT}[0]{[blue](i-1) circle (0.045)}
\newcommand{\R}[1]{#1*180/\degree}
\newcommand{\RDOT}[0]{[red](i-1) circle (0.045)}
\draw (0,0) circle (1) ;
\foreach \k in {0,2,4,6,8,10,12,14,16,18,20,22} {
  \draw[blue] (\B{\k}:1.19) node {\Tiny\k} ;
}
\foreach \k in {1,3,5,7,9,11,13,15,17,19,21,23} {
  \draw[red] (\R{\k}:1.19) node {\Tiny \k} ;
}
\draw[blue,  name path=B0B2] (\B{0}:1) .. controls(\B{0}:0.7) and (\B{2}:0.7) .. (\B{2}:1) ;
\draw[blue,  name path=B4B6] (\B{4}:1) .. controls(\B{4}:0.7) and (\B{6}:0.7) .. (\B{6}:1) ;
\draw[blue, name path=B8B18] (\B{8}:1) .. controls(\B{8}:0.15) and (\B{18}:0.15) .. (\B{18}:1) ;
\draw[blue, name path=B10B20] (\B{10}:1) .. controls(\B{10}:0.7) and (\B{20}:0.7) .. (\B{20}:1) ;
\draw[blue,  name path=B12B14] (\B{12}:1) .. controls(\B{12}:0.7) and (\B{14}:0.7) .. (\B{14}:1) ;
\draw[blue,  name path=B16B22] (\B{16}:1) .. controls(\B{16}:0.05) and (\B{22}:0.05) .. (\B{22}:1) ;
\draw[red, name path=R1R7] (\R{1}:1) .. controls(\R{1}:0.7) and (\R{7}:0.7) .. (\R{7}:1) ;
\draw[red, name path=R3R21] (\R{3}:1) .. controls(\R{3}:0.1) and (\R{21}:0.1) .. (\R{21}:1) ;
\draw[red, name path=R5R23] (\R{5}:1) .. controls(\R{5}:0.7) and (\R{23}:0.7) .. (\R{23}:1) ;
\draw[red, name path=R9R15] (\R{9}:1) .. controls(\R{9}:0.7) and (\R{15}:0.7) .. (\R{15}:1) ;
\draw[red, name path=R11R13] (\R{11}:1) .. controls(\R{11}:0.7) and (\R{13}:0.7) .. (\R{13}:1) ;
\draw[red, name path=R17R19] (\R{17}:1) .. controls(\R{17}:0.7) and (\R{19}:0.7) .. (\R{19}:1) ;

\fill[name intersections={of=B10B20 and B8B18,name=i}] \BDOT ;
\fill[name intersections={of=R1R7 and R5R23,name=i}] \RDOT ;
\fill[name intersections={of=B12B14 and R11R13,name=i}] \XDOT ;
\fill[name intersections={of=B10B20 and R9R15,name=i}] \XDOT ;
\fill[name intersections={of=B8B18 and R17R19,name=i}] \XDOT ;
\fill[name intersections={of=B16B22 and R3R21,name=i}] \XDOT ;
\fill[name intersections={of=B4B6 and R5R23,name=i}] \XDOT ;
\fill[name intersections={of=B0B2 and R1R7,name=i}] \XDOT ;
\end{tikzpicture}
\end{center}
\vspace{-10pt}
\caption{An example of signature of degree 6. The vertices: $v$ are indexed by numbers, $\bar v$ by black points and $\bar{\bar{v}}$ by red or blue points along  color of the chords }
\end{figure}
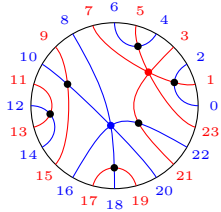
Notice that: drawings are geometrical objects, while signatures are combinatorial ones. Indeed, a drawing $\mathcal{C}_{P}$  characterizes completely the polynomial $P$: coordinates of its roots, of its critical points, singularities of its real and imaginary part. The signatures characterizes only the isotopy classes, to each signature $\sigma$ corresponds an elementum, denoted $A_{\sigma}$. In a signature the blue and red lines (drawn in a curved way) characterize only asympthotic directions of drawing  and the  localisation of the  critical points and values of the classes of polynomials indexed by this signature.
 
\begin{definition}[Short and long chords]~\label{D:diagonalsij} 
Let $\sigma$ be a signature.
A chord (diagonal) connecting a pair of terminal vertices $i,j\in  \{0,...,4d-1\}$ of the same parity, is denoted by $(i,j)$ . 
\begin{itemize}
\item A chord is short if $|i-j|=2$.
\item A chord  is long if it is not short.
\end{itemize}
\end{definition}

\begin{figure}[h]
\begin{center}
\begin{tikzpicture}[scale=1.1]
\newcommand{\degree}[0]{8}
\newcommand{\B}[1]{#1*180/\degree}
\newcommand{\R}[1]{#1*180/\degree }
\draw[black] (0,0) circle (1) ;

\foreach \k in {0,2,4,6,8,10,12,14} {
  \draw[blue] (\B{\k}:1.19) node {\Tiny\k} ;
}
\foreach \k in {1,3,5,7,9,11,13,15} {
  \draw[red] (\R{\k}:1.19) node {\Tiny \k} ;
}
\draw[blue, name path=B2B4](\B{2}:1) .. controls(\B{2}:0.7)and (\B{4}:0.7) .. (\B{4}:1) ;
\draw[blue, name path=B8B6](\B{8}:1) .. controls(\B{8}:0.7)and (\B{6}:0.7) .. (\B{6}:1) ;
\draw[blue, name path=B0B10](\B{0}:1) .. controls(\B{0}:0.7)and (\B{10}:0.7) .. (\B{10}:1) ;
\draw[blue, name path=B12B14](\B{12}:1) .. controls(\B{12}:0.7)and (\B{14}:0.7) .. (\B{14}:1) ;
\draw[red, name path=R3R5](\R{3}:1) .. controls(\R{3}:0.7)and (\R{5}:0.7) .. (\R{5}:1) ;
\draw[red, name path=R7R9](\R{7}:1) .. controls(\R{7}:0.7)and (\R{9}:0.7) .. (\R{9}:1) ;
\draw[red, name path=R11R13](\R{11}:1) .. controls(\R{11}:0.7)and (\R{13}:0.7) .. (\R{13}:1) ;
\draw[red, name path=R15R1](\B{15}:1) .. controls(\B{15}:0.7)and (\B{1}:0.7) .. (\B{1}:1) ;
\end{tikzpicture}
\end{center}
\vspace{-5pt}
\caption{An example of a long blue chord $(0,10)$, intersecting a short red chord $(1,15)$.}
\end{figure}
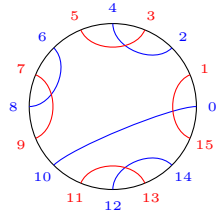

\begin{definition}
 A signature is  generic if it does not contain any vertex $\bar{\bar{v}}$. An elementum $A_{\sigma} $ is generic if the signature $\sigma$ is generic.
 \end{definition}
 
 \begin{definition}[Codimension]\label{D:cod}
The {\rm  real codimension} of an elementum $A_{\sigma}$ is the sum of the local indices of all the vertices $\bar{\bar{v}}$ of the signature $\sigma$. The local index $Ind_{loc}(\bar{\bar{v}})$ at a vertex $\bar{\bar{v}}$ is the number $Ind_{loc}(\bar{\bar{v}})=2deg(\bar{\bar{v}})-3$. 
\end{definition}

All the generic elementa are of codimension 0.

\begin{definition}\label{D:suc}
Let $\sigma$ be a signature. 
\begin{itemize}
\item A 2-cell in $\mathbb{D}\setminus\sigma$ contains in its boundary a set of its edges of a tree. Such a tree is said to bound this 2-cell. 
\item Two trees are adjoining if they lie in the boundary of the same 2-cell in $\mathbb{D}\setminus \sigma$.
\item Consider a pair of chords of the same color in a generic signature: $(i,j)$ and $(k,l)$. We call them successive if their terminal vertices satisfy one of the following conditions: $|i-k|=2$ or $|i-l|=2$ or $|j-k|=2$ or $|j-l|=2$.
\end{itemize}
\end{definition}

\begin{proposition}[\cite{C0,CoJ}]
For every $d$, the number of elementa is finite.
\end{proposition}  

Let us introduce the  topological operations on the signatures: the half-Whitehead moves which allows to define the topological closure of $A_{\sigma}$.
 \begin{definition}[\cite{C0}]~\label{D:2.18}
A half-Whitehead move is a topological operation on the chords of a signature, carried out in the following way: 
 \begin{enumerate}

\item  A {\it  contracting half-Whitehead move} on a signature $\sigma_1$ is a glueing of $m>0$ chords of the same color in one point $\bar{\bar{v}}$ such that the  obtained  diagram is a signature $\sigma$. 

\item A smoothing half-Whitehead move on a signature $\sigma$ is an unsticking of $m> 0$ chords at a point $\bar{\bar{v}}$  such that the  obtained  diagram is a signature $\sigma_{2}$.
\end{enumerate}
\end{definition}

\begin{definition}\label{D:band}
Let $\sigma_1$ and $\sigma_2$ be two different signatures of the same codimension. 
A {\it Whitehead move} $\delta$ is the composition of a contracting and smoothing half-Whitehead moves starting at $\sigma_1$ and ending on $\sigma_2$.Two signatures differing by a Whitehead move are called adjacent signatures and are denoted $\sigma_1\leftrightarrow \sigma_2$.
\end{definition}

Notice that, 
$\leftrightarrow$ is reflexive but not necessarily transitive.

Let us illustrate below a composition of a contracting and a smoothing half-Whitehead move for a pair of chords. 
\begin{figure}[h]
\begin{center}
\begin{tikzpicture}[scale=0.9]
\node (a) at (-2.5,0) {
\begin{tikzpicture}[scale=0.6]
\newcommand{\degree}[0]{4}
\newcommand{\B}[1]{#1*180/\degree}
\newcommand{\R}[1]{#1*180/\degree }
\draw[black,thick, name path=R3R5](\R{3}:1) .. controls(\R{3}:0.3)and (\R{5}:0.3) .. (\R{5}:1) ;
\draw[black, thick,name path=R1R7](\R{1}:1) .. controls(\R{1}:0.3)and (\R{7}:0.3) .. (\R{7}:1) ;
\end{tikzpicture} 
};
\node (b) at (0,0) {
\begin{tikzpicture}[scale=0.6]
\newcommand{\degree}[0]{4}
\newcommand{\B}[1]{#1*180/\degree}
\newcommand{\R}[1]{#1*180/\degree }
\draw[black,  thick, name path=R1R5](\R{1}:1) .. controls(\R{1}:0.3)and (\R{5}:0.3) .. (\R{5}:1) ;
\draw[black,  thick,name path=R3R7](\R{3}:1) .. controls(\R{3}:0.3)and (\R{7}:0.3) .. (\R{7}:1) ;
\end{tikzpicture} 
};
\node (c) at (2.5,0) {
\begin{tikzpicture}[scale=0.6]
\newcommand{\degree}[0]{4}
\newcommand{\B}[1]{#1*180/\degree}
\newcommand{\R}[1]{#1*180/\degree }
\draw[black,  thick, name path=R1R3](\R{1}:1) .. controls(\R{1}:0.3)and (\R{3}:0.3) .. (\R{3}:1) ;
\draw[black, thick,name path=R5R7](\R{5}:1) .. controls(\R{5}:0.3)and (\R{7}:0.3) .. (\R{7}:1) ;
\end{tikzpicture} 
};
\node at (-4,0) {\Large $\delta$ : };
 \draw[black, thick,->] (a) -- (b);
  \draw[black, thick,->]  (b) -- (c);
\end{tikzpicture}
\end{center}

\vspace{-10pt}
\caption{Whitehead move as a contracting half-Whitehead move and a smoothing half-Whitehead move }
\end{figure}
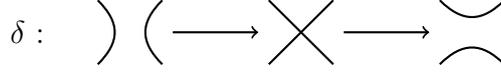
\begin{example} We illustrate a Whitehead move (c.f. definition~\ref{D:band}) on a pair of red chords.
\ 

\begin{itemize}
\item   The two generic signatures, for $d=2$ (on the right and on the left of the figure below) are incident to a codimension 1 signature ( in the middle of the figure)~\ref{F:WHR}.
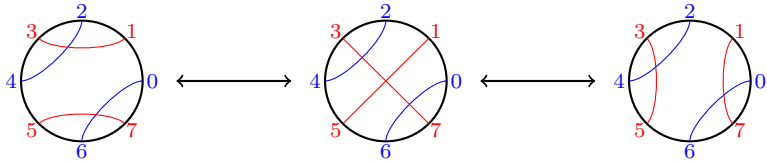
\begin{figure}[h]
\begin{center}
 \begin{tikzpicture}[scale=1]
\node (a) at (-4,0) {
 \begin{tikzpicture}[scale=.8]
\newcommand{\degree}[0]{4}
\newcommand{\B}[1]{#1*180/\degree}
\newcommand{\R}[1]{#1*180/\degree }
\newcommand{\RDOT}[0]{[red](i-1) circle (0.05)}
\newcommand{\BDOT}[0]{[blue](i-1) circle (0.05)}
\draw[black,thick] (0,0) circle (1) ;
\foreach \k in {0,2,4,6} {
  \draw[blue] (\B{\k}:1.16) node {\Small \k};} ;
  \foreach \k in {1,3,5,7} {
 \draw[ red] (\R{\k}:1.16) node {\Small \k} ;
}
\draw[blue, name path=B0B6](\B{0}:1) .. controls(\B{0}:0.7)and (\B{6}:0.7) .. (\B{6}:1) ;
\draw[blue, name path=B2B4](\B{2}:1) .. controls(\B{2}:0.7)and (\B{4}:0.7) .. (\B{4}:1) ;
\draw[red, name path=R1R3](\R{1}:1) .. controls(\R{1}:0.7)and (\R{3}:0.7) .. (\R{3}:1) ;
\draw[red, name path=R5R7](\R{5}:1) .. controls(\R{5}:0.7)and (\R{7}:0.7) .. (\R{7}:1) ;
\end{tikzpicture} 
};
\node (b) at (0,0) {
\begin{tikzpicture}[scale=.8]
\newcommand{\degree}[0]{4}
\newcommand{\B}[1]{#1*180/\degree}
\newcommand{\R}[1]{#1*180/\degree }
\newcommand{\RDOT}[0]{[red](i-1) circle (0.05)}
\newcommand{\BDOT}[0]{[blue](i-1) circle (0.05)}
\draw[black,thick] (0,0) circle (1) ;
\foreach \k in {0,2,4,6} {
  \draw[blue] (\B{\k}:1.16) node {\Small \k };} ;
  \foreach \k in {1,3,5,7} {
 \draw[ red] (\R{\k}:1.16) node {\Small \k} ;
}
\draw[blue, name path=B0B6](\B{0}:1) .. controls(\B{0}:0.7)and (\B{6}:0.7) .. (\B{6}:1) ;
\draw[blue, name path=B2B4](\B{2}:1) .. controls(\B{2}:0.7)and (\B{4}:0.7) .. (\B{4}:1) ;
\draw[red, name path=R1R5](\R{1}:1) .. controls(\R{1}:0.7)and (\R{5}:0.7) .. (\R{5}:1) ;
\draw[red,,name path=R3R7](\R{3}:1) .. controls(\R{3}:0.7)and (\R{7}:0.7) .. (\R{7}:1) ;
\end{tikzpicture} 
};

\node (c) at (4,0) {
\begin{tikzpicture}[scale=.8]
\newcommand{\degree}[0]{4}
\newcommand{\B}[1]{#1*180/\degree}
\newcommand{\R}[1]{#1*180/\degree }
\newcommand{\RDOT}[0]{[red](i-1) circle (0.05)}
\newcommand{\BDOT}[0]{[blue](i-1) circle (0.05)}
\draw[black,thick] (0,0) circle (1) ;
\foreach \k in {0,2,4,6} {
  \draw[blue] (\B{\k}:1.16) node { \Small \k};} ;
  \foreach \k in {1,3,5,7} {
 \draw[ red] (\R{\k}:1.16) node {\Small \k} ;};
\draw[blue, name path=B0B6](\B{0}:1) .. controls(\B{0}:0.7)and (\B{6}:0.7) .. (\B{6}:1) ;
\draw[blue, name path=B2B4](\B{2}:1) .. controls(\B{2}:0.7)and (\B{4}:0.7) .. (\B{4}:1) ;
\draw[red, name path=R3R5](\R{3}:1) .. controls(\R{3}:0.7)and (\R{5}:0.7) .. (\R{5}:1) ;
\draw[red,name path=R1R7](\R{1}:1) .. controls(\R{1}:0.7)and (\R{7}:0.7) .. (\R{7}:1) ;
\end{tikzpicture} 
};
 \draw[<->,thick] (b) -- (a);
  \draw[<->,thick] (b) -- (c);
\end{tikzpicture} 
\end{center}
\vspace{-15pt}
\caption{Whitehead move on a pair of red chords}\label{F:WHR}
\end{figure}

\item  The Whitehead move has incidence on the deformation of coefficients of given polynomial  of degree 3. In this following example, we consider the deformation of blue curves ($Im(P)=0$) of the drawing of the polynomial $P_{1}$:
\[\begin{aligned}
P_{1}&=(z+0.5-0.5\imath)(z+0.5+0.5\imath)(z-0.2-0.6\imath),\\
 P_{2}&=(z+0.5-0.5\imath)(z+0.5+0.5\imath)(z+0.6-0.4\imath),\\
  P_{3}&=(z+0.5-0.5\imath)(z+0.5+0.5\imath)(z+0.5-0.1\imath)
  \end{aligned}\]

\begin{figure}[h]
\begin{center} 
\includegraphics[scale=0.7]{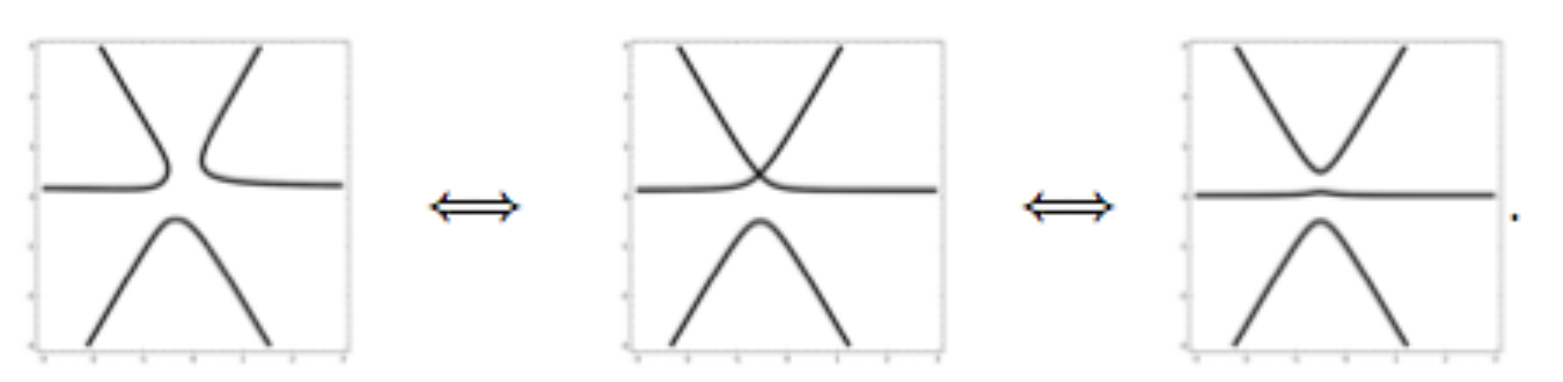}
 \end{center}
\vspace{-10pt}
\caption{Whitehead move of $Im(P)=0$ of monic polynomials}

\end{figure}
\end{itemize}
\end{example}
\section{Classification of signatures and adjacence relations}

\subsection{Matrix notation of signatures and Classification of elementa}
For convenience, we shall represent {\it algebraically} the diagrams and signatures, using some matrix notation. Without loss of generality we focus on generic signatures and matrix notation will be defined only for generic diagrams. Note that in that case, signatures are formed from $d$ trees, having one inner node of valency 4, the four incident edges are colored respectively red, blue, red, blue. The matrix has 2 lines and at most $d$ columns, each column consists of one $(i,j)$, and this indicates the existence of a chord connecting the terminal vertices $i$ and $j$ . Chords having terminal vertices of the same parity lie on the same line in the matrix. 

For instance, the matrix $\left[\begin{smallmatrix}i,j\\k,l\end{smallmatrix}\right]$ indicates that there exists a chord of a given color, connecting the terminal vertices $i$ with $j$ and another chord of the other color, connecting the terminal vertices $k$ with $l$. 

In case we have one short chord in the diagram $(j-1,j+1)$ intersecting a chord $(i,j)$ of the opposite color, we write it as $\left|\begin{smallmatrix}j\\i\end{smallmatrix}\right|$.

\begin{remark}
Each signature is associated to one unique matrix. 
\end{remark}

\subsubsection{Generic elementa}
The generic elementa can be classified by the signatures in the following way~\cite{C0}:
\begin{enumerate}
\item Trees of signatures consisting of two short (red and blue) diagonals are said to be of type $M$. 
An $M$ tree is denoted by $\left|\begin{smallmatrix}i\\i+2\end{smallmatrix}\right|$, where the number on the first line indicates that there exists a short diagonal $(i+1,i+3)$ crossing the diagonal $(i,i+2)$. A signature with only $M$ trees is called an $M$-signature. Note, that  {\it there exist only four} $M$-signatures, for any $d$:
\vspace{4pt}
\begin{enumerate}
\item $M_1\leftrightarrow$ $\left|\begin{smallmatrix}3\\1\end{smallmatrix}\right|\left|\begin{smallmatrix}7\\5\end{smallmatrix}\right|...\left|\begin{smallmatrix}4d-1\\4d-3\end{smallmatrix}\right|$.
\vspace{4pt}
 \item $M_2 \leftrightarrow$ $\left|\begin{smallmatrix}1\\3\end{smallmatrix}\right|\left|\begin{smallmatrix}5\\7\end{smallmatrix}\right|...\left|\begin{smallmatrix}4d-3\\4d-1\end{smallmatrix}\right|$.
\vspace{4pt}
\item $M_3 \leftrightarrow$ $\left|\begin{smallmatrix}3\\5\end{smallmatrix}\right|\left|\begin{smallmatrix}7\\9\end{smallmatrix}\right|...\left|\begin{smallmatrix}1\\4d-1\end{smallmatrix}\right|$.
\vspace{4pt}
 \item $M_4 \leftrightarrow$ $\left|\begin{smallmatrix}5\\3\end{smallmatrix}\right|\left|\begin{smallmatrix}9\\7\end{smallmatrix}\right|...\left|\begin{smallmatrix}4d-1\\1\end{smallmatrix}\right|$.
\end{enumerate} 
\vspace{4pt}
\item Trees consisting of one short and one long diagonal are of type $F$. An $F$ tree is denoted by $\left|\begin{smallmatrix}j\\i\end{smallmatrix}\right|$ where $j$ and $i$ are labels of the terminal vertices of the long diagonal; the number $j$, indicates that there exists a short diagonal of the opposite color joining the vertex $j-1$ to $j+1$. Two $F$ signatures are opposite if they have opposite $F$-trees: i.e. if their indexes are switched. For instance $F^{+}=\left|\begin{smallmatrix}j\\i\end{smallmatrix}\right|$ and $F^{-}=\left|\begin{smallmatrix}i\\j\end{smallmatrix}\right|$ are opposite. A signature with only $M$ and $F$
 trees is called an $F$-signature. A signature of type $F^{\otimes m}$ has exactly $m$ trees of type $F$ (and other of type $M$). 
 \vspace{4pt}
\item Trees consisting of two long diagonals are of type $S$. An $S$ tree is denoted by $\left[\begin{smallmatrix}i,j\\k,l\end{smallmatrix}\right]$, where the first line gives the coordinates of a long diagonal $(i,j)$, the second line gives the coordinates of the second long diagonal $(k,l)$. A signature with only $M$ and $S$ trees is called an $S$-signature. 
As for the previous family of signatures, a signature is of type $S^{\otimes m}$ if there exists $m$ trees of type $S$, the other are $M$ trees. We focus on $S$ trees given by pairs of diagonals $(i,j)$ and $(i+1,j+1)$ or $(i,j)$ and $(i-1,j-1)$ and call the {\it narrow} $S$ trees.  If we have a signature $S^{\otimes d-2}$ then all the $S$ trees are narrow. A pair of $S$-trees are opposite if the matrices are of type $S^{+}=\left[\begin{smallmatrix}i,j\\i+1,j-1\end{smallmatrix}\right]$ and $S^{-}=\left[\begin{smallmatrix}i,j\\i-1,j+1\end{smallmatrix}\right]$.
\vspace{4pt}\item The combination of $F$, $S$ and $M$ trees gives an $FS$-signature.
\end{enumerate}

The figure~\ref{F:MFS} presents  examples of signatures of type $M, F, S$ and $FS$, for $d=4$. Each signature is indexed by its algebraic notation and by its contracted notation. In the contracted notation the parenthesis $\left(\begin{smallmatrix}j\\j\pm1\end{smallmatrix}\right)$  replaces the sequence of repetitive motifs and the orientation.
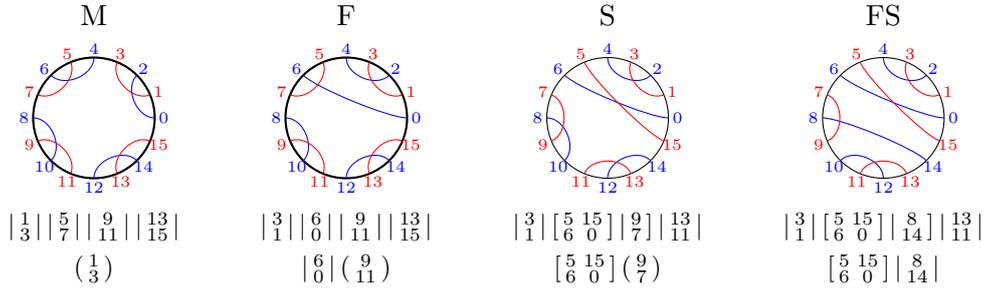
\begin{figure}[h]
\begin{center}
\begin{tikzpicture}[scale=0.8]
\newcommand{\degree}[0]{8}
\newcommand{\B}[1]{#1*180/\degree}
\newcommand{\R}[1]{#1*180/\degree }
\draw[black,thick] (0,0) circle (1) ;
\foreach \k in {0,2,4,6,8,10,12,14} {
  \draw[blue] (\B{\k}:1.15) node {\tiny\k};} ;
  \foreach \k in {1,3,5,7,9,11,13,15} {
 \draw[ red] (\R{\k}:1.15) node {\tiny\k} ;};
\draw[blue, name path=B0B2](\B{0}:1) .. controls(\B{0}:0.7)and (\B{2}:0.7) .. (\B{2}:1) ;
\draw[blue, name path=B4B6](\B{4}:1) .. controls(\B{4}:0.7)and (\B{6}:0.7) .. (\B{6}:1) ;
\draw[blue, name path=B8B10](\B{8}:1) .. controls(\B{8}:0.7)and (\B{10}:0.7) .. (\B{10}:1) ;
\draw[blue, name path=B12B14](\B{12}:1) .. controls(\B{12}:0.7)and (\B{14}:0.7) .. (\B{14}:1) ;
\draw[red,  name path=R1R3](\R{1}:1) .. controls(\R{1}:0.7)and (\R{3}:0.7) .. (\R{3}:1) ;
\draw[red, name path=R5R7](\R{5}:1) .. controls(\R{5}:0.7)and (\R{7}:0.7) .. (\R{7}:1) ;
\draw[red, name path=R9R11](\R{9}:1) .. controls(\R{9}:0.7)and (\R{11}:0.7) .. (\R{11}:1) ;
\draw[red, name path=R13R15](\R{13}:1) .. controls(\R{13}:0.7)and (\R{15}:0.7) .. (\R{15}:1) ;
\node(M) at (0,1.7) {\large M};
\node(m1) at (0,-1.8) {$\left|\begin{smallmatrix}1\\3\end{smallmatrix}\right|\left|\begin{smallmatrix}5\\7\end{smallmatrix}\right| \left|\begin{smallmatrix}9\\11\end{smallmatrix}\right|\left|\begin{smallmatrix}13\\15\end{smallmatrix}\right|$};
\node(m2) at (0,-2.5) {$\left(\begin{smallmatrix}1\\3\end{smallmatrix}\right)$};
\end{tikzpicture}
\qquad
\begin{tikzpicture}[scale=0.8]
\newcommand{\degree}[0]{8}
\newcommand{\B}[1]{#1*180/\degree}
\newcommand{\R}[1]{#1*180/\degree }
\draw[black,thick] (0,0) circle (1) ;
\foreach \k in {0,2,4,6,8,10,12,14} {
  \draw[blue] (\B{\k}:1.15) node { \tiny\k};} ;
  \foreach \k in {1,3,5,7,9,11,13,15} {
 \draw[ red] (\R{\k}:1.15) node {\tiny\k} ;};
\draw[blue,name path=B0B6](\B{0}:1) .. controls(\B{0}:0.7)and (\B{6}:0.7) .. (\B{6}:1) ;
\draw[blue, name path=B2B4](\B{2}:1) .. controls(\B{2}:0.7)and (\B{4}:0.7) .. (\B{4}:1) ;
\draw[blue, name path=B8B10](\B{8}:1) .. controls(\B{8}:0.7)and (\B{10}:0.7) .. (\B{10}:1) ;
\draw[blue, name path=B12B14](\B{12}:1) .. controls(\B{12}:0.7)and (\B{14}:0.7) .. (\B{14}:1) ;
\draw[red,  name path=R1R3](\R{1}:1) .. controls(\R{1}:0.7)and (\R{3}:0.7) .. (\R{3}:1) ;
\draw[red, name path=R5R7](\R{5}:1) .. controls(\R{5}:0.7)and (\R{7}:0.7) .. (\R{7}:1) ;
\draw[red, name path=R9R11](\R{9}:1) .. controls(\R{9}:0.7)and (\R{11}:0.7) .. (\R{11}:1) ;
\draw[red,name path=R13R15](\R{13}:1) .. controls(\R{13}:0.7)and (\R{15}:0.7) .. (\R{15}:1) ;
\node(F) at (0,1.7) {\large F};
\node(f1) at (0,-1.8) {$\left|\begin{smallmatrix}3\\1\end{smallmatrix}\right|\left|\begin{smallmatrix}6\\0\end{smallmatrix}\right| \left|\begin{smallmatrix}9\\11\end{smallmatrix}\right|\left|\begin{smallmatrix}13\\15\end{smallmatrix}\right|$};
\node(f2) at (0,-2.5) {$\left|\begin{smallmatrix}6\\0\end{smallmatrix}\right| \left(\begin{smallmatrix}9\\11\end{smallmatrix}\right)$};
\end{tikzpicture}
\qquad
\begin{tikzpicture}[scale=0.8]
\newcommand{\degree}[0]{8}
\newcommand{\B}[1]{#1*180/\degree}
\newcommand{\R}[1]{#1*180/\degree }
\draw[black] (0,0) circle (1) ;
\foreach \k in {0,2,4,6,8,10,12,14} {
  \draw[blue] (\B{\k}:1.15) node { \tiny\k};} ;
  \foreach \k in {1,3,5,7,9,11,13,15} {
 \draw[ red] (\R{\k}:1.15) node {\tiny\k} ;
};
\draw[blue,name path=B4B2](\B{4}:1) .. controls(\B{4}:0.7)and (\B{2}:0.7) .. (\B{2}:1) ;
\draw[blue, name path=B16B6](\B{16}:1) .. controls(\B{16}:0.7)and (\B{6}:0.7) .. (\B{6}:1) ;
\draw[blue, name path=B8B10](\B{8}:1) .. controls(\B{8}:0.7)and (\B{10}:0.7) .. (\B{10}:1) ;
\draw[blue, name path=B12B14](\B{12}:1) .. controls(\B{12}:0.7)and (\B{14}:0.7) .. (\B{14}:1) ;
\draw[red, name path=R1R3](\R{1}:1) .. controls(\R{1}:0.7)and (\R{3}:0.7) .. (\R{3}:1) ;
\draw[red, name path=R5R15](\R{5}:1) .. controls(\R{5}:0.7)and (\R{15}:0.7) .. (\R{15}:1) ;
\draw[red,name path=R11R13](\R{11}:1) .. controls(\R{11}:0.7)and (\R{13}:0.7) .. (\R{13}:1) ;
\draw[red,name path=R9R7](\R{9}:1) .. controls(\R{9}:0.7)and (\R{7}:0.7) .. (\R{7}:1) ;
\node(S) at (0,1.7) {\large S};
\node(s1) at (0,-1.8) {$\left|\begin{smallmatrix}3\\1\end{smallmatrix}\right|\left[\begin{smallmatrix}5&15\\6&0\end{smallmatrix}\right] \left|\begin{smallmatrix}9\\7\end{smallmatrix}\right]\left|\begin{smallmatrix}13\\11\end{smallmatrix}\right|$};
\node(s2) at (0,-2.5) {$\left[\begin{smallmatrix}5&15\\6&0\end{smallmatrix}\right] \left(\begin{smallmatrix}9\\7\end{smallmatrix}\right)$};
\end{tikzpicture}
\qquad
\begin{tikzpicture}[scale=0.8]
\newcommand{\degree}[0]{8}
\newcommand{\B}[1]{#1*180/\degree}
\newcommand{\R}[1]{#1*180/\degree }
\draw[black] (0,0) circle (1) ;
\foreach \k in {0,2,4,6,8,10,12,14} {
  \draw[blue] (\B{\k}:1.15) node { \tiny\k};} ;
  \foreach \k in {1,3,5,7,9,11,13,15} {
 \draw[ red] (\R{\k}:1.15) node {\tiny\k} ;
};
\draw[blue,name path=B4B2](\B{4}:1) .. controls(\B{4}:0.7)and (\B{2}:0.7) .. (\B{2}:1) ;
\draw[blue, name path=B16B6](\B{16}:1) .. controls(\B{16}:0.7)and (\B{6}:0.7) .. (\B{6}:1) ;
\draw[blue, name path=B8B14](\B{8}:1) .. controls(\B{8}:0.7)and (\B{14}:0.7) .. (\B{14}:1) ;
\draw[blue, name path=B10B12](\B{10}:1) .. controls(\B{10}:0.7)and (\B{12}:0.7) .. (\B{12}:1) ;
\draw[red, name path=R1R3](\R{1}:1) .. controls(\R{1}:0.7)and (\R{3}:0.7) .. (\R{3}:1) ;
\draw[red, name path=R5R15](\R{5}:1) .. controls(\R{5}:0.7)and (\R{15}:0.7) .. (\R{15}:1) ;
\draw[red,name path=R11R13](\R{11}:1) .. controls(\R{11}:0.7)and (\R{13}:0.7) .. (\R{13}:1) ;
\draw[red,name path=R9R7](\R{9}:1) .. controls(\R{9}:0.7)and (\R{7}:0.7) .. (\R{7}:1) ;
\node(S) at (0,1.7) {\large FS};
\node(s1) at (0,-1.8) {$\left|\begin{smallmatrix}3\\1\end{smallmatrix}\right|\left[\begin{smallmatrix}5&15\\6&0\end{smallmatrix}\right] \left|\begin{smallmatrix}8\\14\end{smallmatrix}\right]\left|\begin{smallmatrix}13\\11\end{smallmatrix}\right|$};
\node(s2) at (0,-2.5) {$\left[\begin{smallmatrix}5&15\\6&0\end{smallmatrix}\right] \left|\begin{smallmatrix}8\\14\end{smallmatrix}\right|$};
\end{tikzpicture}
\end{center}
\vspace{-10pt}
\caption{Examples of $M, F,S$ and $FS$ trees and their matrices for $d=4$}~\label{F:MFS}
\end{figure}

\subsubsection{Classification of the signatures for $d=3$} In this subsection we consider for the case of $d=3$ not only the generic signatures but also the signatures of non zero codimension. For each codimension, we give the number of generic signatures and classified by characteristic paterns.

\vspace{5pt}
\noindent{\bf Signatures of codimension 0:} 

In the decomposition by signatures of $\dpol_3$, there exist 22 generic signatures among which there are four $M$ signatures, six $S$ signatures (containing only one $S$ tree) and twelve $F$ signatures (six with one red long diagonal and six with one long blue diagonal): 

\vspace{10pt}
\hspace{2cm}{
\begin{tikzpicture}[scale=0.8]
\newcommand{\degree}[0]{6}
\newcommand{\last}[0]{5}
\newcommand{\B}[1]{#1*180/\degree}
\newcommand{\R}[1]{#1*180/\degree }
\newcommand{\RDOT}[0]{[  red](i-1) circle (0.05)}
\newcommand{\BDOT}[0]{[ blue](i-1) circle (0.05)}
\newcommand{\XDOT}[0]{[black](i-1) circle (0.045)}
\draw (-6,-0.3) node {\bf $\bullet$ 4  signatures of type M:};
\draw (0,0) circle (1) ;
\draw[blue,  name path=B4B6](\B{4}:1) .. controls(\B{4}:0.7)and (\B{6}:0.7) .. (\B{6}:1) ;
\draw[blue,  name path=B8B10](\B{8}:1) .. controls(\B{8}:0.7)and (\B{10}:0.7) .. (\B{10}:1) ;
\draw[blue,  name path=B0B2](\B{0}:1) .. controls(\B{0}:0.7)and (\B{2}:0.7) .. (\B{2}:1) ;
\draw[red, name path=R1R3](\R{1}:1) .. controls(\R{1}:0.7)and (\R{3}:0.7) .. (\R{3}:1) ;
\draw[red, name path=R5R7](\R{5}:1) .. controls(\R{5}:0.7)and (\R{7}:0.7) .. (\R{7}:1) ;
\draw[red, name path=R9R11](\R{9}:1) .. controls(\R{9}:0.7)and (\R{11}:0.7) .. (\R{11}:1) ;
\fill[name intersections={of=B4B6 and R5R7,name=i}] \XDOT ;
\fill[name intersections={of=B8B10 and R9R11,name=i}] \XDOT ;
\fill[name intersections={of=B0B2 and R1R3,name=i}] \XDOT ;
\end{tikzpicture}
}

\vspace{10pt}
\hspace{2cm}{
\begin{tikzpicture}[scale=0.8]
\newcommand{\degree}[0]{6}
\newcommand{\last}[0]{5}
\newcommand{\B}[1]{#1*180/\degree}
\newcommand{\R}[1]{#1*180/\degree }
\newcommand{\RDOT}[0]{[  red](i-1) circle (0.05)}
\newcommand{\BDOT}[0]{[ blue](i-1) circle (0.05)}
\newcommand{\XDOT}[0]{[black](i-1) circle (0.045)}
\draw (-6,-0.3) node {\bf $\bullet$ 6  signatures of type S:};
\draw (0,0) circle (1) ;
\draw[blue,  name path=B6B8](\B{6}:1) .. controls(\B{6}:0.7)and (\B{8}:0.7) .. (\B{8}:1) ;
\draw[blue,  name path=B4B10](\B{4}:1) .. controls(\B{4}:0.7)and (\B{10}:0.7) .. (\B{10}:1) ;
\draw[blue,  name path=B0B2](\B{0}:1) .. controls(\B{0}:0.7)and (\B{2}:0.7) .. (\B{2}:1) ;
\draw[red, name path=R7R9](\R{7}:1) .. controls(\R{7}:0.7)and (\R{9}:0.7) .. (\R{9}:1) ;
\draw[red, name path=R5R11](\R{5}:1) .. controls(\R{5}:0.7)and (\R{11}:0.7) .. (\R{11}:1) ;
\draw[red, name path=R1R3](\R{1}:1) .. controls(\R{1}:0.7)and (\R{3}:0.7) .. (\R{3}:1) ;
\fill[name intersections={of=B6B8 and R7R9,name=i}] \XDOT ;
\fill[name intersections={of=B4B10 and R5R11,name=i}] \XDOT ;
\fill[name intersections={of=B0B2 and R1R3,name=i}] \XDOT ;
\end{tikzpicture}
}

\vspace{10pt}
\hspace{2cm}{
\begin{tikzpicture}[scale=0.8]
\newcommand{\degree}[0]{6}
\newcommand{\last}[0]{5}
\newcommand{\B}[1]{#1*180/\degree}
\newcommand{\R}[1]{#1*180/\degree }
\newcommand{\RDOT}[0]{[  red](i-1) circle (0.05)}
\newcommand{\BDOT}[0]{[ blue](i-1) circle (0.05)}
\newcommand{\XDOT}[0]{[black](i-1) circle (0.045)}
\draw (-6,-0.3) node {\bf $\bullet$ 12  signatures of type F:};
\draw (0,0) circle (1) ;
\draw[blue,  name path=B8B10](\B{8}:1) .. controls(\B{8}:0.7)and (\B{10}:0.7) .. (\B{10}:1) ;
\draw[blue,  name path=B4B6](\B{4}:1) .. controls(\B{4}:0.7)and (\B{6}:0.7) .. (\B{6}:1) ;
\draw[blue,  name path=B0B2](\B{0}:1) .. controls(\B{0}:0.7)and (\B{2}:0.7) .. (\B{2}:1) ;
\draw[red, name path=R7R9](\R{7}:1) .. controls(\R{7}:0.7)and (\R{9}:0.7) .. (\R{9}:1) ;
\draw[red, name path=R5R11](\R{5}:1) .. controls(\R{11}:0.7)and (\R{5}:0.7) .. (\R{11}:1) ;
\draw[red, name path=R1R3](\R{1}:1) .. controls(\R{1}:0.7)and (\R{3}:0.7) .. (\R{3}:1) ;\
\fill[name intersections={of=B8B10 and R7R9,name=i}] \XDOT ;
\fill[name intersections={of=B4B6 and R5R11,name=i}] \XDOT ;
\fill[name intersections={of=B0B2 and R1R3,name=i}] \XDOT ;
\end{tikzpicture}
}

\vspace{10pt}
In addition to generic classes there exist 
\begin{itemize}
\item 48 signatures of codimension 1, four families of twelve signatures which are equivalent up to rotation;
\item 30 signatures of codimension 2;
\item 4 signatures of codimension 3.
\end{itemize} 

Let us notice that the alternating sum of the number $N_{k}$ of  k-codimemsional signatures  $ \sum_{k=0}^{3} (-1)^{k}N_{k}=0$.

\vspace{10pt}
\noindent{\bf Signatures of codimension 1:} 

\begin{center}
\begin{tikzpicture}[scale=0.8]
\newcommand{\degree}[0]{6}
\newcommand{\last}[0]{5}
\newcommand{\B}[1]{#1*180/\degree}
\newcommand{\R}[1]{#1*180/\degree }
\newcommand{\RDOT}[0]{[  red](i-1) circle (0.05)}
\newcommand{\BDOT}[0]{[ blue](i-1) circle (0.05)}
\newcommand{\XDOT}[0]{[black](i-1) circle (0.045)}
\draw (0,0) circle (1) ;
\draw[blue,  name path=B8B10](\B{8}:1) .. controls(\B{8}:0.7)and (\B{10}:0.7) .. (\B{10}:1) ;
\draw[blue,  name path=B4B6](\B{4}:1) .. controls(\B{4}:0.7)and (\B{6}:0.7) .. (\B{6}:1) ;
\draw[blue,  name path=B0B2](\B{0}:1) .. controls(\B{0}:0.7)and (\B{2}:0.7) .. (\B{2}:1) ;
\draw[red, name path=R7R9](\R{7}:1) .. controls(\R{7}:0.7)and (\R{9}:0.7) .. (\R{9}:1) ;
\draw[red, name path=R3R11](\R{3}:1) .. controls(\R{3}:0.7)and (\R{11}:0.7) .. (\R{11}:1) ;
\draw[red, name path=R1R5](\R{1}:1) .. controls(\R{1}:0.7)and (\R{5}:0.7) .. (\R{5}:1) ;
\draw[name intersections={of=R3R11 and R1R5,name=i}] \RDOT ;
\draw[name intersections={of=R1R5 and R3R11,name=i}] \RDOT ;
\fill[name intersections={of=B8B10 and R7R9,name=i}] \XDOT ;
\fill[name intersections={of=B4B6 and R1R5,name=i}] \XDOT ;
\fill[name intersections={of=B0B2 and R1R5,name=i}] \XDOT ;

\end{tikzpicture}
\qquad
\begin{tikzpicture}[scale=0.8]
\newcommand{\degree}[0]{6}
\newcommand{\last}[0]{5}
\newcommand{\B}[1]{#1*180/\degree}
\newcommand{\R}[1]{#1*180/\degree }
\newcommand{\RDOT}[0]{[  red](i-1) circle (0.05)}
\newcommand{\BDOT}[0]{[ blue](i-1) circle (0.05)}
\newcommand{\XDOT}[0]{[black](i-1) circle (0.045)}
\draw (0,0) circle (1) ;
\draw[blue,  name path=B6B10](\B{6}:1) .. controls(\B{6}:0.7)and (\B{10}:0.7) .. (\B{10}:1) ;
\draw[blue,  name path=B4B8](\B{4}:1) .. controls(\B{4}:0.7)and (\B{8}:0.7) .. (\B{8}:1) ;
\draw[blue,  name path=B0B2](\B{0}:1) .. controls(\B{0}:0.7)and (\B{2}:0.7) .. (\B{2}:1) ;
\draw[red, name path=R7R9](\R{7}:1) .. controls(\R{7}:0.7)and (\R{9}:0.7) .. (\R{9}:1) ;
\draw[red, name path=R5R11](\R{5}:1) .. controls(\R{5}:0.7)and (\R{11}:0.7) .. (\R{11}:1) ;
\draw[red, name path=R1R3](\R{1}:1) .. controls(\R{1}:0.7)and (\R{3}:0.7) .. (\R{3}:1) ;
\draw[name intersections={of=B6B10 and B4B8,name=i}] \BDOT ;
\draw[name intersections={of=B4B8 and B6B10,name=i}] \BDOT ;
\fill[name intersections={of=B4B8 and R7R9,name=i}] \XDOT ;
\fill[name intersections={of=B4B8 and R5R11,name=i}] \XDOT ;
\fill[name intersections={of=B0B2 and R1R3,name=i}] \XDOT ;
\end{tikzpicture}
\qquad
\begin{tikzpicture}[scale=0.8]
\newcommand{\degree}[0]{6}
\newcommand{\last}[0]{5}
\newcommand{\B}[1]{#1*180/\degree}
\newcommand{\R}[1]{#1*180/\degree }
\newcommand{\RDOT}[0]{[  red](i-1) circle (0.05)}
\newcommand{\BDOT}[0]{[ blue](i-1) circle (0.05)}
\newcommand{\XDOT}[0]{[black](i-1) circle (0.045)}
\draw (0,0) circle (1) ;
\draw[blue, name path=B6B10](\B{6}:1) .. controls(\B{6}:0.7)and (\B{10}:0.7) .. (\B{10}:1) ;
\draw[blue, name path=B4B8](\B{4}:1) .. controls(\B{4}:0.7)and (\B{8}:0.7) .. (\B{8}:1) ;
\draw[blue, name path=B0B2](\B{0}:1) .. controls(\B{0}:0.7)and (\B{2}:0.7) .. (\B{2}:1) ;
\draw[red, name path=R5R7](\R{5}:1) .. controls(\R{5}:0.7)and (\R{7}:0.7) .. (\R{7}:1) ;
\draw[red, name path=R3R9](\R{3}:1) .. controls(\R{3}:0.7)and (\R{9}:0.7) .. (\R{9}:1) ;
\draw[red, name path=R1R11](\R{1}:1) .. controls(\R{1}:0.7)and (\R{11}:0.7) .. (\R{11}:1) ;
\draw[name intersections={of=B6B10 and B4B8,name=i}] \BDOT ;
\draw[name intersections={of=B4B8 and B6B10,name=i}] \BDOT ;
\fill[name intersections={of=B6B10 and R5R7,name=i}] \XDOT ;
\fill[name intersections={of=B6B10 and R3R9,name=i}] \XDOT ;
\fill[name intersections={of=B0B2 and R1R11,name=i}] \XDOT ;
\end{tikzpicture}
\qquad
\begin{tikzpicture}[scale=0.8]
\newcommand{\degree}[0]{6}
\newcommand{\last}[0]{5}
\newcommand{\B}[1]{#1*180/\degree}
\newcommand{\R}[1]{#1*180/\degree }
\newcommand{\RDOT}[0]{[  red](i-1) circle (0.05)}
\newcommand{\BDOT}[0]{[ blue](i-1) circle (0.05)}
\newcommand{\XDOT}[0]{[black](i-1) circle (0.045)}
\draw (0,0) circle (1) ;

\draw[blue,  name path=B8B10](\B{8}:1) .. controls(\B{8}:0.7)and (\B{10}:0.7) .. (\B{10}:1) ;
\draw[blue,  name path=B4B6](\B{4}:1) .. controls(\B{4}:0.7)and (\B{6}:0.7) .. (\B{6}:1) ;
\draw[blue, name path=B0B2](\B{0}:1) .. controls(\B{0}:0.7)and (\B{2}:0.7) .. (\B{2}:1) ;
\draw[red, name path=R7R11](\R{7}:1) .. controls(\R{7}:0.7)and (\R{11}:0.7) .. (\R{11}:1) ;
\draw[red, name path=R5R9](\R{5}:1) .. controls(\R{5}:0.7)and (\R{9}:0.7) .. (\R{9}:1) ;
\draw[red, name path=R1R3](\R{1}:1) .. controls(\R{1}:0.7)and (\R{3}:0.7) .. (\R{3}:1) ;
\draw[name intersections={of=R5R9 and R7R11,name=i}] \RDOT ;
\fill[name intersections={of=B8B10 and R5R9,name=i}] \XDOT ;
\fill[name intersections={of=B4B6 and R5R9,name=i}] \XDOT ;
\fill[name intersections={of=B0B2 and R1R3,name=i}] \XDOT ;
\end{tikzpicture}
\end{center}

\vspace{10pt}
\noindent {\bf Diagrams of codimension 2:} 

\hspace{2cm}{
\begin{tikzpicture}[scale=0.8]
\newcommand{\degree}[0]{6}
\newcommand{\last}[0]{5}
\newcommand{\B}[1]{#1*180/\degree}
\newcommand{\R}[1]{#1*180/\degree }
\newcommand{\RDOT}[0]{[  red](i-1) circle (0.05)}
\newcommand{\BDOT}[0]{[ blue](i-1) circle (0.05)}
\newcommand{\XDOT}[0]{[black](i-1) circle (0.045)}
\draw (-6,-0.3) node {\bf $\bullet$ 1 family of order 6};
\draw (0,0) circle (1) ;
\draw[blue, name path=B6B8](\B{6}:1) .. controls(\B{6}:0.7)and (\B{8}:0.7) .. (\B{8}:1) ;
\draw[blue, name path=B4B10](\B{4}:1) .. controls(\B{4}:0.7)and (\B{10}:0.7) .. (\B{10}:1) ;
\draw[blue, name path=B0B2](\B{0}:1) .. controls(\B{0}:0.7)and (\B{2}:0.7) .. (\B{2}:1) ;
\draw[red, name path=R5R9](\R{5}:1) .. controls(\R{5}:0.7)and (\R{9}:0.7) .. (\R{9}:1) ;
\draw[red, name path=R3R11](\R{3}:1) .. controls(\R{3}:0.7)and (\R{11}:0.7) .. (\R{11}:1);
\draw[red, name path=R1R7](\R{1}:1) .. controls(\R{1}:0.7)and (\R{7}:0.7) .. (\R{7}:1) ;
\draw[name intersections={of=R5R9 and R1R7,name=i}] \RDOT ;
\draw[name intersections={of=R3R11 and R1R7,name=i}] \RDOT ;
\draw[name intersections={of=R1R7 and R5R9,name=i}] \RDOT ;
\fill[name intersections={of=B6B8 and R1R7,name=i}] \XDOT ;
\fill[name intersections={of=B4B10 and R1R7,name=i}] \XDOT ;
\fill[name intersections={of=B0B2 and R1R7,name=i}] \XDOT ;
\end{tikzpicture}
}

\hspace{1cm}

\hspace{2cm}{
\begin{tikzpicture}[scale=0.8]
\newcommand{\degree}[0]{6}
\newcommand{\last}[0]{5}
\newcommand{\B}[1]{#1*180/\degree}
\newcommand{\R}[1]{#1*180/\degree }
\newcommand{\RDOT}[0]{[  red](i-1) circle (0.05)}
\newcommand{\BDOT}[0]{[ blue](i-1) circle (0.05)}
\newcommand{\XDOT}[0]{[black](i-1) circle (0.045)}
\draw (-6,-0.3) node {\bf $\bullet$ 2 families of order 12};
\draw (0,0) circle (1) ;
\draw[blue, name path=B6B10](\B{6}:1) .. controls(\B{6}:0.7)and (\B{10}:0.7) .. (\B{10}:1) ;
 \draw[blue, name path=B4B8](\B{4}:1) .. controls(\B{4}:0.7)and (\B{8}:0.7) .. (\B{8}:1) ;
 \draw[blue,  name path=B0B2](\B{0}:1) .. controls(\B{0}:0.7)and (\B{2}:0.7) .. (\B{2}:1) ;
 \draw[red, name path=R7R9](\R{7}:1) .. controls(\R{7}:0.7)and (\R{9}:0.7) .. (\R{9}:1) ;
 \draw[red, name path=R3R11](\R{3}:1) .. controls(\R{3}:0.7)and (\R{11}:0.7) .. (\R{11}:1) ;
 \draw[red, name path=R1R5](\R{1}:1) .. controls(\R{1}:0.7)and (\R{5}:0.7) .. (\R{5}:1) ;
 \draw[name intersections={of=B6B10 and B4B8,name=i}] \BDOT ;
\draw[name intersections={of=B4B8 and B6B10,name=i}] \BDOT ;
\draw[name intersections={of=R3R11 and R1R5,name=i}] \RDOT ;
\draw[name intersections={of=R1R5 and R3R11,name=i}] \RDOT ;
\fill[name intersections={of=B4B8 and R7R9,name=i}] \XDOT ;
\fill[name intersections={of=B4B8 and R1R5,name=i}] \XDOT ;
\fill[name intersections={of=B0B2 and R1R5,name=i}] \XDOT ;
\end{tikzpicture}
\qquad
\begin{tikzpicture}[scale=0.8]
\newcommand{\degree}[0]{6}
\newcommand{\last}[0]{5}
\newcommand{\B}[1]{#1*180/\degree}
\newcommand{\R}[1]{#1*180/\degree }
\newcommand{\RDOT}[0]{[  red](i-1) circle (0.05)}
\newcommand{\BDOT}[0]{[ blue](i-1) circle (0.05)}
\newcommand{\XDOT}[0]{[black](i-1) circle (0.045)}
\draw (0,0) circle (1) ;
\draw[blue,  name path=B8B10](\B{8}:1) .. controls(\B{8}:0.7)and (\B{10}:0.7) .. (\B{10}:1) ;
\draw[blue,  name path=B4B6](\B{4}:1) .. controls(\B{4}:0.7)and (\B{6}:0.7) .. (\B{6}:1) ;
\draw[blue,name path=B0B2](\B{0}:1) .. controls(\B{0}:0.7)and (\B{2}:0.7) .. (\B{2}:1) ;
\draw[red, name path=R7R11](\R{7}:1) .. controls(\R{7}:0.7)and (\R{11}:0.7) .. (\R{11}:1) ;
\draw[red, name path=R3R9](\R{3}:1) .. controls(\R{3}:0.7)and (\R{9}:0.7) .. (\R{9}:1) ;
\draw[red, name path=R1R5](\R{1}:1) .. controls(\R{1}:0.7)and (\R{5}:0.7) .. (\R{5}:1) ;
\draw[name intersections={of=R7R11 and R3R9,name=i}] \RDOT ;
\draw[name intersections={of=R3R9 and R7R11,name=i}] \RDOT ;
\draw[name intersections={of=R3R9 and R1R5,name=i}] \RDOT ;
\draw[name intersections={of=R1R5 and R3R9,name=i}] \RDOT ;
\fill[name intersections={of=B8B10 and R3R9,name=i}] \XDOT ;
\fill[name intersections={of=B4B6 and R1R5,name=i}] \XDOT ;
\fill[name intersections={of=B0B2 and R1R5,name=i}] \XDOT ;
\end{tikzpicture}
}

\vspace{10pt}
\noindent {\bf Signatures of codimension 3:} 

\hspace{2cm}{
\begin{tikzpicture}[scale=0.8]
\newcommand{\degree}[0]{6}
\newcommand{\last}[0]{5}
\newcommand{\B}[1]{#1*180/\degree}
\newcommand{\R}[1]{#1*180/\degree }
\newcommand{\RDOT}[0]{[  red](i-1) circle (0.05)}
\newcommand{\BDOT}[0]{[ blue](i-1) circle (0.05)}
\newcommand{\XDOT}[0]{[black](i-1) circle (0.045)}
\draw (-6,-0.3) node {\bf $\bullet$ 1 familiy of order 4};
\draw (0,0) circle (1) ;
\draw[blue,  name path=B8B10](\B{8}:1) .. controls(\B{8}:0.7)and (\B{10}:0.7) .. (\B{10}:1) ;
\draw[blue,  name path=B4B6](\B{4}:1) .. controls(\B{4}:0.7)and (\B{6}:0.7) .. (\B{6}:1) ;
\draw[blue,  name path=B0B2](\B{0}:1) .. controls(\B{0}:0.7)and (\B{2}:0.7) .. (\B{2}:1) ;
\draw[red, name path=R5R11](\R{5}:1) .. controls(\R{5}:0.7)and (\R{11}:0.7) .. (\R{11}:1) ;
\draw[red, name path=R3R9](\R{3}:1) .. controls(\R{3}:0.7)and (\R{9}:0.7) .. (\R{9}:1) ;
\draw[red, name path=R1R7](\R{1}:1) .. controls(\R{1}:0.7)and (\R{7}:0.7) .. (\R{7}:1) ;
\draw[name intersections={of=R5R11 and R3R9,name=i}] \RDOT ;
\draw[name intersections={of=R5R11 and R1R7,name=i}] \RDOT ;
\draw[name intersections={of=R3R9 and R5R11,name=i}] \RDOT ;
\fill[name intersections={of=B8B10 and R3R9,name=i}] \XDOT ;
\fill[name intersections={of=B4B6 and R5R11,name=i}] \XDOT ;
\fill[name intersections={of=B0B2 and R1R7,name=i}] \XDOT ;
\end{tikzpicture}
}

\begin{lemma}\label{L:Diago1}
Let $\sigma$ be a generic signature, having all red (resp. blue) diagonals short.
Consider a pair of adjoining $F$ trees, of long blue (resp. red) diagonals $(i,j)(j+2,k)$ and deform them by a Whitehead move.  If $k= i-2 \mod 4d$, then the number of blue (resp. red) short diagonals is increased by two.
Otherwise, the number is increased by one.
\end{lemma}
\begin{proof}
Consider the first case. Applying a Whitehead move onto this pair of diagonals induces a new signature, where the new pair of diagonals is $(i,i-2)(j, j+2)$, which are both short in the sense of the definition~\ref{D:diagonalsij}. So, the number of short blue (resp. red) diagonals is increased by two.
Concerning the second case, the Whitehead move applied to the pair of diagonals $(i,j)(i-2,k)$ induces $(i,i-2)(j,k)$, where $(i,i-2)$ is a short diagonal. So, the number of short blue (resp. red) diagonals is increased by one.
\end{proof}

\begin{corollary}
Consider a generic signature $\sigma$ having only short red (resp. blue) diagonals. The  repetitive application of Whitehead moves onto pairs of blue (resp. red) diagonals induces an $M$-signature, in a finite number of Whitehead moves.
\end{corollary}

\begin{lemma}\label{L:Diago2}
Let $\sigma$ be a generic signature having all red (resp. blue) short diagonals. Consider a pair of non-successive, adjoining blue diagonals in $\sigma$. 
Then, applying a Whitehead move onto this pair of diagonals we have one of the following situation:
\begin{enumerate}

\item  the number of long blue (resp. red) diagonals increases by two, if both diagonals are short,

\item the number of long blue (resp. red) diagonals increases by one, if one of the diagonals is short,  
  
\item  if both diagonals are long  a new pair of long diagonals appears.
\end{enumerate}
\end{lemma}
\begin{proof}

Let Let $(i,j)$ and $(l,k)$ be the pair of blue (resp. red) diagonals.
 \begin{itemize}
\item If both diagonals are short then $|i-j|=|l-k|=2$, where $k\neq j+2 \mod 4d$, and $l\neq i-2\mod 4d$ (using definition~\ref{D:diagonalsij} and definition~\ref{D:suc}). So, applying the Whitehead move onto the pair $(i,j)$ and $(k,l)$ induces the new pair of diagonals $(i,k)(j,l)$, where $|i-l|\neq 2$ and $|l-j|\neq 2$. 

\item Let us suppose, without loss of generality, that $(k,l)$ is short. Applying the Whitehead move onto $(i,j)$ and $(k,k+2)$ gives  the pair of diagonals $(i,k+2)(j,k)$. Since $(i,j)$ and $(k,l)$ are non-successive, then $l\neq i-2$ and $k\neq j+2$. Therefore, $|i-k-2|\neq 2$ and $|j-k|\neq 2$. Since both diagonals are long, the number of long blue (resp. red) diagonals, is increased by one. 

\item Let us apply  a Whitehead move onto the pair of diagonals $(i,j)$ and $(k,l)$. 
Since $(i,j)$ and $(k,l)$ are disjoint and belong to a given signature $\sigma$, their terminal vertices verify $k\equiv i\equiv  1\mod 4$ and  $j\equiv  l\equiv  3 \mod 4$ (see definition~\ref{De:1}). This pair is first modified by a half-Whitehead move into a pair of diagonals meeting at one point: $(i,k)(j,l)$. This pair is then modified by a smoothing Whitehead move, which induces the unique possible pair of diagonals $(i,l)(j,k)$. 
\end{itemize}
\end{proof}

\subsection{Adjacence relations}
We recall a few results from~\cite{C0}.
\begin{theorem}[Adjacence theorem]\label{T:Adj}
Let $d>3$. The relations between generic signatures obtained in one Whitehead-moves are the following: 
\begin{enumerate}
\item $M$-signatures are connected only to $d(d-1)$ $F$-signatures  ($\binom{d}{2}$ red and $\binom{d}{2}$ blue);
\item $F$-signatures are connected to $M$-, $S$-signatures;
\item $S$-signatures are connected only  to $FS$-signatures or $S$-signatures .
\end{enumerate}
\end{theorem}
\begin{proof}
\ 

\begin{itemize}

\item Consider an $M$ tree. By lemma~\ref{L:Diago2}, we know that in one Whitehead move applied onto a pair of the short diagonals, we obtain a signature having a pair of long diagonals. This is an $F$-signature. There exist $d$ blue or red diagonals. Choosing a pair of blue (or red) diagonals gives $\binom{d}{2}$ possibilities. Therefore we have $d(d-1)$ adjacent $F$-signatures to an $M$-signature.

\item Consider an $F$ signature. One Whitehead move applied onto a pair of short red diagonals in a pair of adjoining $F$ trees gives $S$ trees. If there exist no other long diagonals in the signature, then this is an $S$-signature. Otherwise, it is a $FS$-signature. From (1) we know that $F$-signatures are connected to $M$ signatures.

\item Consider two adjoining trees, one  of those trees being an $S$ tree. If the other one is an $M$ tree then applying a Whitehead move to a pair of adjoining diagonals gives an $FS$-signature (this follows form lemma~\ref{L:Diago2}). Otherwise, we have a couple of adjoining $F$ trees. 
\end{itemize}
\end{proof}


\begin{center}
\begin{tikzpicture}[scale=0.5]
\node (M2) at (0,0) {$\scriptstyle \left(\begin{smallmatrix}1\\3\end{smallmatrix}\right)$};
\node (F17) at (2.5,2.5) {$\scriptstyle\left|\begin{smallmatrix}1\\7\end{smallmatrix}\right|$};
 \node (F11193) at (3.6,1.1)  {$\scriptstyle\left|\begin{smallmatrix}1\\11\end{smallmatrix}\right|\left|\begin{smallmatrix}9\\3\end{smallmatrix}\right|$};
\node (F133) at (1.0,3.2) {$\scriptstyle\left|\begin{smallmatrix}13\\3\end{smallmatrix}\right|$}; 
\node (F515137)at (-3.6,1.1)  {$\scriptstyle \left|\begin{smallmatrix}5\\15\end{smallmatrix}\right| \left|\begin{smallmatrix}13\\7\end{smallmatrix}\right|$};
\node (F511) at (-2.5,2.5) {$\scriptstyle\left|\begin{smallmatrix}5\\11\end{smallmatrix}\right|$} ;
\node (F915) at (-1.0,3.2) {$\scriptstyle\left|\begin{smallmatrix}9\\15\end{smallmatrix}\right|$} ;
\node (F60) at (2.5,-2.5){$\scriptstyle\left|\begin{smallmatrix}6\\0\end{smallmatrix}\right|$};
\node (F28100) at (3.6,-1.2) {$\scriptstyle \left|\begin{smallmatrix}2\\8\end{smallmatrix}\right| \left|\begin{smallmatrix}10\\0\end{smallmatrix}\right|$};
\node (F212) at (1.0,-3.2) {$\scriptstyle\left|\begin{smallmatrix}2\\12\end{smallmatrix}\right|$};
\node (F144612) at (-3.6,-1.1)  {$\scriptstyle \left|\begin{smallmatrix}14\\4\end{smallmatrix}\right| \left|\begin{smallmatrix}6\\12\end{smallmatrix}\right|$};
\node (F104) at (-2.5,-2.5) {$\displaystyle\left|\begin{smallmatrix}10\\4\end{smallmatrix}\right|$};
\node (F148) at (-1.0,-3.2) {$\displaystyle\left|\begin{smallmatrix}14\\8\end{smallmatrix}\right|$};
\draw[->, very thick] (M2) -- (F17);
\draw[->,very thick] (M2) -- (F11193);
\draw[->,very thick] (M2) -- (F133);
\draw[->,very thick] (M2) -- (F515137);
\draw[->,very thick] (M2) -- (F511);
\draw[->,very thick] (M2) -- (F915);
\draw[->,very thick] (M2)--(F60);
\draw[->,very thick] (M2)--(F28100);
\draw[->,very thick] (M2)--(F212);
\draw[->,very thick] (M2)--(F144612);
\draw[->,very thick] (M2)--(F104);
\draw[->,very thick] (M2)--(F148);
\draw (0,-5) node { Deformation of  the $M_{2}$-signature $\scriptstyle\left(\begin{smallmatrix}1\\3\end{smallmatrix}\right)$} ;
\end{tikzpicture}
\end{center}

\begin{center}
\begin{tikzpicture}[scale=0.5]
\node (F0-10) at (0,0) { $\scriptstyle\left|\begin{smallmatrix}0\\10\end{smallmatrix}\right|$};
\node (M3) at (0,3) { $\scriptstyle M_{3}$} ;
\node (S1) at (0,-3) { $\scriptstyle\left|\begin{smallmatrix}1,11\\0,10\end{smallmatrix}\right|$} ;
\node (F0-6) at (-3,0) {$\scriptstyle\left|\begin{smallmatrix}0\\6\end{smallmatrix}\right|$} ;
\node (F4-10) at (3,0) { $\scriptstyle\left|\begin{smallmatrix}4\\10\end{smallmatrix}\right|$} ;
\node (FF2) at (2.5,2.5) {$\scriptstyle\left|\begin{smallmatrix}0\\10\end{smallmatrix}\right|\left|\begin{smallmatrix}8\\2\end{smallmatrix}\right|$} ;
\node (FF3) at (-2.5,2.5) { $\scriptstyle\left|\begin{smallmatrix}0\\10\end{smallmatrix}\right|\left|\begin{smallmatrix}3\\9\end{smallmatrix}\right|$} ;
\node (FS1) at (-2.5,-2.5) {$\left|\begin{smallmatrix}15,9\\0,10\end{smallmatrix}\right|\left|\begin{smallmatrix}7\\1\end{smallmatrix}\right|$} ;
\node (S2) at (2.5,-2.5) {$\scriptstyle\left|\begin{smallmatrix}15,5\\0,10\end{smallmatrix}\right|$} ;
\draw[->,very thick] (F0-10) -- (M3);
\draw[->,very thick] (F0-10) -- (F0-6);
\draw[->,very thick] (F0-10) -- (F4-10);
\draw[->,very thick] (F0-10) -- (FS1);
\draw[->,very thick] (F0-10) -- (FF2);
\draw[->,very thick] (F0-10)-- (FF3);
\draw[->,very thick] (F0-10) -- (S1);
\draw[->,very thick] (F0-10) -- (S2);
\draw (0,-5) node { Deformation of  the $F$-signature $\scriptstyle\left|\begin{smallmatrix}0\\10\end{smallmatrix}\right|$} ;
\end{tikzpicture}
\begin{tikzpicture}[scale=0.5]
\node (a) at (0,0) {$\scriptstyle \left[\begin{smallmatrix}1,11\\10,0\end{smallmatrix}\right]$} ;
\node (c) at (2.7,2.2) {$\scriptstyle\left|\begin{smallmatrix}3\\9\end{smallmatrix}\right|\left[\begin{smallmatrix}1,11\\10,0\end{smallmatrix}\right]$} ;
\node (i) at (2.7,-2.2) {$\scriptstyle\left|\begin{smallmatrix}8\\2\end{smallmatrix}\right|\left[\begin{smallmatrix}1,11\\10,0\end{smallmatrix}\right]$} ;
\node (e) at (-2.7,2.2) {$\scriptstyle\left|\begin{smallmatrix}7\\1\end{smallmatrix}\right|\left|\begin{smallmatrix}10\\0\end{smallmatrix}\right|$} ;
\node (g) at (-2.7,-2.2) { $\scriptstyle\left|\begin{smallmatrix}4\\10\end{smallmatrix}\right|\left|\begin{smallmatrix}1\\11\end{smallmatrix}\right|$} ;
\node (d) at (0,3) {$\scriptstyle\left|\begin{smallmatrix}11\\1\end{smallmatrix}\right|$} ;
\node (h) at (0,-3) { $\scriptstyle\left|\begin{smallmatrix}0\\10\end{smallmatrix}\right|$} ;
\node (b) at (4,0) {$\scriptstyle\left|\begin{smallmatrix}1,11\\0,6\end{smallmatrix}\right|$} ;
\node (f) at (-4,0) { $\scriptstyle\left|\begin{smallmatrix}0,10\\11,5\end{smallmatrix}\right|$} ;

\draw[->,very thick] (a) -- (d);
\draw[->,very thick] (a) -- (b);
\draw[->,very thick] (a) -- (c);
\draw[->,very thick] (a) -- (d);
\draw[->,very thick] (a) -- (e);
\draw[->,very thick] (a) -- (f);
\draw[->,very thick] (a) -- (g);
\draw[->,very thick] (a) -- (h);
\draw[->,very thick] (a) -- (i);
\draw (0,-5) node { Deformation of  the $S$-signature $\scriptstyle \left[\begin{smallmatrix}1,11\\10,0\end{smallmatrix}\right]$} ;
\end{tikzpicture}
\end{center}
\begin{figure}[h]
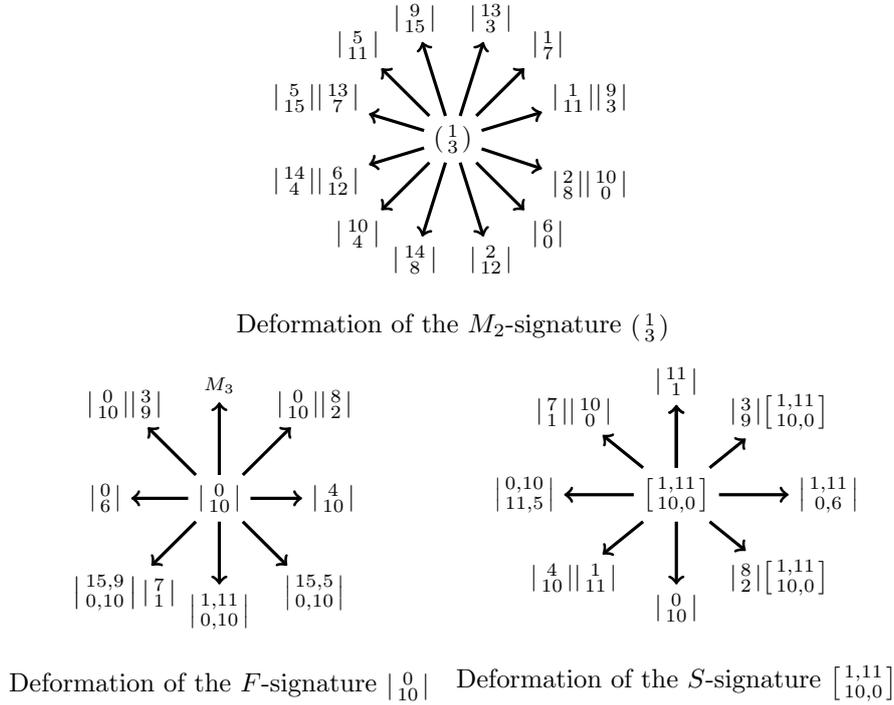

\vspace{-10pt}
\caption{ Example of deformations of an $M,F$ and $S$-signatures} 
\end{figure}

\begin{theorem}~\cite{C0}\label{Th:Path}
Let $\sigma\in \Sigma_{d}$ be a generic signature. Then, for every $\sigma$ there exists a sequence of Whitehead moves starting at $\sigma$ and ending on an $M$-signature. 
\end{theorem}
\begin{proof}
The proof is by induction on the number of long blue diagonals.  
Let $\sigma$ be generic signature such that on the left side of a long blue diagonal there exist only short diagonals of red and blue color.
\begin{enumerate}

\item {\it Base case}. 
Let $\sigma$ have only one long blue diagonal. Then two cases are discussed:
\begin{enumerate}
\item The red diagonals are all short. 
\item Not all red diagonals are short. 
\end{enumerate}
Consider the first case. Let us apply one of the lemma~\ref{L:Diago1} onto the long blue diagonal and the blue short diagonals on its right side.
Each such Whitehead move step increases by one the number of short blue diagonals. So, we proceed using this method until there are $d$ short diagonals in the signature: this is an $M$ signature.
Consider the second case, where we have $k$ red long diagonals. Then, the union of the long blue diagonal and of these $k$ long red diagonals compartment the signature into $q$ disjoint adjacent  2-cells lying in $\mathbb{C}\setminus\sigma$. If we consider each of these regions independently from the signature, we can interpret them as local $M$-signatures of smaller degree than $d$. 

Apply lemma~\ref{L:Diago1} to the long blue diagonal lying in one compartment a finite number of times (Whitehead moves on the long blue diagonal and short blue diagonals in the compartment). This compartment is thus, locally, an $M$-signature. Let $L$ be the new long blue diagonal, obtained from this procedure. This diagonal $L$ intersects now an adjacent compartment. 

As previously, we apply lemma~\ref{L:Diago1} a finite number of times onto $L$ and the short blue diagonals in the adjacent compartment, in order to have a local $M$-signature. 

Proceeding in this way on the blue long diagonal for all adjacent compartments, gives in final only short blue diagonals in the signature, which defines an F-signature. The remaining step is to apply the same procedure done for the blue long diagonals onto the red long diagonals: after a finite number of Whitehead moves all the red diagonals are short and this defines an $M$ signature.\\

\item {\it Induction case.}
 Suppose that for a signature with $m$ long blue diagonals there exists a path from the signature to an $M$ signature. Let us show that for $m+1$ long diagonals this statement is also true. Take a block of adjacent $m$ long diagonals and apply the induction hypothesis to it. Then there exists a finite number of deformations such that these $m$ long blue and red diagonals are all short, leaving only one long blue diagonal in the signature. We can thus apply the case (1) from the discussion above. 
\end{enumerate}
\end{proof}

\begin{remark}

The theorem~\ref{Th:Path} above can be interpreted as the path connectedness of $\dpol_{d}$, from which we recover the fact that $\dpol_{d}$ is connected since, as the complement of a hyperplane arrangement it is a open subset of $\mathbb{C}^{d}$. 
\end{remark}

\section{Inclusion diagrams}~\label{S:incl}

In this section we show the existence of geometric invariants of configuration spaces. These geometric invariants (that we call {\it inclusion diagrams}) are obtained from the topological stratification. These objects are constructed by studying the incidence and adjacence relations between strata.  As a corollary from previous works~\cite{Co1}, we show that these geometric objects are in bijection with the nerve of the \v Cech cover. For simplicity, we define those geometric invariants directly as the nerve of this cover. 

\subsection{Combinatorial closure of a signature}

A contracting half Whitehead move  defines a partial strict order ($\prec$) on  the set  $\Sigma_{d}$ of all signatures. 
Let $\sigma, \tau \in\Sigma_{d} $ in such a subset,  we say that $\sigma \prec \tau$ if there exists a sequence of signatures $\sigma=\sigma_{1} \prec \sigma_{2} ... \prec\sigma_{n}=\tau$. The symbol $\prec$ is an incidence relation between those signatures.

\vskip.1cm
\begin{lemma}
The incidence relation on the set $\Sigma_{d}$ forms a partial strict order. 
\end{lemma}
\begin{proof}
Clearly, we have an irreflexive relation: $\sigma\prec \sigma$ does not hold for any $\sigma$ in $\Sigma_{d}$, and a transitive relation. We prove that the relation is antisymmetric. 
If one has $\sigma\prec\tau$ then $codim(\tau) > codim(\sigma)$. This strict order $\prec$ induces the partial partial order $\preceq$ by\[\sigma \preceq \tau = \begin{cases}\sigma \prec \tau &\text{if\  } codim(\tau) > codim(\sigma),\\ \sigma =\tau &\text{if } codim(\tau)= codim(\sigma). \end{cases}\]

\end{proof}
 \begin{definition}
 The set of $\sigma$ and all signatures incident to the signature $\sigma$ will be denoted by $\overline{\sigma}$.
 By abuse of notation we call $\overline{\sigma}$ the  combinatorial closure of $\sigma$.
\end{definition}

\vskip.1cm 

From~\cite{C0,C1} we know that the combinatorial closure of an elementa is equivalent to its topological closure, i.e.: \[A_{\overline{\sigma}}= \bigcup_{\tau\in \bar \sigma}A_{\tau}=\overline{A_{\sigma}}.\]
 
\subsection{Properties of inclusion diagram}

Let us consider the union of topological closure of $\overline{A_{\sigma}}$ and its tubular $Tub$ neighborhood~\cite{C0,C1}, we call  $A^{+}_{\sigma}=\overline{A_{\sigma}}\cup Tub$ the thickened elementa. The set of thickened elementa  $A^{+}_{\sigma_{g}}$ for generic $\sigma_{g}$  forms a good cover of $\dpol_{d}$, in the sense of \v Cech (i.e. multiple intersections are either empty or contractible).  The {\it nerve} $\mathcal{N}$ of this \v Cech covering
is the cell complex on the vertex set $\Sigma_{G}$ of generic signatures, consisting of those finite non-empty subsets $ I$ of  $\Sigma_{G}$ such that $\cap_{i\in I}$ $A^{+}_{\sigma_{i}}\neq \emptyset$.

\begin{definition}\label{D:dual}
We call inclusion diagram the cell complex $(\mathcal{W},\subset)$, where the set of its $k$-faces is in bijection with the set of codimension $k$-elementa and which is order-preserving.
 \end{definition}

So, if a couple $(W_{\sigma}, W_{\tau})$ in the inclusion diagram verifies $W_{\sigma}\subset \overline{W_{\tau}}$, then $\sigma\prec \tau$.

As well, suppose that  an element $A_{\mu}$ of codimension $k$ is incident to a collection of elementa $A_{\sigma_{1}},...,A_{\sigma_{m}}$ of smaller codimensions,  such that $A_{\mu} \subset  \bigcap_{i=1}^{m} \overline{A_{\sigma_{i}}}$. Then, in the inclusion diagram, the collection of faces $W_{\sigma_{1}},...,W_{\sigma_{m}}$ lies in the boundary of a $k$-dimensional face $W_{\mu}$. 

\begin{corollary}
The inclusion diagram is isomorphic to the nerve $\mathcal{N}$ of the cover $(A^{+}_{\sigma_{g}})_{\sigma_{g}\in\Sigma_{G}}$.
\end{corollary}

\begin{lemma}[Edges and 2-faces of the inclusion diagram]~\label{Lem:quadra}
Let $(\mathcal{W},\subset)$ be the inclusion diagram associated to $(A_{\sigma})_{\sigma\in {\rm \Sigma}_d}$. 
Then:
\begin{enumerate}
\item each 1-dimensional face in $\mathcal{W}$ is bounded by 2 vertices. 
\item each 2-dimensional face in $\mathcal{W}$ is bounded by 4 vertices, 4 edges and forms a quadrangle. 
\end{enumerate}
\end{lemma}
\begin{proof}
Statement (1):
Consider a signature $\beta$, of codimension 1. Then, by definition~\ref{D:cod} there exists a pair of intersecting chords of the same color. 
Suppose that the set of indexes of terminal vertices of those diagonals is $\{i,j,k,l\}$ where $i<j<k<l$. Those numbers $i,j,k,l$ are of the same parity and by definition~\ref{De:1} verify: $i\equiv k\equiv 1\mod 4$ and $j\equiv l\equiv 3\mod 4$ (resp. $i\equiv k\equiv 2\mod 4$ and $j\equiv l\equiv 0\mod 4$). Again, from definition ~\ref{De:1} we know that in a generic signature each terminal vertex congruent to 1 $\mod 4$ (resp. 2 $\mod 4$) is attached by an edge to a terminal vertex which is congruent to $3 \mod 4$ (resp. 0 $\mod 4$ ). So, using a smoothing half-Whitehead move 
the intersection point is smoothed and we obtain two different possible pairs of diagonals: $(i,j)(k,l)$ or $(i,l)(j,k)$, with all the other diagonals of the signature remaining invariant. 
So, applying the definition~\ref{D:dual} to construct the inclusion diagram, we have that each 1-dimensional face (corresponding to a codimension 1 signature) in $\mathcal{W}$ is bounded by exactly 2 vertices (corresponding to the signatures obtained by smoothing the meeting point in $\beta$).

\vspace{.2cm}

Statement (2):
Consider a signature of codimension 2, denoted by $\omega$. By definition~\ref{D:cod}, there exist two critical points. Applying the smoothing half-Whitehead move onto one of the critical points gives two different possible signatures of codimension 1 (this last statement follows from the first point above). So, applying the same arguments to the second critical point, implies that there exist four signatures of codimension 1, incident to $\omega$. In other words: there exist $\{\beta_{0},\beta_{1},\beta_{2},\beta_{3}\}\prec\omega$, where $codim(\beta_{i})=1$ and $i\in \{0,...,3\}$. The smoothing modification applied simultaneously to both critical points, gives four codimension 0 signatures, all incident to $\omega$: $\{\sigma_{0},\sigma_{1},\sigma_{2},\sigma_{3}\}\prec \omega$.  Applying (1) to every diagram of codimension 1  implies that there exist two signatures of codimension 0 , which are incident to each signature of codimension 1.  So, we obtain the following relations: 
 \[\{\sigma_{0},\sigma_{1}\}\prec \beta_{0},\]
 \[\{ \sigma_{1},\sigma_{2}\}\prec \beta_{1},\]
 \[ \{\sigma_{2},\sigma_{3}\}\prec \beta_{2},\]
 \[ \{\sigma_{3},\sigma_{0}\}\prec \beta_{3}\]
The construction of the inclusion diagram $\mathcal{W}$ from definition~\ref{D:dual}, implies that we have a quadrangle. 
 \end{proof} 
\begin{example}
An explicit construction of the inclusion diagrams is given below, for $d=2$. 
We illustrate the relations between the diagrams.

There exist 4 generic signatures and 4 signatures  of codimension 1, illustrated on the figure below.

\begin{center}~\label{F:d=2}
\begin{tikzpicture}[scale=.8]
\node (a) at (-4,0) {
 \begin{tikzpicture}[scale=.6]
\newcommand{\degree}[0]{4}
\newcommand{\B}[1]{#1*180/\degree}
\newcommand{\R}[1]{#1*180/\degree }
\draw[black,thick] (0,0) circle (1) ;
\foreach \k in {0,2,4,6} {
  \draw[blue] (\B{\k}:1.15) node { \tiny \k};} ;
  \foreach \k in {1,3,5,7} {
 \draw[ red] (\R{\k}:1.15) node { \tiny \k} ;
};
\draw[blue, name path=B0B6](\B{0}:1) .. controls(\B{0}:0.7)and (\B{6}:0.7) .. (\B{6}:1) ;
\draw[blue, name path=B2B4](\B{2}:1) .. controls(\B{2}:0.7)and (\B{4}:0.7) .. (\B{4}:1) ;
\draw[red, name path=R1R3](\R{1}:1) .. controls(\R{1}:0.7)and (\R{3}:0.7) .. (\R{3}:1) ;
\draw[red, name path=R5R7](\R{5}:1) .. controls(\R{5}:0.7)and (\R{7}:0.7) .. (\R{7}:1) ;
\end{tikzpicture} 
};
\node (b) at (0,0) {
\begin{tikzpicture}[scale=.6]
\newcommand{\degree}[0]{4}
\newcommand{\B}[1]{#1*180/\degree}
\newcommand{\R}[1]{#1*180/\degree }
\newcommand{\RDOT}[0]{[red](i-1) circle (0.05)}
\newcommand{\BDOT}[0]{[blue](i-1) circle (0.05)}
\draw[black,thick] (0,0) circle (1) ;
\foreach \k in {0,2,4,6} {
  \draw[blue] (\B{\k}:1.15) node {\tiny \k};} ;
  \foreach \k in {1,3,5,7} {
 \draw[ red] (\R{\k}:1.15) node {\tiny\k} ;
}
\draw[blue, name path=B0B4](\B{0}:1) .. controls(\B{0}:0.7)and (\B{4}:0.7) .. (\B{4}:1) ;
\draw[blue, name path=B2B6](\B{2}:1) .. controls(\B{2}:0.7)and (\B{6}:0.7) .. (\B{6}:1) ;
\draw[red, name path=R1R3](\R{1}:1) .. controls(\R{1}:0.7)and (\R{3}:0.7) .. (\R{3}:1) ;
\draw[red,,name path=R5R7](\R{5}:1) .. controls(\R{5}:0.7)and (\R{7}:0.7) .. (\R{7}:1) ;
\end{tikzpicture} 
};
\node (c) at (4,0) {
\begin{tikzpicture}[scale=.6]
\newcommand{\degree}[0]{4}
\newcommand{\B}[1]{#1*180/\degree}
\newcommand{\R}[1]{#1*180/\degree }
\newcommand{\RDOT}[0]{[red](i-1) circle (0.05)}
\newcommand{\BDOT}[0]{[blue](i-1) circle (0.05)}
\draw[black,thick] (0,0) circle (1) ;
\foreach \k in {0,2,4,6} {
  \draw[blue] (\B{\k}:1.15) node {\tiny \k};} ;
  \foreach \k in {1,3,5,7} {
 \draw[ red] (\R{\k}:1.15) node {\tiny \k} ;
};
\draw[blue, name path=B0B2](\B{0}:1) .. controls(\B{0}:0.7)and (\B{2}:0.7) .. (\B{2}:1) ;
\draw[blue, name path=B4B6](\B{4}:1) .. controls(\B{4}:0.7)and (\B{6}:0.7) .. (\B{6}:1) ;
\draw[red, name path=R5R7](\R{5}:1) .. controls(\R{5}:0.7)and (\R{7}:0.7) .. (\R{7}:1) ;
\draw[red,name path=R1R3](\R{1}:1) .. controls(\R{1}:0.7)and (\R{3}:0.7) .. (\R{3}:1) ;
\end{tikzpicture} 
};
 \draw[<->,thick] (b) -- (a);
  \draw[<->,thick] (b) -- (c);
\end{tikzpicture}

\begin{tikzpicture}[scale=.8]
\node (a) at (-4,0) {
 \begin{tikzpicture}[scale=.6]
\newcommand{\degree}[0]{4}
\newcommand{\B}[1]{#1*180/\degree}
\newcommand{\R}[1]{#1*180/\degree }
\newcommand{\RDOT}[0]{[red](i-1) circle (0.05)}
\newcommand{\BDOT}[0]{[blue](i-1) circle (0.05)}
\draw[black,thick] (0,0) circle (1) ;
\foreach \k in {0,2,4,6} {
  \draw[blue] (\B{\k}:1.15) node { \tiny\k};} ;
  \foreach \k in {1,3,5,7} {
 \draw[ red] (\R{\k}:1.15) node {\tiny\k} ;
}
\draw[blue, name path=B0B6](\B{0}:1) .. controls(\B{0}:0.7)and (\B{6}:0.7) .. (\B{6}:1) ;
\draw[blue, name path=B2B4](\B{2}:1) .. controls(\B{2}:0.7)and (\B{4}:0.7) .. (\B{4}:1) ;
\draw[red, name path=R1R7](\R{1}:1) .. controls(\R{1}:0.7)and (\R{7}:0.7) .. (\R{7}:1) ;
\draw[red, name path=R3R5](\R{3}:1) .. controls(\R{3}:0.7)and (\R{5}:0.7) .. (\R{5}:1) ;
\end{tikzpicture} 
};
\node (b) at (0,0) {
\begin{tikzpicture}[scale=.6]
\newcommand{\degree}[0]{4}
\newcommand{\B}[1]{#1*180/\degree}
\newcommand{\R}[1]{#1*180/\degree }
\newcommand{\RDOT}[0]{[red](i-1) circle (0.05)}
\newcommand{\BDOT}[0]{[blue](i-1) circle (0.05)}
\draw[black,thick] (0,0) circle (1) ;
\foreach \k in {0,2,4,6} {
  \draw[blue] (\B{\k}:1.15) node {\tiny \k};} ;
  \foreach \k in {1,3,5,7} {
 \draw[ red] (\R{\k}:1.15) node {\tiny \k} ;
};
\draw[blue, name path=B0B4](\B{0}:1) .. controls(\B{0}:0.7)and (\B{4}:0.7) .. (\B{4}:1) ;
\draw[blue, name path=B2B6](\B{2}:1) .. controls(\B{2}:0.7)and (\B{6}:0.7) .. (\B{6}:1) ;
\draw[red, name path=R1R7](\R{1}:1) .. controls(\R{1}:0.7)and (\R{7}:0.7) .. (\R{7}:1) ;
\draw[red, name path=R3R5](\R{3}:1) .. controls(\R{3}:0.7)and (\R{5}:0.7) .. (\R{5}:1) ;
\end{tikzpicture} 
};
\node (c) at (4,0) {
\begin{tikzpicture}[scale=.6]
\newcommand{\degree}[0]{4}
\newcommand{\B}[1]{#1*180/\degree}
\newcommand{\R}[1]{#1*180/\degree }
\newcommand{\RDOT}[0]{[red](i-1) circle (0.05)}
\newcommand{\BDOT}[0]{[blue](i-1) circle (0.05)}
\draw[black,thick] (0,0) circle (1) ;
\foreach \k in {0,2,4,6} {
  \draw[blue] (\B{\k}:1.15) node {\tiny \k};} ;
  \foreach \k in {1,3,5,7} {
 \draw[ red] (\R{\k}:1.15) node {\tiny\k} ;
};
\draw[blue, name path=B0B2](\B{0}:1) .. controls(\B{0}:0.7)and (\B{2}:0.7) .. (\B{2}:1) ;
\draw[blue, name path=B4B6](\B{4}:1) .. controls(\B{4}:0.7)and (\B{6}:0.7) .. (\B{6}:1) ;
\draw[red, name path=R1R7](\R{1}:1) .. controls(\R{1}:0.7)and (\R{7}:0.7) .. (\R{7}:1) ;
\draw[red, name path=R3R5](\R{3}:1) .. controls(\R{3}:0.7)and (\R{5}:0.7) .. (\R{5}:1) ;
\end{tikzpicture} 
};
 \draw[<->,thick] (b) -- (a);
  \draw[<->,thick] (b) -- (c);
\end{tikzpicture}

\begin{tikzpicture}[scale=.8]
\node (a) at (-4,0) {
 \begin{tikzpicture}[scale=.6]
\newcommand{\degree}[0]{4}
\newcommand{\B}[1]{#1*180/\degree}
\newcommand{\R}[1]{#1*180/\degree }
\newcommand{\RDOT}[0]{[red](i-1) circle (0.05)}
\newcommand{\BDOT}[0]{[blue](i-1) circle (0.05)}
\draw[black,thick] (0,0) circle (1) ;
\foreach \k in {0,2,4,6} {
  \draw[blue] (\B{\k}:1.15) node {\tiny\k};} ;
  \foreach \k in {1,3,5,7} {
 \draw[ red] (\R{\k}:1.15) node {\tiny\k} ;
};
\draw[blue, name path=B0B6](\B{0}:1) .. controls(\B{0}:0.7)and (\B{6}:0.7) .. (\B{6}:1) ;
\draw[blue, name path=B2B4](\B{2}:1) .. controls(\B{2}:0.7)and (\B{4}:0.7) .. (\B{4}:1) ;
\draw[red, name path=R1R3](\R{1}:1) .. controls(\R{1}:0.7)and (\R{3}:0.7) .. (\R{3}:1) ;
\draw[red, name path=R5R7](\R{5}:1) .. controls(\R{5}:0.7)and (\R{7}:0.7) .. (\R{7}:1) ;
\end{tikzpicture} 
};
\node (b) at (0,0) {
\begin{tikzpicture}[scale=.6]
\newcommand{\degree}[0]{4}
\newcommand{\B}[1]{#1*180/\degree}
\newcommand{\R}[1]{#1*180/\degree }
\newcommand{\RDOT}[0]{[red](i-1) circle (0.05)}
\newcommand{\BDOT}[0]{[blue](i-1) circle (0.05)}
\draw[black,thick] (0,0) circle (1) ;
\foreach \k in {0,2,4,6} {
  \draw[blue] (\B{\k}:1.15) node {\tiny\k};} ;
  \foreach \k in {1,3,5,7} {
 \draw[ red] (\R{\k}:1.15) node {\tiny\k} ;
};
\draw[blue, name path=B0B6](\B{0}:1) .. controls(\B{0}:0.7)and (\B{6}:0.7) .. (\B{6}:1) ;
\draw[blue, name path=B2B4](\B{2}:1) .. controls(\B{2}:0.7)and (\B{4}:0.7) .. (\B{4}:1) ;
\draw[red, name path=R1R5](\R{1}:1) .. controls(\R{1}:0.7)and (\R{5}:0.7) .. (\R{5}:1) ;
\draw[red,,name path=R3R7](\R{3}:1) .. controls(\R{3}:0.7)and (\R{7}:0.7) .. (\R{7}:1) ;
\end{tikzpicture} 
};
\node (c) at (4,0) {
\begin{tikzpicture}[scale=.6]
\newcommand{\degree}[0]{4}
\newcommand{\B}[1]{#1*180/\degree}
\newcommand{\R}[1]{#1*180/\degree }
\newcommand{\RDOT}[0]{[red](i-1) circle (0.05)}
\newcommand{\BDOT}[0]{[blue](i-1) circle (0.05)}
\draw[black,thick] (0,0) circle (1) ;
\foreach \k in {0,2,4,6} {
  \draw[blue] (\B{\k}:1.15) node {\tiny \k};} ;
  \foreach \k in {1,3,5,7} {
 \draw[ red] (\R{\k}:1.15) node {\tiny\k} ;
};
\draw[blue, name path=B0B6](\B{0}:1) .. controls(\B{0}:0.7)and (\B{6}:0.7) .. (\B{6}:1) ;
\draw[blue, name path=B2B4](\B{2}:1) .. controls(\B{2}:0.7)and (\B{4}:0.7) .. (\B{4}:1) ;
\draw[red, name path=R5R7](\R{5}:1) .. controls(\R{5}:0.7)and (\R{7}:0.7) .. (\R{7}:1) ;
\draw[red,name path=R1R3](\R{1}:1) .. controls(\R{1}:0.7)and (\R{3}:0.7) .. (\R{3}:1) ;
\end{tikzpicture} 
};
 \draw[<->,thick] (b) -- (a);
  \draw[<->,thick] (b) -- (c);
\end{tikzpicture} 


\begin{tikzpicture}[scale=.8]
\node (a) at (-4,0) {
 \begin{tikzpicture}[scale=.6]
\newcommand{\degree}[0]{4}
\newcommand{\B}[1]{#1*180/\degree}
\newcommand{\R}[1]{#1*180/\degree }
\newcommand{\RDOT}[0]{[red](i-1) circle (0.05)}
\newcommand{\BDOT}[0]{[blue](i-1) circle (0.05)}
\draw[black,thick] (0,0) circle (1) ;
\foreach \k in {0,2,4,6} {
  \draw[blue] (\B{\k}:1.15) node { \tiny\k};} ;
  \foreach \k in {1,3,5,7} {
 \draw[ red] (\R{\k}:1.15) node {\tiny\k} ;
};
\draw[blue, name path=B0B2](\B{0}:1) .. controls(\B{0}:0.7)and (\B{2}:0.7) .. (\B{2}:1) ;
\draw[blue, name path=B4B6](\B{4}:1) .. controls(\B{4}:0.7)and (\B{6}:0.7) .. (\B{6}:1) ;
\draw[red, name path=R1R3](\R{1}:1) .. controls(\R{1}:0.7)and (\R{3}:0.7) .. (\R{3}:1) ;
\draw[red, name path=R5R7](\R{5}:1) .. controls(\R{5}:0.7)and (\R{7}:0.7) .. (\R{7}:1) ;
\end{tikzpicture} 
};
\node (b) at (0,0) {
\begin{tikzpicture}[scale=.6]
\newcommand{\degree}[0]{4}
\newcommand{\B}[1]{#1*180/\degree}
\newcommand{\R}[1]{#1*180/\degree }
\newcommand{\RDOT}[0]{[red](i-1) circle (0.05)}
\newcommand{\BDOT}[0]{[blue](i-1) circle (0.05)}
\draw[black,thick] (0,0) circle (1) ;
\foreach \k in {0,2,4,6} {
  \draw[blue] (\B{\k}:1.15) node {\tiny \k};} ;
  \foreach \k in {1,3,5,7} {
 \draw[ red] (\R{\k}:1.15) node {\tiny \k} ;
};
\draw[blue, name path=B0B2](\B{0}:1) .. controls(\B{0}:0.7)and (\B{2}:0.7) .. (\B{2}:1) ;
\draw[blue, name path=B4B6](\B{4}:1) .. controls(\B{4}:0.7)and (\B{6}:0.7) .. (\B{6}:1) ;
\draw[red, name path=R1R5](\R{1}:1) .. controls(\R{1}:0.7)and (\R{5}:0.7) .. (\R{5}:1) ;
\draw[red,,name path=R3R7](\R{3}:1) .. controls(\R{3}:0.7)and (\R{7}:0.7) .. (\R{7}:1) ;
\end{tikzpicture} 
};
\node (c) at (4,0) {
\begin{tikzpicture}[scale=.6]
\newcommand{\degree}[0]{4}
\newcommand{\B}[1]{#1*180/\degree}
\newcommand{\R}[1]{#1*180/\degree }
\newcommand{\RDOT}[0]{[red](i-1) circle (0.05)}
\newcommand{\BDOT}[0]{[blue](i-1) circle (0.05)}
\draw[black,thick] (0,0) circle (1) ;
\foreach \k in {0,2,4,6} {
  \draw[blue] (\B{\k}:1.15) node {\tiny \k};} ;
  \foreach \k in {1,3,5,7} {
 \draw[ red] (\R{\k}:1.15) node {\tiny \k} ;
};
\draw[blue, name path=B0B2](\B{0}:1) .. controls(\B{0}:0.7)and (\B{2}:0.7) .. (\B{2}:1) ;
\draw[blue, name path=B4B6](\B{4}:1) .. controls(\B{4}:0.7)and (\B{6}:0.7) .. (\B{6}:1) ;
\draw[red, name path=R3R5](\R{3}:1) .. controls(\R{3}:0.7)and (\R{5}:0.7) .. (\R{5}:1) ;
\draw[red,name path=R1R7](\R{1}:1) .. controls(\R{1}:0.7)and (\R{7}:0.7) .. (\R{7}:1) ;
\end{tikzpicture} 
};
 \draw[<->,thick] (b) -- (a);
  \draw[<->,thick] (b) -- (c);
\end{tikzpicture} 
\end{center}
\vspace{-10pt}
\begin{figure}[h]
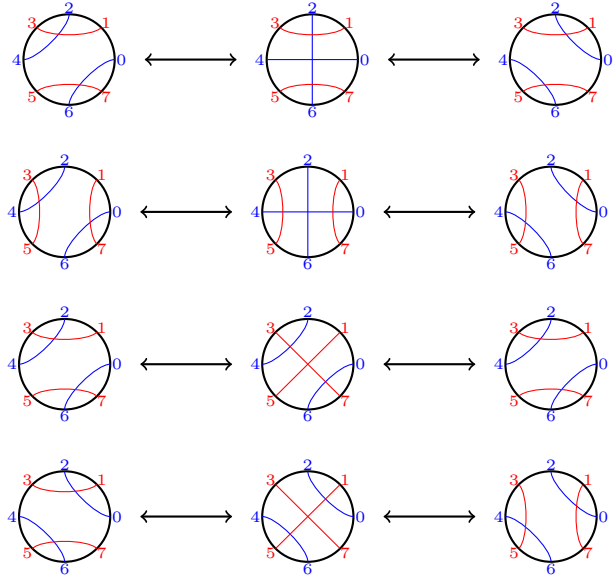

\caption{Relations between diagrams for $d=2$}
\end{figure}

Note, that in this case ($d=2$), there are {\it no} 2-faces. Indeed, composing by two Whitehead moves on the pairs of chords of red and blue color, gives a a superimposition of chord diagrams which is not compatible with the definition~\ref{De:1}. 

Therefore, the inclusion diagram  $\mathcal{W}$  for this $d=2$ case, is a quadrangle (see Fig~\ref{F:2-face}) constituted from:
\begin{enumerate}
\item four vertices, corresponding to the generic signatures,
\item four edges, corresponding to the codimension 1 signatures
\end{enumerate}

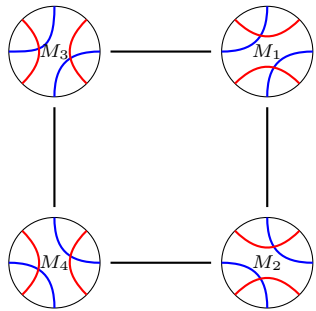
\begin{figure}[h]
\begin{center}
 \begin{tikzpicture}[scale=0.7]
\node (M1) at (2,2) {
 \begin{tikzpicture}[scale=0.6]
\newcommand{\degree}[0]{8}
\newcommand{\last}[0]{7}
\newcommand{\B}[1]{#1*360/\degree}
\draw[black] (0,0) circle (1) ;
\draw[blue, thick, name path=B0B6](\B{0}:1) .. controls(\B{0}:0.3)and (\B{6}:0.3) .. (\B{6}:1) ;
\draw[blue, thick, name path=B2B4](\B{2}:1) .. controls(\B{2}:0.3)and (\B{4}:0.3) .. (\B{4}:1) ;
\draw[red,thick, name path=B1B3](\B{1}:1) .. controls(\B{1}:0.3)and (\B{3}:0.3) .. (\B{3}:1) ;
\draw[red,thick, name path=B5B7](\B{5}:1) .. controls(\B{5}:0.3)and (\B{7}:0.3) .. (\B{7}:1) ;
\end{tikzpicture} 
};
\node (M2) at (2,-2) {
\begin{tikzpicture}[scale=0.6]
\newcommand{\degree}[0]{8}
\newcommand{\last}[0]{7}
\newcommand{\B}[1]{#1*360/\degree}
\draw[black] (0,0) circle (1) ;
\draw[blue, thick, name path=B0B2](\B{0}:1) .. controls(\B{0}:0.3)and (\B{2}:0.3) .. (\B{2}:1) ;
\draw[blue, thick, name path=B4B6](\B{4}:1) .. controls(\B{4}:0.3)and (\B{6}:0.3) .. (\B{6}:1) ;
\draw[red, thick, name path=B5B7](\B{5}:1) .. controls(\B{5}:0.3)and (\B{7}:0.3) .. (\B{7}:1) ;
\draw[red, thick, name path=B1B3](\B{1}:1) .. controls(\B{1}:0.3)and (\B{3}:0.3) .. (\B{3}:1) ;
\end{tikzpicture} 
};
\node (M3) at (-2,2) {
\begin{tikzpicture}[scale=0.6]
\newcommand{\degree}[0]{8}
\newcommand{\last}[0]{7}
\newcommand{\B}[1]{#1*360/\degree}
\draw[black] (0,0) circle (1) ;
\draw[blue, thick, name path=B0B6](\B{0}:1) .. controls(\B{0}:0.3)and (\B{6}:0.3) .. (\B{6}:1) ;
\draw[blue, thick, name path=B2B4](\B{2}:1) .. controls(\B{2}:0.0)and (\B{4}:0.3) .. (\B{4}:1) ;
\draw[red, thick,  name path=B1B7](\B{1}:1) .. controls(\B{1}:0.3)and (\B{7}:0.3) .. (\B{7}:1) ;
\draw[red, thick, ,name path=B3B5](\B{3}:1) .. controls(\B{3}:0.3)and (\B{5}:0.3) .. (\B{5}:1) ;
\end{tikzpicture} 
};
\node (M4) at (-2,-2) {
\begin{tikzpicture}[scale=0.6]
\newcommand{\degree}[0]{8}
\newcommand{\last}[0]{7}
\newcommand{\B}[1]{#1*360/\degree}
\draw[black] (0,0) circle (1) ;
\draw[blue, thick, name path=B0B2](\B{0}:1) .. controls(\B{0}:0.3)and (\B{2}:0.3) .. (\B{2}:1) ;
\draw[blue, thick, name path=B4B6](\B{4}:1) .. controls(\B{4}:0.3)and (\B{6}:0.3) .. (\B{6}:1) ;
\draw[red, thick, name path=B3B5](\B{3}:1) .. controls(\B{3}:0.3)and (\B{5}:0.3) .. (\B{5}:1) ;
\draw[red, thick,name path=B1B7](\B{1}:1) .. controls(\B{1}:0.3)and (\B{7}:0.3) .. (\B{7}:1) ;
\end{tikzpicture} 
};

 \draw[-, black,thick] (M1) -- (M2);
 \draw[-,black,thick] (M1) -- (M3);
 \draw[-,black,thick] (M2) -- (M4);
 \draw[-,black,thick] (M4) -- (M3);
 \draw (2,2) node {$\scriptstyle M_{1}$};     
  \draw (2,-2) node { $\scriptstyle M_{2}$}; 
   \draw (-2,2) node { $\scriptstyle M_{3}$};     
  \draw (-2,-2) node {$\scriptstyle M_{4}$}; 
\end{tikzpicture} 
\end{center}
\vspace{-10pt}
\caption{Inclusion diagram for $d=2$}\label{F:2-face}
\end{figure}
\end{example}
\begin{corollary}
\ 

\begin{itemize}
\item Let $\sigma_0$ and $\sigma_1$ be two generic signatures.
If  $\sigma_0$ and $\sigma_1$ are both incident to a signature of codimension 1, then this signature of codimension 1 is unique. 
\item Let $\sigma_0,\sigma_1,\sigma_2,\sigma_3$ be four generic signatures. If these signatures are incident to a signature of codimension 2 , then this signature of codimension 2 is unique. 
\end{itemize}
\end{corollary}

\subsection{Structure of inclusion diagrams}

\begin{theorem}
Let $d>2$. The inclusion diagram is a cell-complex. 
\end{theorem}

\begin{proof}
Let $X^{n}$ be the set set of faces of dimension $n$. 
We know that each $n$-face in $X^{n}$ is a topological ball of dimension $n$.
Indeed, for $n=0$: the set $X^{0}$ is the set of vertices of the inclusion diagram. 
Now, for any $n$, the set $X^{n}$ is in bijection with the set of codimension $n$ signatures, which are the indices of elementa of codimension $n$. Those elementa are known to be topological balls of codimension $n$, from~\cite{C0,C1}. So, the face of $X^{n}$ are $n$-balls.
The glueing between faces is done by the following criterion: 
if two signatures verify $\tau\prec \tau'$ such that $k=codim(\tau)<codim(\tau')=k+1$, then $A_{\tau'}\subset A_{\bar{\tau}}$. In particular, the $k$-face $f_{\tau}$ in $X^{k}$ (corresponding to $\tau$) lies in the boundary of the face $f_{\tau'}$ (corresponding to $\tau'$) where $dim(f_{\tau})+1=dim(f_{\tau'})$.          
\end{proof}

For any $d>2$, the inclusion diagram contains two main parts: a part that we call {\it exterior} part and a part to which we refer as an {\it interior} part.

\begin{enumerate}
\item  The \underline{ exterior part} is formed from a necklace of four beads. Those beads, denoted $NC_{d}$, are obtained from a union of cells of the cell complex, forming a connected component.
A structure $NC_{d}$ is a linearly ordered subset of  $(\Sigma_{d},\prec)$ with one upper bound $\sigma$ having $d$ blue (resp. red) chords intersecting at one point. The notation $NC_{d}$ (or $NC$ in short) is due to the relation with the non-crossing partition of the set of $d$ elements $\{1,2,...,d\}$~\cite{Kre72}.  
Indeed, there is a bijection between the set of vertices in a $NC_{d}$ structure and the set of non-crossing partitions of $\{1,2,...,d\}$. 
Consider the upper bound of an $NC$ structure: it is a signature with $d$ short diagonals of a given color and $d$ long diagonals intersecting in one vertex $\bar{\bar{v}}$ of the other color. The vertex is of valency $2d$. Two $NC$ structures are said to be of the same color if the intersecting chords of the upper bound signature are of the same color. In each $NC$ there exist two $M$-signatures. Those $NC$ structures are connected one to another by the $M$-signatures (see an example of one of $NC_{4}$ structures, figures~\ref{F:NC} and~\ref{F:NC12}). 

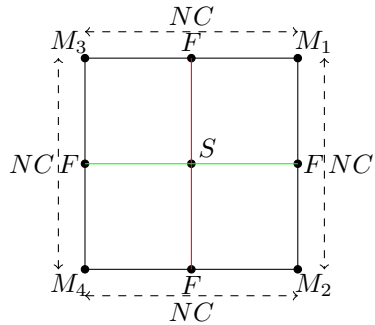
\begin{figure}[h]
\begin{center}
\begin{tikzpicture}[scale=0.7]
\draw[fill] (-2, 2) circle (.07cm); \node (-2, 2) at (-2.3,2.3) {$M_{3}$};
\draw[fill] (0, 2) circle (.07cm);  \node (0, 2) at (0,2.3) {$F$};
\draw[fill] (2, 2) circle (.07cm); \node(2,2) at (2.3,2.3) {$M_{1}$};
\draw[fill] (2, 0) circle (.07cm);  \node (2, 0) at (2.3,0) {$F$};
\draw[fill] (2, -2) circle (.07cm);\node (2,-2) at (2.3,-2.3) {$M_{2}$};
\draw[fill] (0, -2) circle (.07cm);  \node (0, -2) at (0,-2.3) {$F$};
\draw[fill] (-2, -2) circle (.07cm);  \node (-2, -2) at (-2.3,-2.3) {$M_{4}$};
\draw[fill] (-2, 0) circle (.07cm);  \node (-2, 0) at (-2.3,0) {$F$};
\draw[fill] (0, 0) circle (.07cm);  \node (0, 0) at (0.3,0.3) {$S$};
\draw (-2, 2) -- (2, 2) -- (2,-2) -- (-2,-2) -- (-2,2);
\draw[red] (0,2) -- (0,-2);
 \draw[green] (-2,0) -- (2,0);
\node  at (0,2.8) {$NC$};
\node  at (0,-2.8) {$NC$};
\node  at (-3,0) {$NC$};
\node  at (3,0) {$NC$};
 \draw[dashed,<->](-2,2.5) -- (2,2.5);
 \draw[dashed,<->](2.5,2) -- (2.5,-2);
 \draw[dashed,<->](-2,-2.5) -- (2,-2.5);
 \draw[dashed,<->](-2.5,-2) -- (-2.5,2);
\end{tikzpicture}
\vspace{-10pt}
\caption{Exterior structure: $NC_{d}$ are drawn as black lines between the $M$-signatures. Horizontal, vertical bridge structures}\label{F:NC}
\end{center}
\end{figure}

In the case of $d=4$, the connections between $M_{1}$ and $M_{3}$ are represented on figure~\ref{F:NC12}. The other connections in figure~\ref{F:NC} are of the same type.
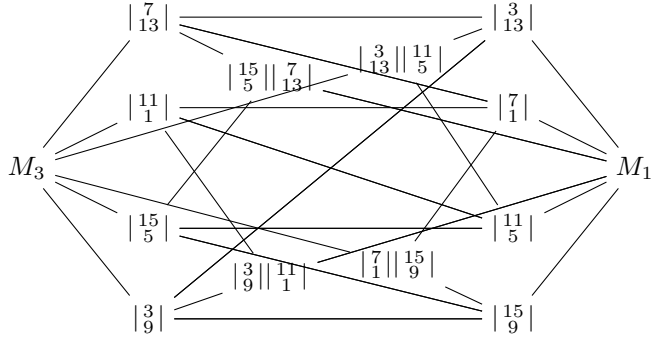
\begin{figure}[h]
\begin{center} 
     \begin{tikzpicture}[scale=0.4]
  \node  (a) at (10,0){$M_{1}$};
  \node  (b1) at (6,5) {$\scriptstyle\left|\begin{smallmatrix}3\\13\end{smallmatrix}\right|$};
\node  (b2) at (6,-5){$\scriptstyle\left|\begin{smallmatrix}15\\9\end{smallmatrix}\right|$}; 
\node  (b3) at (6,-2){$\scriptstyle\left|\begin{smallmatrix}11\\5\end{smallmatrix}\right|$}; 
\node  (b4) at (6,2){$\scriptstyle\left|\begin{smallmatrix}7\\1\end{smallmatrix}\right|$}; 
\node  (c1) at (2.3,3.6){$\scriptstyle\left|\begin{smallmatrix}3\\13\end{smallmatrix}\right|\left|\begin{smallmatrix}11\\5\end{smallmatrix}\right|$}; 
\node  (c2) at (2.2,-3.2){$\scriptstyle\left|\begin{smallmatrix}7\\1\end{smallmatrix}\right|\left|\begin{smallmatrix}15\\9\end{smallmatrix}\right|$}; 
\node  (d1) at (-2,3){$\scriptstyle\left|\begin{smallmatrix}15\\5\end{smallmatrix}\right|\left|\begin{smallmatrix}7\\13\end{smallmatrix}\right|$}; 
\node  (d2) at (-2,-3.6){$\scriptstyle\left|\begin{smallmatrix}3\\9\end{smallmatrix}\right|\left|\begin{smallmatrix}11\\1\end{smallmatrix}\right|$}; 
 \node  (e1) at (-6,5){$\scriptstyle\left|\begin{smallmatrix}7\\13\end{smallmatrix}\right|$}; 
\node  (e2) at (-6,2){$\scriptstyle\left|\begin{smallmatrix}11\\1\end{smallmatrix}\right|$}; 
\node  (e3) at (-6,-2){$\scriptstyle\left|\begin{smallmatrix}15\\5\end{smallmatrix}\right|$}; 
\node  (e4) at (-6,-5){$\scriptstyle\left|\begin{smallmatrix}3\\9\end{smallmatrix}\right|$}; 
\node  (f) at (-10,0){$M_{3}$}; 

\draw (a)--(b1);
\draw (a)--(b2);
\draw (a)--(b3);
\draw (a)--(b4);
\draw (a)--(d1);
\draw (a)--(d2);
\draw (b1)--(c1);
\draw (b1)--(e1);
\draw (b1)--(e4);
\draw (b2)--(c2);
\draw (b2)--(e3);
\draw (b2)--(e4);
\draw (b3)--(c1);
\draw (b3)--(e2);
\draw (b3)--(e3);
\draw (b4)--(c2);
\draw (b4)--(e1);
\draw (b4)--(e2);
\draw (c1)--(f);
\draw (c2)--(f);
\draw (d1)--(a);
\draw (d1)--(e1);
\draw (d1)--(e3);
\draw (d2)--(a);
\draw (d2)--(e2);
\draw (d2)--(e4);
\draw (e1)--(b4);
\draw (e1)--(f);
\draw (e2)--(b3);
\draw (e2)--(f);
\draw (e3)--(b2);
\draw (e3)--(b3);
\draw (e3)--(f);
\draw (e4)--(b1);
\draw (e4)--(b2);
\draw (e4)--(f);
 \end{tikzpicture}
    \end{center}
 \caption{One of $NC_{4}$ structures.}~\label{F:NC12}
\end{figure}
The edges in figure~\ref{F:NC12} correspond to signatures of codimension 1. The signatures of codimensions  3, 4 and 5 corresponding to the faces of higher dimensions are depicted below. 
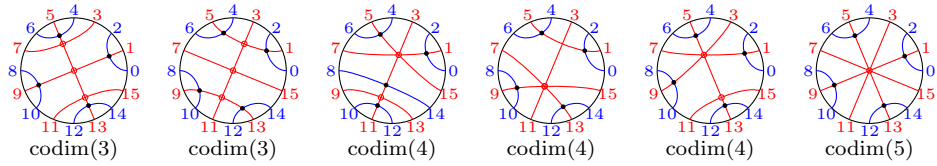
\begin{figure}
\begin{center}
\begin{tikzpicture}[scale=0.7]
\newcommand{\degree}[0]{8}
\newcommand{\last}[0]{7}
\newcommand{\B}[1]{#1*180/\degree}
\newcommand{\R}[1]{#1*180/\degree }
\newcommand{\RDOT}[0]{[  red](i-1) circle (0.05)}
\newcommand{\XDOT}[0]{[black](i-1) circle (0.045)}
\draw (0,-1.5) node {\footnotesize $\codim 3$} ;
\draw (0,0) circle (1) ;
\foreach \k in {0,2,4,6,8,10,12,14} {
  \draw[blue] (\B{\k}:1.15) node {\tiny \k};} ;
  \foreach \k in {1,3,5,7,9,11,13,15} {
 \draw[red] (\R{\k}:1.15) node {\tiny \k} ;
};
\draw[blue,  name path=B12B14](\B{12}:1) .. controls(\B{12}:0.7)and (\B{14}:0.7) .. (\B{14}:1) ;
\draw[blue,  name path=B8B10](\B{8}:1) .. controls(\B{8}:0.7)and (\B{10}:0.7) .. (\B{10}:1) ;
\draw[blue, name path=B4B6](\B{4}:1) .. controls(\B{4}:0.7)and (\B{6}:0.7) .. (\B{6}:1) ;
\draw[blue, name path=B0B2](\B{0}:1) .. controls(\B{0}:0.7)and (\B{2}:0.7) .. (\B{2}:1) ;
\draw[red, name path=R11R15](\R{11}:1) .. controls(\R{11}:0.7)and (\R{15}:0.7) .. (\R{15}:1) ;
\draw[red, name path=R5R13](\R{5}:1) .. controls(\R{5}:0.7)and (\R{13}:0.7) .. (\R{13}:1) ;
\draw[red, name path=R3R7](\R{3}:1) .. controls(\R{3}:0.7)and (\R{7}:0.7) .. (\R{7}:1) ;
\draw[red, name path=R1R9](\R{1}:1) .. controls(\R{1}:0.7)and (\R{9}:0.7) .. (\R{9}:1) ;
\draw[name intersections={of=R5R13 and R11R15,name=i}] \RDOT ;
\draw[name intersections={of=R5R13 and R1R9,name=i}] \RDOT ;
\draw[name intersections={of=R5R13 and R3R7,name=i}] \RDOT ;
\fill[name intersections={of=B12B14 and R5R13,name=i}] \XDOT ;
\fill[name intersections={of=B8B10 and R1R9,name=i}] \XDOT ;
\fill[name intersections={of=B4B6 and R5R13,name=i}] \XDOT ;
\fill[name intersections={of=B0B2 and R1R9,name=i}] \XDOT ;
\end{tikzpicture}
\begin{tikzpicture}[scale=0.7]
\newcommand{\degree}[0]{8}
\newcommand{\last}[0]{7}
\newcommand{\B}[1]{#1*180/\degree}
\newcommand{\R}[1]{#1*180/\degree }
\newcommand{\RDOT}[0]{[  red](i-1) circle (0.05)}
\newcommand{\XDOT}[0]{[black](i-1) circle (0.045)}
\draw (0,-1.5) node {\footnotesize $\codim 3$} ;
\draw (0,0) circle (1) ;
\foreach \k in {0,2,4,6,8,10,12,14} {
  \draw[blue] (\B{\k}:1.15) node {\tiny \k};} ;
  \foreach \k in {1,3,5,7,9,11,13,15} {
 \draw[red] (\R{\k}:1.15) node {\tiny \k} ;
};
\draw[blue,  name path=B12B14](\B{12}:1) .. controls(\B{12}:0.7)and (\B{14}:0.7) .. (\B{14}:1) ;
\draw[blue,  name path=B8B10](\B{8}:1) .. controls(\B{8}:0.7)and (\B{10}:0.7) .. (\B{10}:1) ;
\draw[blue,  name path=B4B6](\B{4}:1) .. controls(\B{4}:0.7)and (\B{6}:0.7) .. (\B{6}:1) ;
\draw[blue,  name path=B0B2](\B{0}:1) .. controls(\B{0}:0.7)and (\B{2}:0.7) .. (\B{2}:1) ;
\draw[red, name path=R9R13](\R{9}:1) .. controls(\R{9}:0.7)and (\R{13}:0.7) .. (\R{13}:1) ;
\draw[red, name path=R7R15](\R{7}:1) .. controls(\R{7}:0.7)and (\R{15}:0.7) .. (\R{15}:1) ;
\draw[red, name path=R3R11](\R{3}:1) .. controls(\R{3}:0.7)and (\R{11}:0.7) .. (\R{11}:1) ;
\draw[red, name path=R1R5](\R{1}:1) .. controls(\R{1}:0.7)and (\R{5}:0.7) .. (\R{5}:1) ;
\draw[name intersections={of=R9R13 and R3R11,name=i}] \RDOT ;
\draw[name intersections={of=R7R15 and R3R11,name=i}] \RDOT ;
\draw[name intersections={of=R3R11 and R1R5,name=i}] \RDOT ;
\fill[name intersections={of=B12B14 and R9R13,name=i}] \XDOT ;
\fill[name intersections={of=B8B10 and R9R13,name=i}] \XDOT ;
\fill[name intersections={of=B4B6 and R1R5,name=i}] \XDOT ;
\fill[name intersections={of=B0B2 and R1R5,name=i}] \XDOT ;
\end{tikzpicture}  
\begin{tikzpicture}[scale=0.7]
\newcommand{\degree}[0]{8}
\newcommand{\last}[0]{7}
\newcommand{\B}[1]{#1*180/\degree}
\newcommand{\R}[1]{#1*180/\degree }
\newcommand{\RDOT}[0]{[  red](i-1) circle (0.05)}
\newcommand{\XDOT}[0]{[black](i-1) circle (0.045)}
\draw (0,-1.5) node {\footnotesize $\codim 4$} ;
\draw (0,0) circle (1) ;
\foreach \k in {0,2,4,6,8,10,12,14} {
  \draw[blue] (\B{\k}:1.15) node {\tiny \k};} ;
  \foreach \k in {1,3,5,7,9,11,13,15} {
 \draw[red] (\R{\k}:1.15) node {\tiny \k} ;
};
\draw[blue,  name path=B10B12](\B{10}:1) .. controls(\B{10}:0.7)and (\B{12}:0.7) .. (\B{12}:1) ;
\draw[blue,  name path=B8B14](\B{8}:1) .. controls(\B{8}:0.7)and (\B{14}:0.7) .. (\B{14}:1) ;
\draw[blue,  name path=B4B6](\B{4}:1) .. controls(\B{4}:0.7)and (\B{6}:0.7) .. (\B{6}:1) ;
\draw[blue,  name path=B0B2](\B{0}:1) .. controls(\B{0}:0.7)and (\B{2}:0.7) .. (\B{2}:1) ;
\draw[red, name path=R9R13](\R{9}:1) .. controls(\R{9}:0.7)and (\R{13}:0.7) .. (\R{13}:1) ;
\draw[red, name path=R5R15](\R{5}:1) .. controls(\R{5}:0.7)and (\R{15}:0.7) .. (\R{15}:1) ;
\draw[red, name path=R3R11](\R{3}:1) .. controls(\R{3}:0.7)and (\R{11}:0.7) .. (\R{11}:1) ;
\draw[red, name path=R1R7](\R{1}:1) .. controls(\R{1}:0.7)and (\R{7}:0.7) .. (\R{7}:1) ;
\draw[name intersections={of=R3R11 and R9R13,name=i}] \RDOT ;
\draw[name intersections={of=R3R11 and R5R15,name=i}] \RDOT ;
\fill[name intersections={of=B10B12 and R3R11,name=i}] \XDOT ;
\fill[name intersections={of=B8B14 and R3R11,name=i}] \XDOT ;
\fill[name intersections={of=B4B6 and R5R15,name=i}] \XDOT ;
\fill[name intersections={of=B0B2 and R1R7,name=i}] \XDOT ;
\end{tikzpicture}
\begin{tikzpicture}[scale=0.7]
\newcommand{\degree}[0]{8}
\newcommand{\last}[0]{7}
\newcommand{\B}[1]{#1*180/\degree}
\newcommand{\R}[1]{#1*180/\degree }
\newcommand{\RDOT}[0]{[  red](i-1) circle (0.05)}
\newcommand{\XDOT}[0]{[black](i-1) circle (0.045)}
\draw (0,-1.5) node {\footnotesize $\codim 4$} ;
\draw (0,0) circle (1) ;
\foreach \k in {0,2,4,6,8,10,12,14} {
  \draw[blue] (\B{\k}:1.15) node {\tiny \k};} ;
  \foreach \k in {1,3,5,7,9,11,13,15} {
 \draw[red] (\R{\k}:1.15) node {\tiny \k} ;
};
\draw[blue,  name path=B12B14](\B{12}:1) .. controls(\B{12}:0.7)and (\B{14}:0.7) .. (\B{14}:1) ;
\draw[blue,  name path=B8B10](\B{8}:1) .. controls(\B{8}:0.7)and (\B{10}:0.7) .. (\B{10}:1) ;
\draw[blue,  name path=B4B6](\B{4}:1) .. controls(\B{4}:0.7)and (\B{6}:0.7) .. (\B{6}:1) ;
\draw[blue,  name path=B0B2](\B{0}:1) .. controls(\B{0}:0.7)and (\B{2}:0.7) .. (\B{2}:1) ;
\draw[red, name path=R9R15](\R{9}:1) .. controls(\R{9}:0.7)and (\R{15}:0.7) .. (\R{15}:1) ;
\draw[red, name path=R7R13](\R{7}:1) .. controls(\R{7}:0.7)and (\R{13}:0.7) .. (\R{13}:1) ;
\draw[red, name path=R3R11](\R{3}:1) .. controls(\R{3}:0.7)and (\R{11}:0.7) .. (\R{11}:1) ;
\draw[red, name path=R1R5](\R{1}:1) .. controls(\R{1}:0.7)and (\R{5}:0.7) .. (\R{5}:1) ;
\draw[name intersections={of=R3R11 and R9R15,name=i}] \RDOT ;
\draw[name intersections={of=R3R11 and R7R13,name=i}] \RDOT ;
\fill[name intersections={of=B12B14 and R7R13,name=i}] \XDOT ;
\fill[name intersections={of=B8B10 and R9R15,name=i}] \XDOT ;
\fill[name intersections={of=B4B6 and R1R5,name=i}] \XDOT ;
\fill[name intersections={of=B0B2 and R1R5,name=i}] \XDOT ;
\end{tikzpicture}
\begin{tikzpicture}[scale=0.7]
\newcommand{\degree}[0]{8}
\newcommand{\last}[0]{7}
\newcommand{\B}[1]{#1*180/\degree}
\newcommand{\R}[1]{#1*180/\degree }
\newcommand{\RDOT}[0]{[  red](i-1) circle (0.05)}
\newcommand{\XDOT}[0]{[black](i-1) circle (0.045)}
\draw (0,-1.5) node {\footnotesize $\codim 4$} ;
\draw (0,0) circle (1) ;
\foreach \k in {0,2,4,6,8,10,12,14} {
  \draw[blue] (\B{\k}:1.15) node {\tiny \k};} ;
  \foreach \k in {1,3,5,7,9,11,13,15} {
 \draw[red] (\R{\k}:1.15) node {\tiny \k} ;
};
\draw[blue,  name path=B12B14](\B{12}:1) .. controls(\B{12}:0.7)and (\B{14}:0.7) .. (\B{14}:1) ;
\draw[blue,  name path=B8B10](\B{8}:1) .. controls(\B{8}:0.7)and (\B{10}:0.7) .. (\B{10}:1) ;
\draw[blue,  name path=B4B6](\B{4}:1) .. controls(\B{4}:0.7)and (\B{6}:0.7) .. (\B{6}:1) ;
\draw[blue,  name path=B0B2](\B{0}:1) .. controls(\B{0}:0.7)and (\B{2}:0.7) .. (\B{2}:1) ;
\draw[red, name path=R1R7](\R{1}:1) .. controls(\R{1}:0.7)and (\R{7}:0.7) .. (\R{7}:1) ;
\draw[red, name path=R5R13](\R{5}:1) .. controls(\R{5}:0.7)and (\R{13}:0.7) .. (\R{13}:1) ;
\draw[red, name path=R3R9](\R{3}:1) .. controls(\R{3}:0.7)and (\R{9}:0.7) .. (\R{9}:1) ;
\draw[red, name path=R11R15](\R{11}:1) .. controls(\R{11}:0.7)and (\R{15}:0.7) .. (\R{15}:1) ;
\draw[name intersections={of=R5R13 and R11R15,name=i}] \RDOT ;
\draw[name intersections={of=R5R13 and R1R7,name=i}] \RDOT ;
\fill[name intersections={of=B12B14 and R5R13,name=i}] \XDOT ;
\fill[name intersections={of=B8B10 and R3R9,name=i}] \XDOT ;
\fill[name intersections={of=B4B6 and R5R13,name=i}] \XDOT ;
\fill[name intersections={of=B0B2 and R1R7,name=i}] \XDOT ;
\end{tikzpicture}
\begin{tikzpicture}[scale=0.7]
\newcommand{\degree}[0]{8}
\newcommand{\last}[0]{7}
\newcommand{\B}[1]{#1*180/\degree}
\newcommand{\R}[1]{#1*180/\degree }
\newcommand{\RDOT}[0]{[  red](i-1) circle (0.05)}
\newcommand{\XDOT}[0]{[black](i-1) circle (0.045)}
\draw (0,-1.5) node {\footnotesize $\codim 5$} ;
\draw (0,0) circle (1) ;
\foreach \k in {0,2,4,6,8,10,12,14} {
  \draw[blue] (\B{\k}:1.15) node {\tiny \k};} ;
  \foreach \k in {1,3,5,7,9,11,13,15} {
 \draw[red] (\R{\k}:1.15) node {\tiny \k} ;
};
\draw[blue,  name path=B12B14](\B{12}:1) .. controls(\B{12}:0.7)and (\B{14}:0.7) .. (\B{14}:1) ;
\draw[blue,  name path=B8B10](\B{8}:1) .. controls(\B{8}:0.7)and (\B{10}:0.7) .. (\B{10}:1) ;
\draw[blue,  name path=B4B6](\B{4}:1) .. controls(\B{4}:0.7)and (\B{6}:0.7) .. (\B{6}:1) ;
\draw[blue,  name path=B0B2](\B{0}:1) .. controls(\B{0}:0.7)and (\B{2}:0.7) .. (\B{2}:1) ;
\draw[red, name path=R1R9] (\R{1}:1) .. controls(\R{1}:0.7) and (\R{9}:0.7) .. (\R{9}:1) ;
\draw[red, name path=R3R11] (\R{3}:1) .. controls(\R{3}:0.7) and (\R{11}:0.7) .. (\R{11}:1) ;
\draw[red, name path=R5R13] (\R{5}:1) .. controls(\R{5}:0.7) and (\R{13}:0.7) .. (\R{13}:1) ;
\draw[red, name path=R7R15] (\R{7}:1) .. controls(\R{7}:0.7) and (\R{15}:0.7) .. (\R{15}:1) ;
\fill[name intersections={of=B0B2 and R1R9,name=i}] \XDOT ;
\fill[name intersections={of=B4B6 and R5R13,name=i}] \XDOT ;
\fill[name intersections={of=B8B10 and R1R9,name=i}] \XDOT ;
\fill[name intersections={of=B12B14 and R5R13,name=i}] \XDOT ;
\draw[name intersections={of=R3R11 and R7R15,name=i}] \RDOT ;
\end{tikzpicture}
\end{center}
\vspace{-10pt}
\caption{Diagrams of codimensions 3,4, and 5 between $M_{1}$ and $M_{3}$}\label{F:NC13}
\end{figure}

\vspace{1cm}
 
\item The \underline{ interior part} is divided into two structures: bridges $B$ and open-book $O$. 
\begin{itemize}
\item  A {\it bridge} $B$ is a linearly ordered subset with one upper bound. This upper bound is incident to a couple of opposite $F$-signatures, respectively lying in $NC$ structures of the same color. In figure~\ref{F:NC}, one vertical (resp. horizontal) bridge structures is drawn as a red (resp. green) line. 
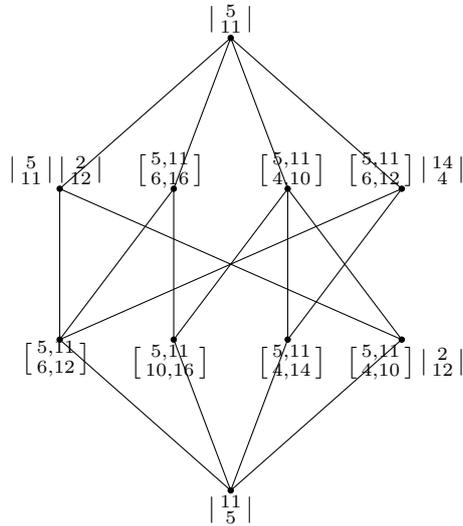
\begin{figure}[h]
\begin{center}
\begin{tikzpicture}[scale=0.5]
\draw[fill] (0, 6) circle (.07cm); \node (0, 6) at (0,6.5) {$\left|\begin{smallmatrix}5\\11\end{smallmatrix}\right|$};
 \draw[fill] (-4.5, 2) circle (.07cm);  \node (-4.5, 2) at (-4.6,2.5) {$\left|\begin{smallmatrix}5\\11\end{smallmatrix}\right|\left|\begin{smallmatrix}2\\12\end{smallmatrix}\right|$};
 \draw[fill] (-1.5, 2) circle (.07cm); \node(-1.5,2) at (-1.6,2.5) {$\left[\begin{smallmatrix}5,11\\6,16\end{smallmatrix}\right]$};
 \draw[fill] (1.5, 2) circle (.07cm);  \node (1.5, 2) at (1.6,2.5) {$\left[\begin{smallmatrix}5,11\\4,10\end{smallmatrix}\right]$};
 \draw[fill] (4.5, 2) circle (.07cm);\node (4.5,2) at (4.6,2.5) {$\left[\begin{smallmatrix}5,11\\6,12\end{smallmatrix}\right]\left|\begin{smallmatrix}14\\4\end{smallmatrix}\right|$};
 \draw[fill] (-4.5, -2) circle (.07cm);  \node (-4.5, -2) at (-4.6,-2.5) {$\left[\begin{smallmatrix}5,11\\6,12\end{smallmatrix}\right]$};
 \draw[fill] (-1.5, -2) circle (.07cm); \node(-1.5,-2) at (-1.6,-2.6) {$\left[\begin{smallmatrix}5,11\\10,16\end{smallmatrix}\right]$};
 \draw[fill] (1.5, -2) circle (.07cm);  \node (1.5, -2) at (1.6,-2.6) {$\left[\begin{smallmatrix}5,11\\4,14\end{smallmatrix}\right]$};
 \draw[fill] (4.5, -2) circle (.07cm);\node (4.5,-2) at (4.6,-2.6) {$\left[\begin{smallmatrix}5,11\\4,10\end{smallmatrix}\right]\left|\begin{smallmatrix}2\\12\end{smallmatrix}\right|$};
 \draw[fill] (0, -6) circle (.07cm);  \node (0, -6) at (0,-6.5) {$\left|\begin{smallmatrix}11\\5\end{smallmatrix}\right|$};
\draw (0, 6) -- (-4.5, 2) -- (-4.5,-2) -- (0,-6);
\draw (0, 6) -- (-1.5, 2) -- (-1.5,-2) -- (0,-6);
\draw (0, 6) -- (4.5, 2);
\draw  (-1.5, 2) -- (-4.5,-2);
\draw  (1.5, 2) -- (-1.5,-2);
\draw  (1.5, 2) -- (4.5,-2);
\draw  (-4.5, 2) -- (4.5,-2);
\draw  (4.5, 2) -- (-4.5,-2);
\draw  (4.5, -2) -- (0,-6);
\draw  (4.5, 2) -- (1.5,-2);
\draw (0, 6) -- (1.5, 2) -- (1.5,-2) -- (0,-6);
\end{tikzpicture}
\end{center}
\vspace{-10pt}
\caption{Example of bridge structures between two opposite $NC_{4}$ substructures}~\label{F:B4}
\end{figure}

\item An {\it open book} $O$ is  a linearly ordered subset having one upper bound incident to generic signatures lying in two adjacent $NC$ structures of the opposite color. This upper bound is a signature of codimension $2d-4$ having at least two intersection points of different colors. Notice that  $O$ substructures exist only for $d>3$.

\end{itemize}

\end{enumerate}

\begin{proposition}
For any $d>1$ there exist four structures $NC_{d}$ which are glued one to another by their $M$-signatures, so that each $M$-signature is incident to only two $NC_{d}$ structures of the opposite colors and a pair of $NC_{d}$ structures have at most  one $M$-signature in common.
\end{proposition}
\begin{proof}
We know that there exists four $NC_{d}$ structures in the inclusion diagram, where each $NC_{d}$ structure containing among its set of vertices a pair of $M$-signatures, for any $d>1$. The vertices along which the $NC_{d}$  are glued to each other are the  $M$-signatures. An $M$-signature is a vertex of valency  2$\binom{d}{2}$, since there exist $\binom{d}{2}$  possibilities to make one Whitehead move starting form an $M$-signature for the red (resp. blue) diagonals (theorem~\ref{T:Adj}).
So, this argument shows that an $M$-signature is the intersection of a pair of $NC_{d}$ structures of each color. Suppose, that this common $M$-signature is denoted by $M_1$. Note, that there still remain two $M$-signatures in this pair of $NC_{d}$ structures. We will show that the two remaining $M$-signatures are different. Indeed, in one $NC_{d}$ structure only the red diagonals were modified; in the other $NC_{d}$ structure only the blue diagonals were modified. Therefore, in the first $NC_{d}$ structure, the ending $M$-signature has different blue diagonals than in $M_1$ and in the second $NC_{d}$ structure the $M$-signature has different red diagonals than in $M _1$.
Now, since there exist four $M$-signatures, the four $NC_{d}$ structures are glued to each other by their $M$-signatures, and a pair of $NC_{d}$ structures have at most  one $M$-signature in common.  
\end{proof}

\begin{example}{Inclusion diagram for $d=3$}
The figure~\ref{F:graph3} illustrates the inclusion diagram for $d=3$, using the matricial notation. 
\begin{figure}[ht] 
\begin{center} 
     \begin{tikzpicture}[scale=0.75]
    \draw[fill] (-4, 4) circle (.07cm);  \node (-4,4) at (-4.3,4.3) {$\scriptstyle M_{3}$};
    \draw[fill] (4, 4) circle (.07cm);\node (4,4) at (4.3,4.3) {$\scriptstyle M_{1}$};
         \draw[fill] (-0.5, 5) circle (.07cm); ;\node (-0,5,5) at (-0.5,5.3) {$\scriptstyle \left|\begin{smallmatrix}7\\1\end{smallmatrix}\right|$};
      \draw[fill] (0, 4) circle (.07cm);\node (0,4) at (0,4.3) {$\scriptstyle\left|\begin{smallmatrix}3\\9\end{smallmatrix}\right|$};%
       \draw[fill] (0.5, 3) circle (.07cm) ;\node (0.5,3) at (0.5,3.3) {$\scriptstyle\left|\begin{smallmatrix}11\\5\end{smallmatrix}\right|$};
           \draw[fill] (-0.5, -3) circle (.07cm) ;\node (-0.5,-3) at (-0.5,-3.3) {$\scriptstyle\left|\begin{smallmatrix}5\\11\end{smallmatrix}\right|$};
              \draw[fill] (0, -4) circle (.07cm); \node (0,-4) at (0,-4.3) {$\scriptstyle\left|\begin{smallmatrix}9\\3\end{smallmatrix}\right|$};
                \draw[fill] (0.5, -5) circle (.07cm); \node (0.5,-5) at (0.5,-5.3) {$\scriptstyle\left|\begin{smallmatrix}1\\7\end{smallmatrix}\right|$};
                
    \draw[fill] (-4, -4) circle (.07cm); \node (-4,-4) at (-4.3,-4.3) {$\scriptstyle M_{4}$}; 
   \draw[fill] (4, -4) circle (.07cm); \node (4,-4) at (4.3,-4.3) {$\scriptstyle M_{2}$};
     \draw[fill] (-6, 0) circle (.07cm); \node (-6,0) at (-6.3,0) {$\scriptstyle\left|\begin{smallmatrix}8\\2\end{smallmatrix}\right|$};
         \draw[fill] (-4, 0) circle (.07cm);\node (-4,0) at (-4.4,0) {$\scriptstyle\left|\begin{smallmatrix}4\\10\end{smallmatrix}\right|$};%
             \draw[fill] (-2, 0) circle (.07cm);\node (-2,0) at (-2.4,0) {$\scriptstyle\left|\begin{smallmatrix}12\\6\end{smallmatrix}\right|$};
                 \draw[fill] (2, 0) circle (.07cm);\node (2,0) at (2.4,0) {$\scriptstyle\left|\begin{smallmatrix}6\\12\end{smallmatrix}\right|$};
                     \draw[fill] (4, 0) circle (.07cm);\node (4,0) at (4.4,0) {$\scriptstyle\left|\begin{smallmatrix}10\\4\end{smallmatrix}\right|$};
                         \draw[fill] (6, 0) circle (.07cm);\node (6,0) at (6.4,0) {$\scriptstyle\left|\begin{smallmatrix}2\\8\end{smallmatrix}\right|$};
                         
           \draw[fill] (-0.5, 0.5) circle (.07cm); \node (-0.5,0.5) at (-0,0.4) {$\scriptstyle\left[\begin{smallmatrix}6,12\\5,11\end{smallmatrix}\right]$};
            \draw[fill] (-1, 1) circle (.07cm);  \node (-1,1) at (-0.5,1) {$\scriptstyle\left[\begin{smallmatrix}3,9\\4,10\end{smallmatrix}\right]$};               
        \draw[fill] (-1.5, 1.5) circle (.07cm); \node (-1.5,1.5) at (-1.1,1.6) {$\scriptstyle\left[\begin{smallmatrix}1,7\\2,8\end{smallmatrix}\right]$};
        \draw[fill] (1.5,- 1.5) circle (.07cm); \node (1.5,-1.5) at (1,-1.6) {$\scriptstyle\left[\begin{smallmatrix}4,10\\5,11\end{smallmatrix}\right]$};
            \draw[fill] (1, -1) circle (.07cm);   \node (1,-1) at (0.6,-1) {$\scriptstyle\left[\begin{smallmatrix}2,8\\3,9\end{smallmatrix}\right]$};              
        \draw[fill] (0.5, -0.5) circle (.07cm); \node (0.5,-0.5) at (0,-0.4){$\scriptstyle\left[\begin{smallmatrix}12,6\\1,7\end{smallmatrix}\right]$};

 \draw[red,dashed] (-0.5,-3) -- (-0.5,0.5);
   \draw[red,dashed] (0.5,3) -- (-0.5,0.5);
      \draw[red] (0.5,3) -- (1.5,-1.5);
       \draw[red] (1.5,-1.5) --  (-0.5,-3);
         \draw[red,dashed] (-0.5,0.5) -- (-2,0);
           \draw[red] (-2,0) -- (0.5,-0.5);
             \draw[red] (0.5,-0.5) -- (2,0);
               \draw[red,dashed] (2,0) -- (-0.5,0.5);
           
   \draw[green,dashed] (-1,1) -- (0,-4);
     \draw[green,dashed] (-1,1) -- (-4,0);
       \draw[green,dashed] (-1,1) -- (4,0);
         \draw[green,dashed] (-1,1) -- (0,4);
           \draw[green] (-4,0) -- (1.5,-1.5);
             \draw[green] (4,0) -- (1.5,-1.5);
              \draw[green] (0,4) -- (1,-1);
                \draw[green] (1,-1) -- (0,-4);

 \draw[blue,dashed] (-1.5,1.5) -- (-0.5,5);
   \draw[blue,dashed] (-1.5,1.5) -- (-6,0);
     \draw[blue] (-6,-0) -- (1,-1);
       \draw[blue] (1,-1) -- (6,0);
         \draw[blue,dashed] (6,0) -- (-1.5,1.5);
           \draw[blue ](-0.5,5) -- (0.5,-0.5);
         \draw[blue] (0.5,-0.5) -- (0.5,-5);         
             \draw[blue,dashed ]  (0.5,-5)--(-1.5,1.5);

 \draw (-4,4) -- (-0.5,5);
               \draw (-4,4) -- (0,4);    
                   \draw (-4,4) -- (0.5,3);
                    \draw (-4,4) -- (-6,0);
               \draw (-4,4) -- (-4,0);    
                   \draw (-4,4) -- (-2,0);
                   \draw (4,4) -- (-0.5,5);
               \draw (4,4) -- (0,4);    
                   \draw (4,4) -- (0.5,3);
                    \draw (4,4) -- (6,0);
               \draw (4,4) -- (4,0);    
                   \draw (4,4) -- (2,0);
                    \draw (-4,-4) -- (-6,0);
               \draw (-4,-4) -- (-4,0);    
                   \draw (-4,-4) -- (-2,0);
                   \draw (-4,-4) -- (-0.5,-3);
               \draw (-4,-4) -- (0,-4);    
                   \draw (-4,-4) -- (0.5,-5);
                    \draw (4,-4) -- (-0.5,-3);
               \draw (4,-4) -- (0,-4);    
                   \draw (4,-4) -- (0.5,-5);
                    \draw (4,-4) -- (2,0);
               \draw (4,-4) -- (4,0);    
                   \draw (4,-4) -- (6,0);
        
                  \end{tikzpicture}
    \end{center}
\caption{Inclusion diagram for $d=3$}\label{F:graph3} 
\end{figure}
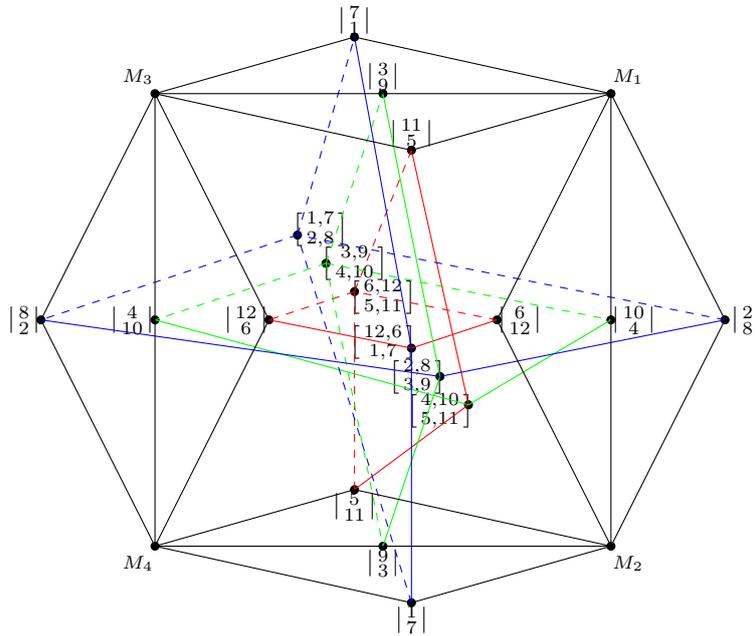

For any The inclusion diagram contains two distinct parts:
\begin{enumerate}
\item The four substructures in black are the $NC_{3}$ substructures.They connect pairs of $M$-signatures, having the same short diagonals of a given color.
Except from $M$-signatures, these black substructures contain three vertices corresponding to $F$-signatures. The 3-face of the black substructure corresponds to the  codimension 3 signature  and it is incident to three 2-faces which correspond to the  codimension 2 signatures. Those signatures have two inner vertices, incident to four edges of the same color and three inner vertices which are incident to four edges of alternating colors. \\

\item The substructure with colored edges. This colored part of the inclusion diagram corresponds to the parts which appear in the construction given in figure~\ref{F:Q3}. 
In particular, the vertices in this colored part are the $S$-signatures. The 2-faces in the interior part correspond to codimension 2 signatures. Those signatures of codimension 2, in addition to the three inner nodes with incident edges of alternating color, have one inner vertex incident to 4 blue edges and the other one incident to 4 red edges.
\end{enumerate}

\vspace{5pt}
This inclusion diagram is resumed by the construction in Appendix~\ref{A: Incldiag3}: Figure~\ref{F:2Fincldiag} and Figure~\ref{F: Incldiag3}. 

 \begin{itemize}
\item The first quadrangle 2-face connecting the following two $F$-signatures and two $S$-signatures $\left|\begin{smallmatrix}2\\8\end{smallmatrix}\right|$,$\left|\begin{smallmatrix}8\\2\end{smallmatrix}\right|$, $\left[\begin{smallmatrix}1,7\\2,8\end{smallmatrix}\right]$ and$\left[\begin{smallmatrix}2,8\\3,9\end{smallmatrix}\right]$ corresponds in figure~\ref{F:graph3} to the blue vertical cycle. The $F$-signatures have one long red diagonal. \\

\item The second quadrangle 2-face connecting the following two $F$-signatures and two $S$-signatures  $\left|\begin{smallmatrix}9\\3\end{smallmatrix}\right|$,  $\left|\begin{smallmatrix}3\\9\end{smallmatrix}\right|$, $\left[\begin{smallmatrix}2,8\\3,9\end{smallmatrix}\right]$, $\left[\begin{smallmatrix}3,9\\4,10\end{smallmatrix}\right]$ corresponds in figure~\ref{F:graph3} to the blue horizontal cycle. The $F$-signatures have two long blue diagonals. 
\end{itemize}
\end{example}

\subsection{$Q$-diagrams and $Q$-pieces}
In the following, for $d\geq 3$ we introduce the notions of $Q$-diagrams and $Q$-pieces.
\begin{definition}\label{D:Q}

\begin{enumerate}
\item  A {\bf $Q$-diagram} is a signature in $\Sigma_d$, having $(d-2)$ $S$-trees. This $Q$-diagram is a superimposition of two monochromatic diagrams (one blue, one red) both having $d-2$ long diagonals; the red (resp.blue) diagram is turned through $\frac{\pi}{4d}$ relatively to the center of the diagram and to the blue (resp. red) diagram. 
By $\mathcal{L}$ we denote the reflection axis of the blue (resp. red) diagram. It is parallel to the long diagonals.

\item A {\bf $Q$-piece} is the union of elementa $A_{\sigma}$, indexed by those signatures which are adjacent to a $Q$-diagram, and having at least one long diagonal parallel to the ones of the $Q$-diagram. To this set of signatures, we add those $M$-signatures, being adjacent to them, in a minimal number of Whitehead moves.

\item  A pair of {\bf adjacent $Q$-pieces} denoted $(Q_{i},Q_{i+1})$ verify the following properties:

\begin{enumerate} 

\item \underline{(Copy and Paste)}: There exists an identity map from the set of  blue (resp. red) monochromatic diagrams in the $Q_{i}$-piece to the set of blue (resp. red) monochromatic diagrams in the $Q_{i+1}$-piece. 

\item \underline{(Copy/Paste and Rotate)}: There exists an identity map from the set of red (resp. blue) monochromatic diagrams in the $Q_{i}$-piece to the set of red (resp. blue) monochromatic diagrams in the $Q_{i+1}$-piece, composed with a rotation of $-\frac{\pi}{2d}$ about the center of the polygon. 

 \item One signature in a $Q_i$-piece is bijectively mapped to another one in the $Q_{i+1}$-piece, if their blue (resp. red) monochromatic diagrams are identical and if their red (resp. blue) monochromatic diagrams are symmetric to each other, about the reflection axis $\mathcal{L}$ of the blue (resp. red) monochromatic $Q$-diagram. 
Any pair of such signatures are said to be {\bf consecutive}. 
 \end{enumerate}
\item A {\bf Connection piece} is the union of elementa which glue a pair of adjacent $Q$-pieces together.
\end{enumerate}
\end{definition}

\vskip.1cm 

A $Q$-diagram is associated to a matrix $[\begin{smallmatrix}L_{0} \\L_{1} \end{smallmatrix}]$ where $L_{0} $ (resp. $L_{1}$) describes the monochromatic blue (resp. red) diagram in terms of pairs of indices of terminal vertices. A pairing of integers $i, j$, denoted by $(i,j)$, corresponds, in the signature, to a diagonal connecting the terminal vertices labeled respectively by $i$ and $j$; The integers in $L_{0}$ and $L_{1}$ belong respectively to the sets $\{0,2....,4d-2\}$ and $\{1,3,...,4d-1\}$. 

\begin{definition}
The matrix of the $Q$-diagram satisfies the following conditions:
   
- The matrix is of size $(d-2)\times 2$, where each column of the matrix contains a pairing of integers of the same parity;

- in one column: there is a couple of pairs of integers, which correspond to intersecting diagonals of opposite colors; 

- the first and the last columns of the matrix contains the pair of numbers on the terminal vertices of the shortest long diagonals;

-  the paired integers, lying in adjacent columns, correspond to a pair of adjoining diagonals (i.e. lying in the boundary of a 2-cell in $\mathbb{D}\setminus \sigma$). 
\end{definition}
\begin{lemma}~\label{L:rotateM}
Consider the $2d$ regular polygon formed from the blue (resp. red) terminal vertices in a signature.  Let $r$ be the rotation through $\frac{\pi}{2d}$  about the center of the polygon. Let $(i,j)$  be a diagonal of the polygon. So,  $r^{k}$ : $(i,j)\mapsto (i+2k,j+2k)\mod 4d$. \end{lemma}
\begin{proof}
Le us proceed by induction on the angle $\frac{k\pi}{2d}$.
\begin{enumerate}
\item {\sl Base case.} Let $k=1$. Let us rotate by $r=\frac{\pi}{2d}$ the diagonal $(i,j)$. Since $i,j\in\{1,3,..,4d-1\}$ (resp.\{2,4,..,4d\} ) are the vertices of a regular $2d$-gon, then the rotation maps the diagonal $(i,j)$ to the diagonal $(i+2,j+2)$. 
\item {\it Induction case.} Suppose that for a given $k$ in $\{1,...,2d\}$ the statement is true: the diagonal $(i,j)$ rotated by an angle $r^{k}=\frac{k\pi}{2d}$ is mapped onto the diagonal $(i+2k,j+2k)$. Let us show that for $k+1$ the statement is true. We rotate about an angle $\frac{(k+1)\pi}{2d}$ the diagonal $(i,j)$. By induction hypothesis rotating by $\frac{k\pi}{2d}$ maps $(i,j)$ onto $(i+2k,j+2k)$. Rotating the diagonal $(i+2k,j+2k)$ by an angle of $\frac{\pi}{2d}$ maps it by (1)  onto the diagonal $(i+2k+2,j+2k+2)$ {\it i.e.} $(i+2(k+1),j+2(k+1))$.\end{enumerate}
\end{proof}

\begin{example}{Adjacent $Q$-pieces of the inclusion diagram for $d=4$}
In this example, we detail the construction by induction of the inclusion diagram for $d=4$ .
\begin{itemize}
\item There are eight $Q$-pieces. Each $Q$-piece satisfies the relations in figure~\ref{F:2Qp}.
\item Each pair of adjacent $Q$-pieces are connected by a connection piece. 

\item Two consecutive $Q$-diagrams are related by the following commutative diagram:

\begin{center} 
     \begin{tikzpicture}[scale=0.6]
  \node(a) at (0,2){$\scriptstyle\left[\begin{smallmatrix}i,i+6\\i+15,i+5\end{smallmatrix}\right]\left[\begin{smallmatrix}i+14,i+8\\i+13,i+7\end{smallmatrix}\right]$};
 
  \node(b1) at (-6,1) {$\scriptstyle\left|\begin{smallmatrix}i+14\\i+8\end{smallmatrix}\right|\scriptstyle\left|\begin{smallmatrix}i+6\\i\end{smallmatrix}\right|$};

\node(b2) at (6,1){$\scriptstyle\left|\begin{smallmatrix}i+8\\i+14\end{smallmatrix}\right|\left[\begin{smallmatrix}i,i+6\\i+5,i+5\end{smallmatrix}\right]$}; 

\node(c1) at (-6,-1){$\scriptstyle\left|\begin{smallmatrix}i+4\\i+8\end{smallmatrix}\right|\left[\begin{smallmatrix}i+1,i+7\\i,i+6\end{smallmatrix}\right]$}; 

\node(c2) at (6,-1){$\scriptstyle\left|\begin{smallmatrix}i\\i+6\end{smallmatrix}\right|\left|\begin{smallmatrix}i+8\\i+14\end{smallmatrix}\right|$}; 

\node(d) at (0,-2){$\scriptstyle\left|\begin{smallmatrix}i,i+6\\i+1,i+7\end{smallmatrix}\right|\left|\begin{smallmatrix}i+8,1+14\\i+9,i+15\end{smallmatrix}\right|$};
\draw (a)--(b1);
\draw (a)--(b2);
\draw (b1)--(c1);
\draw (b2)--(c2);
\draw (d)--(c1);
\draw (d)--(c2);
 \end{tikzpicture}
    \end{center}
\end{itemize}


In the following paragraph, the explicit construction of the first two $Q$-pieces is done. An illustration of this construction is in Figure~\ref{F:2Qp}.
\begin{enumerate} 
\item Let us start with the first $Q$-piece, having the $Q$ diagram $\left[\begin{smallmatrix}1,11&3,9 \\0,10 &2,8 \end{smallmatrix}\right]$. 
\vspace{3pt}
\begin{enumerate}
\item Apply a Whitehead-move to the pair of blue diagonals $(2,8),(4,6)$, where $(2,8)$ belongs to the $Q$-diagram and $(4,6)$ is a short adjacent diagonal,  in order to obtain $\left[\begin{smallmatrix}1,11 \\0,10 \end{smallmatrix}\right]\left|\begin{smallmatrix}9\\3\end{smallmatrix}\right|$. So, one obtains an $FS$-signature. Let us deform the long red diagonal in the $F$-tree $\left|\begin{smallmatrix}9\\3\end{smallmatrix}\right|$ with the short red diagonal in the $M$ tree $\left|\begin{smallmatrix}5\\7\end{smallmatrix}\right|$: it gives an $M$ tree. So, there remains only one $S$ diagram: $\left[\begin{smallmatrix}1,11 \\0,10 \end{smallmatrix}\right]$.
The $F$-signatures, obtained from this $S$-signature by a minimal number of Whitehead moves are: $\left|\begin{smallmatrix}7\\1\end{smallmatrix}\right|,\left|\begin{smallmatrix}0\\10\end{smallmatrix}\right|, \left|\begin{smallmatrix}10\\0\end{smallmatrix}\right|,\left|\begin{smallmatrix}11\\1\end{smallmatrix}\right|,\left|\begin{smallmatrix}1\\11\end{smallmatrix}\right|,\left|\begin{smallmatrix}4\\10\end{smallmatrix}\right|$. 
\vspace{3pt}
\item Deform the pair of blue diagonals $(0,10)$,$(12,14)$: this gives an $SF$-signature $\left|\begin{smallmatrix}11\\1\end{smallmatrix}\right|\left[\begin{smallmatrix}3,9 \\2,8 \end{smallmatrix}\right]$. Deform the long red diagonal in the $F$-tree $\left|\begin{smallmatrix}11\\1\end{smallmatrix}\right|$ with the short red one, in the $M$ tree $\left|\begin{smallmatrix}13\\15\end{smallmatrix}\right|$.
This gives also an $M$ tree and so, there remains one $S$ diagram: $\left[\begin{smallmatrix}3,9 \\2,8 \end{smallmatrix}\right]$. The $F$-signatures which are obtained from this $S$ diagram by a minimal number of deformation operations are $\left|\begin{smallmatrix}15\\9\end{smallmatrix}\right|,\left|\begin{smallmatrix}2\\8\end{smallmatrix}\right|, \left|\begin{smallmatrix}8\\2\end{smallmatrix}\right|,\left|\begin{smallmatrix}12\\2\end{smallmatrix}\right|,\left|\begin{smallmatrix}3\\9\end{smallmatrix}\right|,\left|\begin{smallmatrix}9\\3\end{smallmatrix}\right|$. 
\end{enumerate}

\vspace{5pt}
\item The generic signatures of an adjacent $Q$-piece to the previous one are described below. This construction is done in two steps. 
\vspace{3pt}
 \begin{enumerate}
\item The adjacent $Q$-piece contains the following $Q$-diagram $\left[\begin{smallmatrix}0,10 &2,8 \\ 15,9&1,7 \end{smallmatrix}\right]$ and the $S$-signature $\left[\begin{smallmatrix}1,7\\2,8\end{smallmatrix}\right]$, which is obtained in one Whitehead moves, vie two possible ways. The first possibility is to deform a pair of blue diagonals giving the $FS$-signature $\left|\begin{smallmatrix}10\\0\end{smallmatrix}\right|\left[\begin{smallmatrix}2,8 \\1,7 \end{smallmatrix}\right]$. The second possibility is to deform the pair of red diagonals giving an $FS$-signature $\left|\begin{smallmatrix}15\\9\end{smallmatrix}\right|\left[\begin{smallmatrix}2,8 \\1,7 \end{smallmatrix}\right]$. 
The $S$ signature $\left[\begin{smallmatrix}2,8 \\1,7 \end{smallmatrix}\right]$ is adjacent after one Whitehead move to the following $F$-signatures: 
$\left|\begin{smallmatrix}15\\9\end{smallmatrix}\right|$, $\left|\begin{smallmatrix}9\\15\end{smallmatrix}\right|$, $\left|\begin{smallmatrix}11\\1\end{smallmatrix}\right|$, $\left|\begin{smallmatrix}14\\8\end{smallmatrix}\right|$, $\left|\begin{smallmatrix}2\\8\end{smallmatrix}\right|$, $\left|\begin{smallmatrix}8\\2\end{smallmatrix}\right|$. 

\vspace{3pt}
\item The $Q$-diagram $\left[\begin{smallmatrix}0,10 &2,8 \\ 15,9&1,7 \end{smallmatrix}\right]$ is also adjacent to the $S$-signature $\left[\begin{smallmatrix}15,9\\0,10\end{smallmatrix}\right]$. This is obtained by deforming a pair of red diagonals giving the $FS$-signature $\left|\begin{smallmatrix}10\\0\end{smallmatrix}\right|\left[\begin{smallmatrix}15,9 \\0,10 \end{smallmatrix}\right]$, or a pair of blue diagonals giving an $FS$-signature $\left|\begin{smallmatrix}8\\2\end{smallmatrix}\right|\left[\begin{smallmatrix}2,8 \\1,7 \end{smallmatrix}\right]$ . 
The $S$ signature $\left[\begin{smallmatrix}15,9 \\0,10 \end{smallmatrix}\right]$ is adjacent after one Whitehead move to the following $F$-signatures 
$\left|\begin{smallmatrix}7\\1\end{smallmatrix}\right|,\left|\begin{smallmatrix}1\\7\end{smallmatrix}\right|, \left|\begin{smallmatrix}6\\0\end{smallmatrix}\right|,\left|\begin{smallmatrix}3\\9\end{smallmatrix}\right|,$ $ \left|\begin{smallmatrix}0\\10\end{smallmatrix}\right|$, $\left|\begin{smallmatrix}10\\0\end{smallmatrix}\right|$. 
\end{enumerate}
\end{enumerate}

\begin{figure}[h]
\begin{center}
 \begin{tikzpicture}[scale=.36]
 \node (a2u) at (-6,23) {$\left|\begin{smallmatrix}1\\11\end{smallmatrix}\right|$};
    \node(a4u) at (0,23) {$\left|\begin{smallmatrix}11\\1\end{smallmatrix}\right|\left|\begin{smallmatrix}3\\9\end{smallmatrix}\right|$};
       \node(a6u) at (6,23) {$\left|\begin{smallmatrix}9\\3\end{smallmatrix}\right|$};
  \node (a2d) at (-6,21) {$\left|\begin{smallmatrix}11\\1\end{smallmatrix}\right|$};
    \node(a4d) at (0,21) {$\left|\begin{smallmatrix}1\\11\end{smallmatrix}\right|\left|\begin{smallmatrix}9\\3\end{smallmatrix}\right|$};
       \node(a6d) at (6,21) {$\left|\begin{smallmatrix}3\\9\end{smallmatrix}\right|$};

\node (b0) at (-13,18.5){$ M_{4}$}; 
 \node (b1) at (-8.7,18.5) {$\left|\begin{smallmatrix}4\\10\end{smallmatrix}\right|\left|\begin{smallmatrix}1\\11\end{smallmatrix}\right|$};
\node (b1d) at (-11,16.5) {$\left|\begin{smallmatrix}4\\10\end{smallmatrix}\right|$}; 
\node (b3) at (-3.2,18.5) {$\left|\begin{smallmatrix}3\\9\end{smallmatrix}\right|\left[\begin{smallmatrix}1,11\\10,0\end{smallmatrix}\right]$};
 \node (b5) at (3.2,18.5) {$\left|\begin{smallmatrix}11\\1\end{smallmatrix}\right|\left[\begin{smallmatrix}3,9\\2,8\end{smallmatrix}\right]$};
  \node (b7) at (8.7,18.5) {$\left|\begin{smallmatrix}12\\2\end{smallmatrix}\right|\left|\begin{smallmatrix}9\\3\end{smallmatrix}\right|$};
  \node (b7d) at (11,16.5) {$\left|\begin{smallmatrix}12\\2\end{smallmatrix}\right|$};
\node (b8) at (12.7,18.5){$ M_{4}$};
 
  \node (c2) at (-6,16) {$\left[\begin{smallmatrix}1,11\\0,10\end{smallmatrix}\right]$}; 
  \node (c4) at (0,16) {$\left[\begin{smallmatrix}1,11&3,9\\0,10&2,8\end{smallmatrix}\right]$};   
    \node (c6) at (6,16) {$\left[\begin{smallmatrix}3,9\\2,8\end{smallmatrix}\right]$};
\node (d0) at (-12.7,12.7){$ M_{1}$};
 \node (d1) at (-8.5,13.5) {$\left|\begin{smallmatrix}7\\1\end{smallmatrix}\right|\left|\begin{smallmatrix}10\\0\end{smallmatrix}\right|$};
\node (d1u) at (-11.2,15) {$\left|\begin{smallmatrix}7\\1\end{smallmatrix}\right|$};
 \node (d1d) at (-6,9.5) {$\left|\begin{smallmatrix}10\\0\end{smallmatrix}\right|$};
 \node (d3) at (-3.2,13.5) {$\left|\begin{smallmatrix}8\\2\end{smallmatrix}\right|\left[\begin{smallmatrix}1,11\\10,0\end{smallmatrix}\right]$}; 
   \node (d5) at (3.2,13.5) {$\left|\begin{smallmatrix}0\\10\end{smallmatrix}\right|\left[\begin{smallmatrix}3,9\\2,8\end{smallmatrix}\right]$};
   \node (d7) at (9,13.5) {$\left|\begin{smallmatrix}2\\8\end{smallmatrix}\right|\left|\begin{smallmatrix}15\\9\end{smallmatrix}\right|$};
    \node (d7u) at (11.2,15.1) {$\left|\begin{smallmatrix}15\\9\end{smallmatrix}\right|$};
    \node (d7d) at (6,9.5) {$\left|\begin{smallmatrix}2\\8\end{smallmatrix}\right|$};
     \node (d8) at (12.7,12.5){$ M_{1}$}; 
   
 \node (e2u) at (-6,12) {$\left|\begin{smallmatrix}0\\10\end{smallmatrix}\right|$};
 \node (e2d) at (-6,9.5) {$\left|\begin{smallmatrix}10\\0\end{smallmatrix}\right|$};
 \node (e4u) at (0,12) {$\left|\begin{smallmatrix}2\\8\end{smallmatrix}\right|\left|\begin{smallmatrix}10\\0\end{smallmatrix}\right]$}; 
  \node (e4d) at (0,9.4) {$\left|\begin{smallmatrix}8\\2\end{smallmatrix}\right|\left|\begin{smallmatrix}0\\10\end{smallmatrix}\right]$}; 
  \node (e6u) at (6,12) {$\left|\begin{smallmatrix}8\\2\end{smallmatrix}\right|$};
    \node (e6d) at (6,9.5) {$\left|\begin{smallmatrix}2\\8\end{smallmatrix}\right|$};
   \node (f1u) at (-6,12) {$\left|\begin{smallmatrix}0\\10\end{smallmatrix}\right|$};
 \node (f1) at (-8.7,7.5) {$\left|\begin{smallmatrix}0\\10\end{smallmatrix}\right|\left|\begin{smallmatrix}3\\9\end{smallmatrix}\right|$};
    \node (f0) at (-12.7,8.5){$ M_{3}$};
     \node (f1d) at (-11,6) {$\left|\begin{smallmatrix}3\\9\end{smallmatrix}\right|$}; 
  \node (f3) at (-3.2,7.5) {$\left|\begin{smallmatrix}2\\8\end{smallmatrix}\right|\left[\begin{smallmatrix}0,10\\9,15\end{smallmatrix}\right]$}; 
   \node (f5) at (3.2,7.5) {$\left|\begin{smallmatrix}10\\0\end{smallmatrix}\right|\left[\begin{smallmatrix}2,8\\1,7\end{smallmatrix}\right]$};
  \node (f7) at (8.7,7.5) {$\left|\begin{smallmatrix}8\\2 \end{smallmatrix}\right|\left|\begin{smallmatrix}11\\1\end{smallmatrix}\right|$};
 \node (f7d) at (11,6) {$\left|\begin{smallmatrix}11\\1\end{smallmatrix}\right|$};
   \node (f7u) at (6,12) {$\left|\begin{smallmatrix}8\\2\end{smallmatrix}\right|$};
  \node (f8) at (12.7,8.5){$ M_{3}$};

\node (g2) at (-6,5) {$\left[\begin{smallmatrix}15,9\\0,10\end{smallmatrix}\right]$}; 
  \node (g4) at (0,5) {$\left[\begin{smallmatrix}0,10&2,8\\9,15&1,7\end{smallmatrix}\right]$};   
    \node (g6) at (6,5) {$\left[\begin{smallmatrix}2,8\\1,7\end{smallmatrix}\right]$};

    \node (h1u) at (-10.9,4) {$\left|\begin{smallmatrix}6\\0\end{smallmatrix}\right|$};
   \node (h1) at (-8.7,2.5) {$\left|\begin{smallmatrix}6\\0\end{smallmatrix}\right|\left|\begin{smallmatrix}9\\15\end{smallmatrix}\right|$};
    \node (h0) at (-12.7,2.5){$ M_{2}$};
     \node (h1d) at (-6,-2) {$\left|\begin{smallmatrix}9\\15\end{smallmatrix}\right|$};
   \node (h3) at (-3.2,2.5) {$\left|\begin{smallmatrix}7\\1\end{smallmatrix}\right|\left[\begin{smallmatrix}0,10\\9,15\end{smallmatrix}\right]$};
    \node (h5) at (3.2,2.5) {$\left|\begin{smallmatrix}15\\9\end{smallmatrix}\right|\left[\begin{smallmatrix}2,8\\1,7\end{smallmatrix}\right]$};
   \node (h7) at (8.7,2.5) {$\left|\begin{smallmatrix}14\\8\end{smallmatrix}\right|\left|\begin{smallmatrix}1\\7\end{smallmatrix}\right|$};
    \node (h7u) at (11,4) {$\left|\begin{smallmatrix}14\\8\end{smallmatrix}\right|$};
      \node (h7d) at (6,-2) {$\left|\begin{smallmatrix}1\\7\end{smallmatrix}\right|$};
     \node (h8) at (12.7,2.5){$ M_{2}$};

\node (i2u) at (-6,0.9) {$\left|\begin{smallmatrix}15\\9\end{smallmatrix}\right|$}; 
 \node (i4u) at (0,0.9) {$\left|\begin{smallmatrix}1\\7\end{smallmatrix}\right|\left|\begin{smallmatrix}9\\15\end{smallmatrix}\right|$};   
 \node (i6u) at (6,0.9) {$\left|\begin{smallmatrix}7\\1\end{smallmatrix}\right|$};
  \node (i2d) at (-6,-2) {$\left|\begin{smallmatrix}9\\15\end{smallmatrix}\right|$}; 
 \node (i4d) at (0,-2) {$\left|\begin{smallmatrix}7\\1\end{smallmatrix}\right|\left|\begin{smallmatrix}15\\9\end{smallmatrix}\right|$};   
 \node (i6d) at (6,-2) {$\left|\begin{smallmatrix}1\\7\end{smallmatrix}\right|$};
\node (j0) at (-12.7,-4){$ M_{4}$}; 
\node (j8) at (12.7,-3.5){$ M_{4}$};
    \node (l2u) at (-6,-6) {$\left|\begin{smallmatrix}8\\14\end{smallmatrix}\right|$}; 
     \node (l6u) at (6,-6) {$\left|\begin{smallmatrix}0\\6\end{smallmatrix}\right|$}; 
\node(x8) at(12.7, 22.5){$M_{2}$}; 
 
   \draw[black] (b0) -- (a2u);
    \draw[black] (b1) -- (a2u);
  \draw[black] (a4d) -- (a2u);
  \draw[black,thick] (b0) -- (b1d); 
  \draw[black,thick] (b0) -- (a2u);
 \draw[black,thick] (b1) -- (b1d);
  \draw[black,thick] (b1) -- (a2u); 
     \draw[black,thick] (b8) -- (a6u);
     \draw[black,thick] (b7) -- (a6u); 
\draw[black,thick] (b8) -- (b7d); 
\draw[black,thick] (b7) -- (b7d); 
    \draw[black,thick] (d0) -- (d1u); 
        \draw[black,thick] (d1) -- (d1u); 
\draw[black,thick] (d7u) -- (d8); 
\draw[black,thick] (d7u) -- (d7); 
 \draw [black,thick](h7) -- (6,-2.1);
 \draw[black,thick] (h8) -- (6,-2);
 \draw[black,thick] (f5) -- (6,-1.9); 
 \draw [black,thick](6,-2)-- (h5);
 \draw [black,thick](6,-2)-- (0.2,0.7);
 \draw[black,thick] (-6,-2)--(f3);
 \draw[black,thick] (-6,-2)--(h1);
  \draw[black,thick] (-6,-2)--(h0);
  \draw[black,thick] (-6,-2)--(0.2,0.7);
  \draw [black,dashed](-6,-2) -- (-6,-4);
 \draw [black,dashed](e6u)--(l2u);
 \draw [black,dashed](e2u)--(l6u);
  \draw[black,thick] (i6u)--(f7d);
   \draw[black,thick] (f7) -- (f7d);%
     \draw [black,thick](f8) -- (f7d);
  \draw[black,dashed] (a4u) -- (0,25); %
  \draw[black] (a4u) -- (-12.7,12.9);																					
 \draw[black] (a4u) -- (a2d); 
 \draw[black] (a4u) -- (a6d); 
  \draw[black] (a4d) -- (x8);
  \draw[black,dashed] (a4d) -- (-3,25); 
 \draw[black,dashed] (a4d) -- (3,25); 
  
 \draw[black] (a4d) -- (a6u); 
   \draw[black,thick] (a4u) -- (b3); 
  \draw[black,thick] (a4u) -- (b5); 
  \draw[black,thick] (a4d) -- (c4);  
  \draw[black,thick] (c4) -- (b3);
  \draw[black,thick] (c4) --(b5);
   \draw[black,thick] (c4) -- (d3);
  \draw[black,thick] (c4) -- (d5);
  \draw[black,thick] (c4) -- (e4u); 
 \draw [black,thick] (d3) -- (e4d);
  \draw[black,thick] (d5) -- (e4d);
     \draw[black,thick] (e4u) -- (f5);  
   \draw[black,thick] (e4u) -- (e2d);
    \draw[black,thick] (e4u) -- (f3);  
    \draw[black,thick ] (e4u) -- (e6d);
     \draw[black,thick] (e4u) -- (h8); 
     \draw[black,thick] (e4d) -- (e2u); 
      \draw[black,thick] (e4d) -- (e6u); 
   \draw [black,thick] (e4d) -- (g4);
   \draw[black,thick](e4d) -- (d8);
   \draw [black,thick] (g4) -- (i4u);
      \draw[black,thick] (f3) -- (g4);
   \draw [black,thick](f5) -- (g4);
    \draw [black,thick] (h3) -- (g4);
   \draw [black,thick](g4) -- (h5); 
    \draw [black,thick] (h3) -- (i4d);
   \draw [black,thick](h5) -- (i4d); 
   \draw [black,thick](i4d)--(i2u); 
      \draw [black,thick](i4d)--(i6u);
      \draw [black,thick](i4u)--(i2d); 
       \draw [black,thick](i4u)--(j0); 
       \draw [black,dashed](i4u)--(-3.2,-4); 
        \draw [black,dashed](i4u)--(3.2,-4); 
  \draw [black,thick](i4d)--(f0); 

 \draw [black,dashed](i4d)--(0,-4);
  \draw [black,dashed](i6u)--(4,-4);
\draw [black,dashed](i6u)--(8,-4);
\draw [black,thick](i6u)--(12.4,-4); 
  \draw[black,thick] (h8) -- (6,-2);
 
   \draw [black,thick](e2u)--(h3);
  \draw [black,thick](e2u)--(b1d);
\draw [black,thick](e2u)--(f0);
\draw [black,thick](e2u)--(f1);

\draw [black,thick](e6u)--(f8);
\draw [black,thick](e6u)--(f7);
\draw [black,thick](e6u)--(h5);
\draw [black,thick](e6u)--(b7d);

 \draw [black,thick](i2u)--(f1d);
  \draw [black,thick](i2u)--(j0);
\draw [black,dashed](i2u)--(-8.7,-4);
\draw [black,dashed](i2u)--(-3.6,-4);

\draw [black,thick](e2d)--(d0);
\draw [black,thick](e2d)--(d1);
\draw [black,thick](e2d)--(b3);
\draw [black,thick](e2d)--(h1u);
\draw [black,thick](e2d)--(d0);

\draw [black,thick](e6d)--(d8);
\draw [black,thick](e6d)--(d7);
\draw [black,thick](e6d)--(b5);
\draw [black,thick](e6d)--(b5);
\draw [black,thick](e6d)--(h7u);

\draw [black,thick](c2) -- (b1);
  \draw [black,thick](c2) -- (b3);
 \draw [black,thick](c2) -- (d1);
 \draw [black,thick](c2) -- (d3);												
  \draw [black,thick](c2)--(a2d); 
   \draw [black,thick](c2)--(e2u);
  \draw [black,thick](g2) -- (h1);
 \draw [black,thick](g2) -- (h3);
 \draw [black,thick](g2) -- (f1);
 \draw  [black,thick](g2)--(f3);
  \draw [black,thick](g2)--(e2d); 
  \draw [black,thick](g2)--(i2u);

 \draw [black,thick](c6) -- (b5);
  \draw [black,thick](c6) -- (b7);
 \draw [black,thick](c6) -- (d5);
 \draw [black,thick](c6) -- (d7);
  \draw [black,thick](c6)--(a6d); 
   \draw [black,thick](c6)--(e6u);
 \draw [black,thick](g6) -- (f5);
  \draw [black,thick](g6) -- (f7);
 \draw [black,thick](g6) -- (h5);
 \draw [black,thick](g6) -- (h7);
  \draw [black,thick](g6)--(e6d); 
   \draw [black,thick](g6)--(i6u);
 \draw [black,dashed](6,-4) -- (i6d);
 

  \draw (b0) -- (b1d);  

   \draw (b1) -- (b1d);
      
\draw (d0) -- (d1u);
  \draw (d1) -- (d1u);
    \draw (d1) -- (d1u);
%
\draw (f0) -- (f1d);  %
 ;%
       \draw (f1d) -- (f1);

 \end{tikzpicture}
\end{center}
\caption{The first two adjacent $Q$-pieces  for $d=4$}~\label{F:2Qp}
\end{figure}
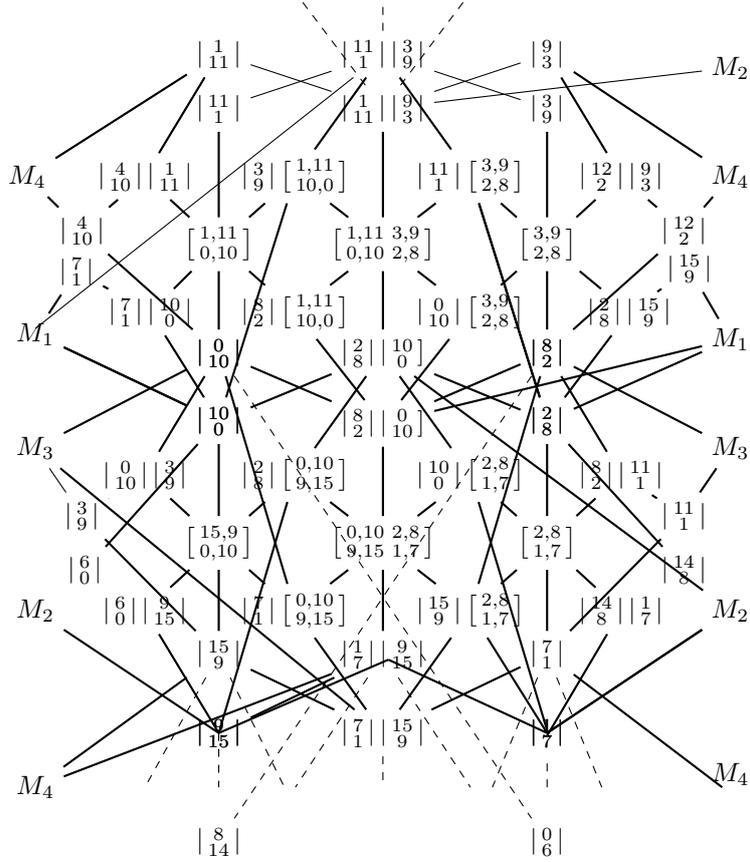




\begin{lemma}
Let $d=4$ and $i\in\{0,..,2d-1\}$. 
Any pair of adjacent $Q$-pieces $(Q_i,Q_{i+1})$ can be constructed from the first two $(Q_1,Q_2)$.
\end{lemma}
\begin{proof}
Any $Q$-piece in $\dpol_4$ can be constructed by the following method.
Let $k\in\{0,1,...,4d-1\}$. By lemma~\ref{L:rotateM}, adding $2k$ to all the pairs of integers of the matrix is equivalent to rotating the diagram by $\frac{k\pi}{2d}$ about the center of the disc. 
\begin{enumerate}
\item Consider the $Q$-diagram $\left[\begin{smallmatrix}1+2k,11+2k&3+2k,9+2k \\0+2k,10+2k &2+2k,8+2k \end{smallmatrix}\right]$. Deform the pair of  diagonals blue diagonals $(2+2k,8+2k),(4+2k,6+2k)$. This gives $\left[\begin{smallmatrix}1+2k,11+2k \\0+2k,10+2k \end{smallmatrix}\right]\left|\begin{smallmatrix}9+2k\\3+2k\end{smallmatrix}\right|$. This is an $FS$-signature. Reciprocally, deform the long red diagonal in $\left|\begin{smallmatrix}9+2k\\3+2k\end{smallmatrix}\right|$ with the short red diagonal in the $M$ tree $\left|\begin{smallmatrix}5+2k\\7+2k\end{smallmatrix}\right|$. It gives an $M$ tree, leaving the diagram to be an $S$-signature $\left[\begin{smallmatrix}1+2k,11+2k \\0+2k,10+2k \end{smallmatrix}\right]$.
It is easy to obtain the $F$ signatures, obtained from the $S$-signature in a minimal number of Whitehead moves. Those signatures are:
  
$\left|\begin{smallmatrix}7+2k\\1+2k\end{smallmatrix}\right|$, $\left|\begin{smallmatrix}0+2k\\10+2k\end{smallmatrix}\right|,$$\left|\begin{smallmatrix}10+2k\\0+2k\end{smallmatrix}\right|,$
 $\left|\begin{smallmatrix}11+2k\\1+2k\end{smallmatrix}\right|$,$\left|\begin{smallmatrix}1+2k\\11+2k\end{smallmatrix}\right|,\left|\begin{smallmatrix}4+2k\\10+2k\end{smallmatrix}\right|$. 
\vspace{3pt}
\item Deform the pair of blue diagonals $(0+2k,10+2k),(12+2k,14+2k)$ in order to obtain $\left|\begin{smallmatrix}11+2k\\1+2k\end{smallmatrix}\right|\left[\begin{smallmatrix}3+2k,9+2k \\2+2k,8+2k \end{smallmatrix}\right]$. So, one obtains an $SF$-signature. Let us deform the long blue diagonal in the tree $\left|\begin{smallmatrix}11+2k\\1+2k\end{smallmatrix}\right|$ with the short blue diagonal in the $M$ tree $\left|\begin{smallmatrix}13+2i\\15+2i\end{smallmatrix}\right|$. This turns it into an $M$ tree and so, we have the $S$-signature $\left[\begin{smallmatrix}3+2k,9+2k \\2+2k,8+2k \end{smallmatrix}\right]$. The $F$-signatures which are obtained from this $S$-signature by a minimal number of deformation operations are $\left|\begin{smallmatrix}15+2k\\9+2k\end{smallmatrix}\right|,\left|\begin{smallmatrix}2+2k\\8+2k\end{smallmatrix}\right|, \left|\begin{smallmatrix}8+2k\\2+2k\end{smallmatrix}\right|,\left|\begin{smallmatrix}12+2k\\2+2k\end{smallmatrix}\right|,\left|\begin{smallmatrix}3+2k\\9+2k\end{smallmatrix}\right|,\left|\begin{smallmatrix}9+2k\\3+2k\end{smallmatrix}\right|$. 

\item The adjacent $Q$-piece to the previous one, contains the following $Q$-diagram: 

$\left[\begin{smallmatrix}0+2k,10+2k &2+2k,8+2k \\ 15+2k,9+2k &1+2k ,7+2k  \end{smallmatrix}\right]$.

The adjacent $S$-signature, via Whitehead move, is $\left[\begin{smallmatrix}1+2k,7+2k \\2+2k ,8+2k \end{smallmatrix}\right]$. It is obtained by deforming a pair of red diagonals giving the $FS$ diagram $\left|\begin{smallmatrix}10+2k \\0+2k \end{smallmatrix}\right|\left[\begin{smallmatrix}2+2k ,8+2k  \\1+2k ,7+2k  \end{smallmatrix}\right]$, or a pair of blue diagonals giving an $FS$-signature $\left|\begin{smallmatrix}15+2k \\9+2k \end{smallmatrix}\right|\left[\begin{smallmatrix}2+2k ,8+2k  \\1+2k ,7+2k \end{smallmatrix}\right]$ . 
The $S$-signature $\left[\begin{smallmatrix}2+2k ,8+2k  \\1+2k ,7+2k  \end{smallmatrix}\right]$ is adjacent after one deformation operation to the following $F$-signatures

$\left|\begin{smallmatrix}15+2k \\9+2k \end{smallmatrix}\right|,$$\left|\begin{smallmatrix}9+2k \\15+2k \end{smallmatrix}\right|$, $\left|\begin{smallmatrix}11+2k \\1+2k \end{smallmatrix}\right|,\left|\begin{smallmatrix}14+2k \\8+2k \end{smallmatrix}\right|$$\left|\begin{smallmatrix}2+2k \\8+2k \end{smallmatrix}\right|\left|\begin{smallmatrix}8+2k \\2+2k \end{smallmatrix}\right|$. 

\vspace{3pt}
\item The $Q$-piece contains the $Q$-diagram $\left[\begin{smallmatrix}0+2k ,10+2k  &2+2k ,8+2k  \\ 15+2k ,9+2k &1+2k ,7+2k  \end{smallmatrix}\right]$ and an adjacent $S$-signature $\left[\begin{smallmatrix}15+2k ,9+2k \\0+2k ,10+2k \end{smallmatrix}\right]$ obtained by deforming a pair of blue diagonals giving the $FS$ diagram $\left|\begin{smallmatrix}10+2k \\0+2k \end{smallmatrix}\right|\left[\begin{smallmatrix}15+2k ,9+2k  \\0+2k ,10+2k  \end{smallmatrix}\right]$, or a pair of red diagonals giving an $FS$-signature $\left|\begin{smallmatrix}8+2k \\2+2k \end{smallmatrix}\right|\left[\begin{smallmatrix}2+2k,8+2k  \\1+2k ,7+2k  \end{smallmatrix}\right]$. 
The $S$-signature $\left[\begin{smallmatrix}15+2k ,9+2k  \\0+2k ,10+2k  \end{smallmatrix}\right]$ is adjacent after one deformation operation to the following $F$-signatures 
$\left|\begin{smallmatrix}7+2k \\1+2k \end{smallmatrix}\right|,\left|\begin{smallmatrix}1+2k \\7+2k \end{smallmatrix}\right|, \left|\begin{smallmatrix}6+2k \\0+2k \end{smallmatrix}\right|,\left|\begin{smallmatrix}3+2k \\9+2k \end{smallmatrix}\right|\left|\begin{smallmatrix}0+2k \\10+2k \end{smallmatrix}\right|\left|\begin{smallmatrix}10+2k \\0+2k \end{smallmatrix}\right|$. 
\end{enumerate}
Note that here we only use the generic signatures and the codimension 1 signatures. To have the higher codimension signatures, we determine among the decomposition in $Q$-pieces the $B, O $ and $NC$ structures.   
\end{proof}
\end{example}
\section{Decomposition invariant under Coxeter groups}

\subsection{Adjacence, chambers and galleries }
In this section, we show that the decomposition in elementa (discussed previously) is invariant under a Coxeter group.  

In order to prove the main statement,
we use the geometric properties of the $Q$-piece decomposition 
and the fact that signatures are invariant under polyhedral groups since they are the superimposition of diagrams, having a dihedral symmetry. 

Via this method of construction, we show explicitly the existence of chambers and galleries in the decomposition.

\begin{lemma}
The group of rotations acting on the $Q$-diagram is of order $2d$.
\end{lemma}
\begin{proof}
Consider in the $Q$-diagram separately the the blue/red monochromatic diagrams. The blue (resp. red) diagram has a mirror line $\mathcal{L}$ being parallel to the long diagonals. 
Indexing the long diagonals from 1 to $d-2$, note that if $d$ is even then this mirror line lies between the diagonals $\frac{d-2}{2}$ and $\frac{d}{2}$. If $d$ is odd then, the mirror lies on the diagonal $\frac{d-1}{2}$.  
Therefore, there exists a finite group acting independently on both monochromatic diagrams, which is of order of $2d$ defined by $\langle r | r^{2d}=Id\rangle$, $r=\frac{\pi}{2d}$. The order of the group of rotations acting on the $Q$-diagram is also of order $2d$. Thus, rotating the $Q$-diagram by an angle $\pi$ (i.e. $2d$ rotations by an angle $\frac{\pi}{2d}$) gives the identity. 
\end{proof}

 \begin{lemma}
For any $d>3$, there exist $2d$ $Q$-pieces in the decomposition.
 \end{lemma}
 \begin{proof}
Apply Whitehead moves onto the $Q$-diagram such that in each of the new generic signatures there remains at least one long blue (or red) diagonal parallel to the mirror line. In one Whitehead move, one obtains the $M$, diagrams from an $F$ signature. The union of all the strata indexed by those signatures defines a $Q$-piece of the stratification. There are $2d$ such $Q$-pieces, since there exist $2d$ rotations of a $Q$-diagram about an angle of $\frac{\pi}{2d}$. 
\end{proof}

\begin{example} Detail of a $Q$-piece for $d=4$ is given in figure~\ref{detailLego4}
\begin{figure}[h]
\begin{center}
\input{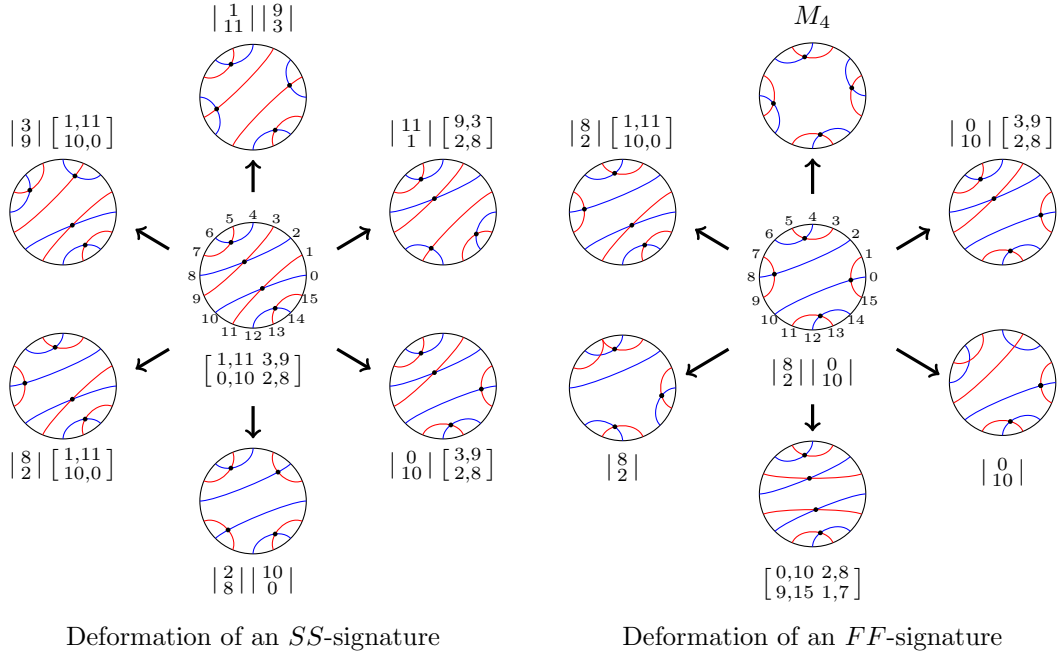}
\end{center}
\caption{Some $Q$-piece details for $d=4$}\label{detailLego4}
\end{figure}
\end{example}

To prove that the stratification is invariant under a Coxeter group, we describe the construction using an inductive procedure. 





\begin{lemma}\label{L:consec}
Let $Q_1$ and $Q_2$ be two $``Q"$-diagrams, lying in a pair of adjacent $Q$-pieces. Then, their associated matrices verify \[Q_1=\left[\begin{smallmatrix}L_{0} \\L_{1} \end{smallmatrix}\right ],Q_2=\left[\begin{smallmatrix}L_{1} \\L_{2}:=L_{1}-1\end{smallmatrix}\right],\] where the symbol $L_{1}-1$ means that 1 is substracted from each integer in $L_{1}$ modulo $4d$. 
\end{lemma}
\begin{proof}
Suppose that $Q_1=\left[\begin{smallmatrix}L_{0} \\L_{1} \end{smallmatrix}\right ]$. 
By hypothesis, since $Q_{2}$ belongs to an adjacent $Q$-piece, it is the superimposition of the blue (resp. red) monochromatic diagram of $Q_{1}$, and of a red monochromatic diagram, which is a rotation of $\frac{-\pi}{2d}$ about the red one of $Q_1$, relatively to the center of the polygon.
 
Suppose, that $L_{1}$ is the pairing of indexes of terminal vertices (a diagonal is an edge connecting a pair of labelled endverticess) of the monochromatic diagram, common to $Q_{1}$ and $Q_{2}$. From the lemma~\ref{L:rotateM} we have that $\left[\begin{smallmatrix}L_{1} \\L_{0}-2 \end{smallmatrix}\right ]$ in $Q_{2}$. Now, since the monochromatic diagrams differ by a rotation about $\frac{-\pi}{4d}$ we have that $L_{1}=L_{0}-1$ and thus $L_{2}=L_{0}-1$, modulo $4d$.   
\end{proof}
\begin{corollary}\label{C:cons}
Any pair of consecutive diagrams lying in adjacent $Q$-pieces have respectively the matrices:
 \[Q_1=\left[\begin{smallmatrix}L_{0} \\L_{1} \end{smallmatrix}\right ],Q_2=\left[\begin{smallmatrix}L_{1} \\L_{2}:=L_{1}-1\end{smallmatrix}\right],\] where the symbol $L_{1}-1$ means that 1 is substracted from each integer in $L_{1}$ modulo $4d$. 
\end{corollary}
\begin{example}
In the case of $d=3$, the adjacent $Q$-diagrams forming a spine are illustrated in Appendix A .

 The adjacent $Q$-diagrams for  $d=4$ is given in appendix, Fig~\ref{F:lego4}.
 \end{example}
 
 In the following part, one Whitehead move is an operation on only one pair of diagonals of the same color. This corresponds to the modification of one generic signature into another one. 

\begin{lemma}
Let $Q_1$ and $Q_2$ be two $Q$-diagrams, lying in adjacent $Q$-pieces. Then, these diagrams are glued to each other via a $F^{\otimes{d-2}}$ diagram.
\end{lemma}
\begin{proof}
We induct on $d>3$ (lower degrees are irrelevant since the $Q$ diagrams do not exist).
\begin{itemize}
\item {\it Base case} $d=4$. The $Q$ diagram $[\begin{smallmatrix}L_{0} \\L_{1} \end{smallmatrix}]$ has two long blue diagonals and two long red diagonals. The pair of red and blue long diagonals bound, the same 2-face in $\mathbb{D}$. Let us modify by a Whitehead move the red diagonals: this operation leaves the (blue) long diagonals fixed and thus gives a signature of type $F^{\otimes{2}}$. Therefore, a $Q$-diagram is connected to $F^{\otimes{2}}$, in one Whitehead move. 
Consider in the signature $F^{\otimes{2}}$ two pairs of adjoining short red diagonals. In each of those pairs, only one short diagonal intersect a long blue diagonal. Deform each pair simultaneously by a Whitehead move. One obtains a signature of type $SS$, with matrix $[\begin{smallmatrix}L_{1} \\L_{1}-1\end{smallmatrix}]$.
\\

Before we consider the general induction case, we enumerate the types of generic signatures obtained in one Whitehead move from a $Q$-diagram.
One deformation of the $Q$-diagram gives an $SF$ signature belong to one of the three following types: 
\begin{enumerate}
\item $S^{\otimes{d-2}}$ is modified in one Whitehead move into a signature $S^{\otimes{d-3}}F$ (or to $FS^{\otimes{d-3}}$). This is obtained by modifying a pair of diagonals (one in $M$ and one in adjoining $S$, both being of the same color).
\\
\item $S^{\otimes{d-2}}$  is deformed in one deformation step to $S^{\otimes{d-4}}FF$  (or to $FFS^{\otimes{d-4}}$). This is obtained by modifying a pair of long diagonals: one being the shortest long diagonal of the signature, the other one being adjoining to it. 
\\ 
\item $S^{\otimes{d-2}}$  is deformed in one deformation step to $S...SFFS...S$ where the number of $S$ is $d-4$. This is obtained by deforming a pair of adjoining long diagonals (which are not the shortest long diagonals of the signature).
\end{enumerate} 

\item {\it Induction case}. Consider a $Q$-diagram in $\Sigma_{d+1}$. Assume that the statement is true for a pair of adjacent $Q$-pieces where the $Q$-diagrams belong to $\Sigma_{d}$ and are of type $S^{\otimes{d-2}}$. We will prove that for two signatures of type $S^{\otimes{d-1}}$ signatures with $d-1$ pairs of $S$ trees lying in adjacent $Q$-pieces, there exists a sequence of deformations containing a signature of type $F^{\otimes{d-1}}$. Consider the signature $S^{\otimes{d-1}}$: it contains one $S$-tree added at the right (or left) of $S^{\otimes{d-2}}$, in the signature.  Using the induction hypothesis for $S^{\otimes{d-2}}$, it is known that there exists a sequence of deformations from $S^{\otimes{d-1}}$ to a signature $SF^{\otimes{d-2}}$ (resp. $F^{\otimes{d-2}}S$), where $S$ is a tree given by the crossing of the shortest long diagonals in the signature. It remains to deform a pair of diagonals lying in the $S$ tree and in its adjoining $M$ tree, so as to obtain a signature $F^{\otimes{d-1}}$. Now, starting from $F^{\otimes{d-1}}$, there exists again, by induction hypothesis, a sequence of Whitehead moves giving the signature $S^{\otimes{d-2}}F$ (resp. $FS^{\otimes{d-2}}$). Finally, one more Whitehead move applied to the pair 
of short diagonals, lying in the adjoining trees $F$ and $M$ induces the $S^{\otimes{d-1}}$ signature. 
\end{itemize}
\end{proof}

 
\subsection{Main theorems and their proofs}
The next result allows a decomposition in chambers and galleries of the configuration space. 
\begin{lemma}\label{L:O2}
The group of symmetries of a $Q$-piece is $\langle r | r^{2}= Id \rangle$.
\end{lemma}
\begin{proof}
Let us describe the construction of a $Q$-piece.  
Recall, from the previous lemmas, that the monochromatic diagrams in a $Q$-diagram have a mirror line, parallel to the long diagonals. 

The set of possible Whitehead moves applied to the set of pairs of diagonals on the right hand side of the reflection line of the $Q$-diagram, is the isomorphic to the set of possible Whitehead moves applied to the set of pairs of diagonals in the left hand side of the reflection line. Therefore, the set of adjacent signatures to the $Q$-diagram, via Whitehead moves, forming a $Q$-piece is invariant under the group $\langle r | r^{2}= Id \rangle$.






\end{proof}
 \begin{lemma}
 If $d$ is even, then the $Q$-piece has two reflections, in the sense of Coxeter groups. 
 \end{lemma}
 \begin{proof}
 From lemma~\ref{L:O2}, it is known that in a $Q$-piece there exists one reflection hyperplane, coming across the $Q$-diagram.
 Now, consider  $F$ signatures in the $Q$-piece such that there is one $F$ tree of long diagonal $(i,j)$ where $j=i+6\mod 4d$. Its reflection gives an $F$ signature with one $F$ tree of long diagonal $(i+2d,j+2d) \mod 4d$. Notice that the short diagonals of the opposite color are the same if $d$ is even. We modify $(i,j)$ by a Whitehead move with the diagonal $(i+2,i+4)\mod 4d$, so that it becomes $(i,i+2)$ $(i+4,i+6)$. By symmetry, we proceed on 
   $ (i+2d,j+2d)$ and $(i+2+2d,i+4+2d)$ so, that the Whitehead operation gives the pair $(i+2d,i+2+2d) (i+4+2d,i+6+2d) \mod 4d$. Hence, we obtain the same $M$ signatures. Therefore, there is a second reflection hyperplane coming across the $M$ signatures in the $Q$-piece.  
\end{proof}

\begin{lemma}\label{Pr:G}
Let $Z$ be a couple of adjacent $Q$-pieces and let $Y$ be the stratification. 
Then ${\bf p}:Y\to Z$ is a Galois covering, with Galois group of order $d$. 
\end{lemma}
\begin{proof}

Let us define ${\bf p}:Y\to Z$ where $Y=\cup_{\rho \in G}Z^{\rho}$ and $G$ is a cyclic group of finite order $d$, where $Y$ is connected. 
Each inverse image ${\bf p}^{-1}(A_{\sigma})$ of $A_{\sigma}\in Z$ is constituted from classes, having signatures which are equivalent up to a rotation of $\sigma$. So, any action $\rho \in Aut_{Z}(Y)$ on $\sigma$ gives a signature $\sigma'$ which belongs to ${\bf p}^{-1}(A_{\sigma})$.
We show that for $Y\times_{Z} Y =\{(z,z')\in Y\times Y, {\bf p}(z)={\bf p}(z')\}$, the map 
\[\phi: G\times Y \to Y\times_{Z} Y, \]
\[(g,z)\mapsto (z,gz)\] is a homeomorphism. 
\begin{itemize}
\item The map is a bijection. 
First let us show that the map is injective. Consider $\phi(g,z)=\phi(g',z')$ i.e. $(z,gz)=(z',g'z')\in Y\times_{Z} Y$. Then, $z=z'$ and $g'=g$ and in particular $(g,z)=(g',z')$ in $G\times Y$, so the map is injective. The map is surjective since for every $(z,gz)\in Y\times_{Z} Y$, there exists at least one element in $(g,z)\in G\times Y$ such that $\phi((g,z))=(z,gz)$. 
\item The map is bicontinuous because the group $G$ continuously acts on $Y$.
\end{itemize}
So, the map ${\bf p}:Y\to Z$ satisfies the definition of a Galois covering for an order $d$ Galois group.
\end{proof}

\begin{lemma}\label{L:rotate}
Let $S_1=[\begin{smallmatrix}L_{0}+2d \\L_{0}+2d-1 \end{smallmatrix}]$ and $S_2=[\begin{smallmatrix}L_{0}-2d \\L_{0}-2d-1 \end{smallmatrix}]$ be two $S$ signatures.
 Then, $S_1=S_2$. 
\end{lemma}
\begin{proof}
Consider any diagonal $(i,j)\in L_0$, where $i,j$ are integers of the same parity modulo $4d$. We have $i\equiv i \mod 4d\iff i+4d\equiv i\mod 4d \iff i+2d\equiv i-2d\mod 4d$.
Proceeding similarly for $j$ we have that $S_1=S_2$.
\end{proof}

\noindent \underline{{\bf Decomposition in $Q$-pieces}}

Consider a $Q$-piece. Let LHS (resp. RHS) denote the left (resp. right) hand side of the $Q$-piece about its reflection axis; let $S_1$ (resp. $S_2$) be a signature in the LHS  (resp. RHS) of the $Q-$piece, such that $S_1=[\begin{smallmatrix}L_{0} \\L_{0}-1 \end{smallmatrix}],S_2=[\begin{smallmatrix}L_{0}+2d \\L_{0}-1+2d \end{smallmatrix}]$. 
Let us discuss the relations between each pair of adjacent $Q$-piece. 
\begin{enumerate}
\item Consider a $Q_1$-piece. Any signature $\sigma_1$ in the LHS is associated to $[\begin{smallmatrix}L_0\\L_0-1\end{smallmatrix}]$. Its reflection $r(\sigma_1)$ in the RHS is $\sigma_1$ rotated by an angle $\pi$ and is given by $[\begin{smallmatrix} L_0+2d\\ L_0-1+2d \end{smallmatrix}]$. 
Consider the adjacent $Q_2$-piece. Applying lemma~\ref{L:consec}: any signature $\sigma_2$ of the $Q_2$-piece is obtained by the copy-paste and turn step. Therefore, if $\sigma_2$ lies in the LHS part of the $Q_2$-piece then, its matrix is $[\begin{smallmatrix}L_0-1\\L_0-2\end{smallmatrix}]$. Its reflection in the RHS part is $[\begin{smallmatrix} L_0+2d-1\\ L_0-1+2d-2 \end{smallmatrix}].$
\\
\item More generally, take any signature in the LHS of the $i$-th $Q$-piece. Its matrix is $[\begin{smallmatrix}L_0-i+1\\L_0-i\end{smallmatrix}]$ and its reflection in RHS has matrix $[\begin{smallmatrix} L_0-i+1+2d\\ L_0-i+2d \end{smallmatrix}]$. Applying lemma~\ref{L:consec}, the consecutive signatures are respectively of type $[\begin{smallmatrix}L_0-i\\L_0-(i+1)\end{smallmatrix}]$ and $[\begin{smallmatrix} L_0-i+2d\\L_0-(i+1)+2d \end{smallmatrix}]$.
\\
\item For the $(2d-1)$-th $Q_{2d-1}$-piece, we have $[\begin{smallmatrix}L_0-(2d-1)\\L_0-2d\end{smallmatrix}]$ and $[\begin{smallmatrix} L_0+2d-(2d-2)\\ L_0+2d-1-(2d-1) \end{smallmatrix}]$.
The consecutive signatures are respectively $[\begin{smallmatrix}L_0-(2d-1)\\L_0-2d \end{smallmatrix}]$ and $[\begin{smallmatrix} L_0+2d-2d\\ L_0+2d-1-2d \end{smallmatrix}]$.
  
\begin{remark}
Note that, after $4d$ iterations of this procedure the identity is obtained.  
\end{remark}


\end{enumerate}
Below we present a part of the collection of glued $Q$-pieces, the connections are represented by thin vertical double arrow, deformations by a thick horizontal double arrow. The final structure is obtained by glueing the top and the bottom of Fig~\ref{F:Q}, along the long and thin horizontal arrows.
\begin{figure}[h]
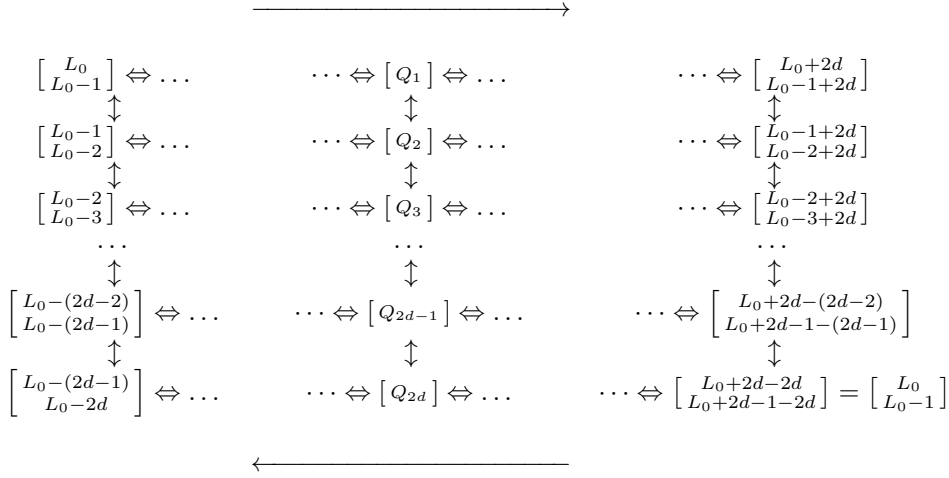

\begin{center}
\begin{tabular}{c c c }
&$\xrightarrow{\hspace*{4cm}} $&\\
& \\
$\left[\begin{smallmatrix}L_0\\L_0-1\end{smallmatrix}\right]\Leftrightarrow \dots$& $\dots\Leftrightarrow \left[\begin{smallmatrix}Q_{1}\end{smallmatrix}\right]\Leftrightarrow\dots$&$\dots\Leftrightarrow\left[\begin{smallmatrix} L_0+2d\\ L_0-1+2d \end{smallmatrix}\right]$\\
$\updownarrow$&  $\updownarrow$&$\updownarrow$ \\
 $\left[\begin{smallmatrix}L_0-1\\L_0-2\end{smallmatrix}\right] \Leftrightarrow\dots $ & $\dots\Leftrightarrow\left[\begin{smallmatrix}Q_{2}\end{smallmatrix}\right]\Leftrightarrow\dots $&$\dots \Leftrightarrow\left[\begin{smallmatrix} L_0-1+2d\\ L_0-2+2d \end{smallmatrix}\right]$\\
 $\updownarrow$&  $\updownarrow$&$\updownarrow$ \\
$\left[\begin{smallmatrix}L_0-2\\L_0-3\end{smallmatrix}\right] \Leftrightarrow\dots $ & $\dots \Leftrightarrow\left[\begin{smallmatrix}Q_{3}\end{smallmatrix}\right]\Leftrightarrow\dots $&$\dots \Leftrightarrow\left[\begin{smallmatrix} L_0-2+2d\\ L_0-3+2d \end{smallmatrix}\right]$\\
$\dots$&$\dots$ & $\dots$\\
$\updownarrow$&  $\updownarrow$&$\updownarrow$ \\
$\left[\begin{smallmatrix}L_0-(2d-2)\\L_0-(2d-1)\end{smallmatrix}\right] \Leftrightarrow\dots $ & $\dots \Leftrightarrow\left[\begin{smallmatrix}Q_{2d-1}\end{smallmatrix}\right]\Leftrightarrow\dots $&$\dots \Leftrightarrow\left[\begin{smallmatrix} L_0+2d-(2d-2)\\ L_0+2d-1-(2d-1) \end{smallmatrix}\right]$\\
$\updownarrow$&  $\updownarrow$&$\updownarrow$ \\
$\left[\begin{smallmatrix}L_0-(2d-1)\\L_0-2d\end{smallmatrix}\right] \Leftrightarrow\dots $ & $\dots \Leftrightarrow\left[\begin{smallmatrix}Q_{2d}\end{smallmatrix}\right]\Leftrightarrow\dots $&$\dots \Leftrightarrow\left[\begin{smallmatrix} L_0+2d-2d\\ L_0+2d-1-2d \end{smallmatrix}\right]=\left[\begin{smallmatrix} L_0\\ L_0-1\end{smallmatrix}\right]$\\
 & \\
 &$\xleftarrow{\hspace*{4cm}} $&\\
\end{tabular}
\end{center}
\caption{Detail of the collection of glued $Q$-pieces.}\label{F:Q}
\end{figure}

\begin{theorem}\label{Th:mob}
The decomposition of $\dpol_{d}$ in $4d$ ($d>2$) Weyl-Coxeter-chambers, is induced by the topological stratification of $\dpol_{d}$ in elementa and is invariant under the the Coxeter group given by the presentation:
\[W=\langle r,s_1,...,s_{2d} | r^2=Id, (rs_j)^{2p}=Id, p|d\rangle, p\in \mathbb{N}^{*}.\]
\end{theorem}

\begin{proof}
Let us recall from lemma~\ref{L:O2} that each $Q$-piece is invariant under an automorphism group of order 2. Moreover, it follows from the proposition~\ref{Pr:G} that the stratification is invariant under a cyclic group of order $2d$. 
We show that the stratification is invariant under the Coxeter group $W=\langle r,s_1,...,s_{2d} | r^2=Id, (rs_j)^{2p}=Id, p|d\rangle$, using the $Q$-procedure described previously.

Between any pair of adjacent $Q$-pieces there exists a reflection hyperplane.  Let us consider the pair of $i$-th and $(i+1)$ adjacent $Q$-pieces, where $i\in \{0,...,2d-1\}$. Denote this pair by $(Q_{i},Q_{i+1}$).

The the set of blue (resp. red) monochromatic diagrams are bijectively mapped, by the Identity, to the set of blue (resp. red) monochromatic diagrams in the $Q_{i+1}$-piece: this is the copy and paste step. As for the set of red (resp. blue) monochromatic diagrams, they are bijectively mapped to the set of red (resp. blue) monochromatic diagrams in the $Q_{i+1}$-piece as follows: after making a copy and paste step, red (resp. blue) diagrams are rotated by $\pi$ about the center of the diagram. In other words, each red monochromatic diagram in the $Q_{i+1}$-piece is symmetric to its pre-image in the $Q_{i}$-piece, about the vertical reflection axis of the blue (resp. red) monochromatic diagram in the $Q$-diagram. 

Therefore, we have a bijection which is order preserving between from the set of signatures in the $Q_i$-piece to the set of signatures in the  $Q_{i+1}$-piece: the incidence relations in both $Q$-pieces are preserved.


From the definition~\ref{D:Q} of adjacent $Q$-pieces it follows that the adjacent $Q$-pieces are symmetric about a reflection hyperplane.

The next $Q$-piece, indexed $(i+2)$, is obtained by making a copy and paste step from the set of signatures in the $i$-th $Q$-piece and by rotating those signatures by $\frac{\pi}{2d}$ about the center of the signature. 
Using the same argument as previously, we have thus, $2d$ reflection hyperplanes.
to conclude, this implies that there exist $2d$ horizontal reflection hyperplanes and one vertical reflection hyperplane.
\end{proof}

\begin{corollary} 
For $d>2$, there exist $4d$ chambers in this stratification, in the sense of Weyl-Coxeter. One $Q$-piece is the union of two chambers.
\end{corollary}

\begin{corollary}\label{Th:S} 
Let $\mathcal{W}$ be the inclusion diagram. Then $\mathcal{W}$ is invariant under the Klein group  $\mathbb{Z}_{2}\rtimes\mathbb{Z}_{2}$.
\end{corollary}

\subsection{Braid towers}
We have shown the existence of a stratification of $\dpol_d$, which forms a decomposition of this space, which is invariant under a Coxeter group. 
Since $\dpol_d$ is isomorphic to the configuration space of $d$ marked points on the complex plane, denoted $Conf(\mathbb{C})_d$, we have therefore a stratification of this configuration space which is invariant under this polyhedral Coxeter group. From the other side, configuration spaces are related to the theory of braids. By a result of Fadell, Fox, Neuwirth~\cite{FaN62, FoN62} it is well known that the braid group with $d$ strands is the fundamental group of the configuration space $Conf(\mathbb{C})_d$. It is interesting to explain our geometric approach to braid theory.
\begin{theorem}
Any braid relations following from the natural inclusion of braids $i^{\star}: B_d  \hookrightarrow B_{d+1}$ can be described using the decomposition in $Q$-pieces. 
\end{theorem}
\begin{proof}
\hspace{1cm}(i) The first application of the decomposition in $Q$-pieces appears in the case of the embedding $\dpol_{d}\hookrightarrow  \dpol_{d+1}$ (table~\ref{T:embed}). More concretely, this $Q$-decomposition lists and enumerates all the possible intertwinings of the roots in $\dpol_{d}$ in an explicit and geometric way. Indeed, starting with a given configuration of polynomial roots (corresponding to the one in a $Q$-diagram), we obtain all the possible intertwinings that may happen between the roots.
\small{\begin{table}[h]
\begin{center}
\renewcommand{\arraystretch}{1.5}
\begin{tabular}{|c|c|c|c|c|c|c|c|}\hline
d &3&4&5&6&7  \\ \hline
mod&12&16&20&24&28  \\
&$0\sim 12$&$0\sim 16$&$0\sim 20$&$0\sim 24$&$0\sim 28$  \\ \hline

      &&&&&$\left[\begin{smallmatrix} 21,15&23,13&25,11&27,9\,&1,7\, \\20,14&22,12&24,10&26,8&28,6\,\end{smallmatrix}\right]$  \\
      
    &&&&&$\left[\begin{smallmatrix} 20,14&22,12&24,10&26,8\,&28,6\,\\19,13&21,11&23,9\,&25,7\,&27,5\,\end{smallmatrix}\right]$  \\\cline{5-5}
  
       &&&&$\left[\begin{smallmatrix} 19,13&21,11&23,13&\,1,7\ \\18,12&20,10&22,12&24,6\end{smallmatrix}\right]$&$\left[\begin{smallmatrix} 19,13&21,11&23,9\,&25,7\,&27,5\,\\18,12&20,10&22,8\,&24,6\,&26,4\,\end{smallmatrix}\right]$\\ 
       
        &&&&$\left[\begin{smallmatrix} 18,12&20,10&22,12&24,6\,\\17,11&19,9\, &21,7\, &23,5\,\end{smallmatrix}\right]$&$\left[\begin{smallmatrix} 18,12&20,10&22,8\,&24,6\,&26,4\,\\17,11&19,9\ &21,7\,&23,5\, &25,3\,\end{smallmatrix}\right]$  \\  \cline{4-4}
  
 &&&$\left[\begin{smallmatrix} 17,11&19,9\ &\,1,7\,\\16,10&18,8\ &20,6\,\end{smallmatrix}\right]$ &$\left[\begin{smallmatrix} 17,11&19,9\,&21,7\,&23,5\,\\16,10&18,8\,&20,6\,&22,4\,\end{smallmatrix}\right]$&$\left[\begin{smallmatrix} 17,11&19,9\,&21,7\ ,&23,5\,&25,3\,\\16,10&18,8\,&20,6\ &22,4\ &24,2\,\end{smallmatrix}\right]$   \\
 
   &&&$\left[\begin{smallmatrix} 16,10\,&18,8\,&20,6\,\\15,9\,&17,7\,&19,5\, \end{smallmatrix}\right]$&$\left[\begin{smallmatrix} 16,10&18,8\,&20,6\,&22,4\,\\15,9\,&17,7\,&19,5\,&21,3\,\end{smallmatrix}\right]$&$\left[\begin{smallmatrix} 16,10&18,8\,&20,6\,&22,4\, &24,2\,\\15,9&17,7\,&19,5\ &21,3\ &23,1\,\end{smallmatrix}\right]$   \\  \cline{3-3}
   
 &&$\left[\begin{smallmatrix}15,9\,&\,1,7\,\\14,8\,&16,6\,\end{smallmatrix}\right]$&$\left[\begin{smallmatrix} 15,9\ &17,7\,&19,5\,\\14,8\,Ê&16,6\,Ê&18,4\,\end{smallmatrix}\right]$&$\left[\begin{smallmatrix} 15,9\,&17,7\,&19,5\,&21,3\,\\14,8\,&16,6\,&18,4\,&20,2\,\end{smallmatrix}\right]$
 &$\left[\begin{smallmatrix} 15,9\,&17,7\,\ &19,5\,&21,3\ &23,1\,\\14,8\,&16,6\ &18,4\,&20,2\ &22,28\,\end{smallmatrix}\right]$ \\ 
 
&&$\left[\begin{smallmatrix}14,8\,&16,6\,\\13,7\,&15,5\,\end{smallmatrix}\right]$&$\left[\begin{smallmatrix} 14,8\,&16,6\,&18,4\,\\13,7\,&15,5\,&17,3\,\end{smallmatrix}\right]$&$\left[\begin{smallmatrix} 14,8\,&16,6\,&18,4\,&20,2\,\\13,7\,&15,5\,&17,3\,&19,1\,\end{smallmatrix}\right]$&$\left[\begin{smallmatrix} 14,8\,&16,6\ &18,4\,&20,2\ &22,28\,\\13,7\,&15,5\ &17,3\,&19,1\ &21,27\end{smallmatrix}\right]$\\ \cline{2-2}

 &$\left[\begin{smallmatrix}\,1,7\,\\12,6\,\end{smallmatrix}\right]$&$\left[\begin{smallmatrix}13,7\,&15,5\,\\12,6\,&14,4\,\end{smallmatrix}\right]$&$\left[\begin{smallmatrix} 13,7\,&15,5\,&17,3\,\\12,6\,&14,4\,&16,2\,\end{smallmatrix}\right]$&$\left[\begin{smallmatrix} 13,7\,&15,5\,&17,3\,&19,1\,\\12,6\,&14,4\,&16,2\,&18,24\end{smallmatrix}\right]$&$\left[\begin{smallmatrix} 13,7\,&15,5\,&17,3\,&19,1\, &21,27\\12,6\,&14,4\,&16,2\,&18,28\ &20,26\end{smallmatrix}\right]$ \\ 
 
 &$\left[\begin{smallmatrix}12,6\\11,5\\\end{smallmatrix}\right]$&$\left[\begin{smallmatrix}12,6\,&14,4\,\\11,5\,&13,3\,\\\end{smallmatrix}\right]$&$\left[\begin{smallmatrix} 12,6\,&14,4\,&16,2\,\\11,5\,\,&13,3\,&15,1\,\end{smallmatrix}\right]$&$\left[\begin{smallmatrix} 12,6\,&14,4\,&16,2\,&18,24\\11,5\,&13,3\,&15,1\,&17,23\end{smallmatrix}\right]$&$\left[\begin{smallmatrix} 12,6\,&14,4\, &16,2\,&18,28\ &20,26\\11,5\,&13,3\,&15,1\,&17,27\ &19,25\end{smallmatrix}\right]$\\ 
 
 &$\left[\begin{smallmatrix}11,5\,\\10,4\,\\\end{smallmatrix}\right]$&$\left[\begin{smallmatrix}11,5\,&13,3\,\\10,4\,&12,2\,\end{smallmatrix}\right]$&$\left[\begin{smallmatrix} 11,5\,&13,5\,&15,1\,\\10,4\,&12,2\,&14,20\\\end{smallmatrix}\right]$&$\left[\begin{smallmatrix} 11,5\,&13,3\,&15,1\,&17,23\\10,4\,&12,2\,&14,24&16,22\end{smallmatrix}\right]$&$\left[\begin{smallmatrix} 11,5\,&13,3\,&15,1\,&17,27\ &19,25\\10,4\,&12,2\,&14,28\,&16,26\ &18,24\end{smallmatrix}\right]$\\ 
 
 &$\left[\begin{smallmatrix}10,4\,\\\,9,3\,\end{smallmatrix}\right]$&$\left[\begin{smallmatrix}10,4&12,2\\9,3&11,1\end{smallmatrix}\right]$&$\left[\begin{smallmatrix} 10,4\,&12,4\,&14,20\\9,3\,&11,1\,&13,19\\\end{smallmatrix}\right]$&$\left[\begin{smallmatrix} 10,4\,&12,2\,&14,24&16,22\\9,3\,&11,1\,&13,23&15,21\end{smallmatrix}\right]$&$\left[\begin{smallmatrix} 10,4\ &12,2\, &14,28&16,26&18,24\\ \,9,3\,&11,1\, &13,27&15,25&17,23\end{smallmatrix}\right]$\\ 
 
 &$\left[\begin{smallmatrix}  9,3\,&\\ \, 8,2\,\end{smallmatrix}\right]$&$\left[\begin{smallmatrix} \,9,3\,&11,1\,\\ \,8,2\,&10,16\,\end{smallmatrix}\right]$&$\left[\begin{smallmatrix} 9,3\,&11,1\,&13,19\\8,2\,&10,20&12,18\\\end{smallmatrix}\right]$&$\left[\begin{smallmatrix} 9,3\,&11,1\,&13,23&15,21\\ 8,2\,&10,24&12,22&14,20\end{smallmatrix}\right]$&$\left[\begin{smallmatrix} \,\,9,3\,&11,1\, &13,27&15,25&17,23 \\ \,8,2\,&\,10,28\, &12,26&14,24&16,22\end{smallmatrix}\right]$\\ 
 
 &$\left[\begin{smallmatrix}\, 8,2\,\\  \,7,1\,\end{smallmatrix}\right]$&$\left[\begin{smallmatrix}\, 8,2\,&10,16\,\\\,7,1\,&\,9,15\end{smallmatrix}\right]$&$\left[\begin{smallmatrix} 8,2\,\,&10,20&12,18\\7,1&9,19&11,17\end{smallmatrix}\right]$&$\left[\begin{smallmatrix} \,8,2\,&10,24&12,22&14,20\\ \,7,1\,&\,9,23&11,21&13,19\end{smallmatrix}\right]$&$\left[\begin{smallmatrix} \,8,2\,&\,10,28&12,26&14,24&16,22\\ \,7,1\,&\,9,27&11,25&13,23&15,21\end{smallmatrix}\right]$\\ 
\hline
\end{tabular}
\end{center}
\vspace{5pt}
\caption{Table of inclusions}\label{T:embed}
\end{table}
}

More precisely, a pair of intertwining roots corresponds to a set of 4 generic signatures - connected between each other by 4 codimension 1 signatures, forming a quadrangle, in the nerve (see figure~\ref{F:lego4}, for $d=4$ ). 

The natural embedding $Q(d) \hookrightarrow Q(d+1)$  gives an inductive and explicit  method of the embedding $\dpol_{d}\hookrightarrow  \dpol_{d+1}$. Indeed, the  procedure of  $Q(d) \hookrightarrow Q(d+1)$ is equivalent to adding a tree in the  $Q(d)$ in trigonometric way. 

As a remark, note that the embedding $\dpol_{d} \hookrightarrow \dpol_{d+1}$
extends to the embedding $ j: \mathbb{C}^{d} \hookrightarrow \mathbb{C}^{d+1}$, defined as follows. 
Take a polynomial $P$ in $\mathbb{C}^{d}$ with $d$ roots (possibly with some multiplicities), $ z_1,\dots,z_d$. The  roots of the  polynomial $j\circ P$ are given by $ z_1,\dots,z_d,z_{d+1}$ where $ z_1,\dots,z_d$ are roots of $P$ and $z_{d+1}= z_0+\max_{1 \leq i < d}\,|z_i-z_0| +1$, with $z_0$ is the arithmetic mean of the roots $z_1,\dots,z_d$, for more details see~\cite{Ar70}(page 44).

\vspace{3pt} 

\hspace{1cm}(ii)This application of the construction has implications for the injective homomorphism of groups (reciprocally for the projection map $B_{d+1} \hookrightarrow B_{d}$). 
Indeed, each quadrangle is an intertwining of roots, therefore it corresponds to an elementary braid. Therefore, the algorithmic method and the induction explained above illustrates the inclusion 
$i^{\star}: B_d  \hookrightarrow B_{d+1}$. 

\vspace{3pt} 
\hspace{1cm}(iii) Finally, using this construction we have a geometric interpretation of the increasing chain of braid groups. 
 
More precisely,
 there exist a natural embedding $Q(d) \hookrightarrow Q(d+1)$ between $Q(d)$-diagrams of polynomials of degree $d$ and $Q(d+1)$-diagrams of polynomials of degree $d+1$ as shown in  table~\ref{T:embed}. The construction gives an explicit presentation of the  embedding $\dpol_{d} \hookrightarrow \dpol_{d+1}$. 
This property holds in a more general framework: for $Q(d) \hookrightarrow Q(d+k)$, where $k\geq1$. 

The decomposition formed from the union of all $Q$-pieces presents, in an explicit way, all the possible intertwining of the roots. Thus, considering a decomposition for the $d+1$ marked points, we have  an explicit illustration, of the embedding $\dpol_{d}\hookrightarrow \dpol_{d+1}$ and thus for $j: \mathbb{C}^{d} \hookrightarrow \mathbb{C}^{d+1}$.

Since, the intertwining of the roots corresponds to an elementary braid, this construction shows the natural inclusion of braids.  This inclusion is an injective homomorphism of groups, such that:
$ i^{\star}: B_{d} \to B_{d+1}$, with $i^{\star}( \sigma^{(d)}_i) = \sigma^{(d+1)}_{i}$ for $i$ in the set from $1$ to $d-1$. So, that  there exists
\begin{itemize}
\item  an embedding of $\dpol_{d}\hookrightarrow  \dpol_{d+1}$ given explicitly by $Q$-pieces;
\item an increasing chain of braid groups  $B_1 \subset B_2\subset B_3\subset\dots B_d \subset B_{d+1}$. associated to $Q$-pieces.
\end{itemize}

In $\pi_1(\dpol_{d})$ there exists  paths  given by a ``quadrangles''  forming a loop~\cite{C0},  corresponding to, adjacent two by two, four elementas, which are generic and of codimemsion 1.
Therefore, the $Q$-pieces allow a description of the braid relations. In particular, an explicit description of the inclusion of braids is this given.
\end{proof}

\begin{corollary}
Any braid relations following from the projection $B_{d+1}\rightarrow B_{d}$ can be described using $Q$-pieces. 
 \end{corollary}
 
 \begin{corollary}
The braid operad can be described using $Q$-pieces. 
 \end{corollary}

\section{Concluding remark}
In this paper, we have investigated a new cell decomposition of $\mathcal{M}_{0,[d]}$, using its natural relation with the space of complex monic degree $d>0$ polynomials in one variable with simple roots, which the $d$-th unordered configuration space of the complex plane space.
The decomposition of $\dpol_d$ has been done considering {\it drawings of polynomials} - objects reminiscent of Grothendieck's {\it dessin's d'enfant}. These objects are decorated graphs, properly embedded in the complex plane, obtained by taking the inverse image under a  polynomial $P$ of the real and imaginary axes. A detailed study of those objects has been given in section 2. 
A stratum, (called {\it elementa}) of the decomposition is defined as a set of polynomials having drawings belonging to the same isotopy class, relatively to the asymptotic directions and which is described as a signature (for simplicity signatures are represented as circular diagrams).
An important step towards the stratification is that the topological closure of a stratum indexed by the signature $\sigma$ is given by the combinatorial closure of the signature $\sigma$~\cite{C0}, which motivated our investigations in this paper concerning the combinatorial closure of an elementa.
The study of this stratification lead to the introduction of the notion of an inclusion diagram which is a precious tool describing geometrically not only the neighborhood of each stratum but the entire stratification itself. It is also known that this inclusion diagram is the nerve in the sense of \v Cech, of the cover in~\cite{C0,C1}. We showed that the stratification for $d>2$ is invariant under a Coxeter group with $4d$ chambers, $2d$ horizontal reflection hyperplanes and one vertical reflection hyperplane. 
To conclude, the paths and loops in $\dpol_d$ have been defined in a new way, using diagrams. Moreover, from the deep relation $\pi_1(\dpol_d)=B_d$ (\cite{Bri71,FaN62,FoN62}) we showed that a braid can be defined, using a sequence of signatures, obtained one from another via Whitehead move. From the result, showed in this paper, a braid can be defined as a path (loop) in one chamber. This result is also very advantageous for the calculation of the \v Cech cohomology of this space, since it reduces significantly the complexity. 
On the other side, it allows an alternative description to the braid operad. 

\vfill\eject

\appendix
\section{Inclusion diagram for $d=3$}~\label{A: Incldiag3}
\begin{figure}[h]
\begin{center}
\begin{tikzpicture}[scale=.5]
  \node (a1) at (-6,8) {$M_{1}$};
   \node (a2) at (-4,8) {$\left|\begin{smallmatrix}7\\1\end{smallmatrix}\right|$};
     \node (a3) at (-2,8) {$M_{3}$};
  
  \node (b1) at (-6,6) {$\left|\begin{smallmatrix}6\\0\end{smallmatrix}\right|$};
    \node (b2) at (-4,6) {$\left[\begin{smallmatrix}0,6\\1,7\end{smallmatrix}\right]$};
      \node (b3) at (-2,6) {$\left|\begin{smallmatrix}0\\6\end{smallmatrix}\right|$};
    
  \node (c1) at (-6,4) {$M_{2}$};
     \node (c2) at (-4,4) {$\left|\begin{smallmatrix}1\\7\end{smallmatrix}\right|$}; 
       \node (c3) at (-2,4) {$M_{4}$};
        \node (c4) at (0,4) {$\left|\begin{smallmatrix}9\\3\end{smallmatrix}\right|$};  
         \node (c5) at (2,4) {$M_{2}$};
    
     \node (d1) at (-6,2) {$\left|\begin{smallmatrix}2\\8\end{smallmatrix}\right|$};
     \node (d2) at (-4,2) {$\left[\begin{smallmatrix}1,7\\2,8\end{smallmatrix}\right]$}; 
       \node (d3) at (-2,2) {$\left|\begin{smallmatrix}8\\2\end{smallmatrix}\right|$};
        \node (d4) at (0,2) {$\left[\begin{smallmatrix}2,8\\3,9\end{smallmatrix}\right]$};  
         \node (d5) at (2,2) {$\left|\begin{smallmatrix}2\\8\end{smallmatrix}\right|$};
   
    \node (e1) at (-6,0) {$M_{1}$};
     \node (e2) at (-4,0) {$\left|\begin{smallmatrix}7\\1\end{smallmatrix}\right|$}; 
       \node (e3) at (-2,0) {$M_{3}$};
        \node (e4) at (0,0) {$\left|\begin{smallmatrix}3\\9\end{smallmatrix}\right|$};
          \node (e5) at (2,0) {$M_{1}$}; 
         \node (e6) at (4,0) {$\left|\begin{smallmatrix}11\\5\end{smallmatrix}\right|$};
                \node (e7) at (6,0) {$M_{3}$};
                
      \node (f3) at (-2,-2) {$\left|\begin{smallmatrix}4\\10\end{smallmatrix}\right|$};
        \node (f4) at (0,-2) {$\left[\begin{smallmatrix}3,9\\4,10\end{smallmatrix}\right]$};
          \node (f5) at (2,-2) {$\left|\begin{smallmatrix}10\\4\end{smallmatrix}\right|$}; 
         \node (f6) at (4,-2) {$\left[\begin{smallmatrix}4,10\\5,11\end{smallmatrix}\right]$};
                \node (f7) at (6,-2) {$\left|\begin{smallmatrix}4\\10\end{smallmatrix}\right|$};

      \node (g3) at (-2,-4) {$M_{4}$};
        \node (g4) at (0,-4) {$\left|\begin{smallmatrix}9\\3\end{smallmatrix}\right|$};
          \node (g5) at (2,-4) {$M_{2}$}; 
         \node (g6) at (4,-4) {$\left|\begin{smallmatrix}5\\11\end{smallmatrix}\right|$};
                \node (g7) at (6,-4) {$M_{4}$};
                
       \node (h5) at (2,-6) {$\left|\begin{smallmatrix}6\\0\end{smallmatrix}\right|$}; 
         \node (h6) at (4,-6) {$\left[\begin{smallmatrix}5,11\\6,0\end{smallmatrix}\right]$};
                \node (h7) at (6,-6) {$\left|\begin{smallmatrix}0\\6\end{smallmatrix}\right|$};  
                
    \node (i5) at (2,-8) {$M_{1}$}; 
         \node (i6) at (4,-8) {$\left|\begin{smallmatrix}11\\5\end{smallmatrix}\right|$};
                \node (i7) at (6,-8) {$M_{3}$};               
  \draw (a1) -- (a2) -- (a3);
   \draw (b1) -- (b2) -- (b3);
   \draw (c1) -- (c2) -- (c3);
    \draw [yellow,ultra thick,double](c3) -- (c4) -- (c5);
    \draw (d1) -- (d2) -- (d3) -- (d4) -- (d5);
      \draw (e1) -- (e2) -- (e3) -- (e4) --(e5);
       \draw[orange,ultra thick,double] (e5)-- (e6) -- (e7);
       \draw (f3) -- (f4) -- (f5)-- (f6) -- (f7);
        \draw  [yellow,ultra thick,double](g3) -- (g4) -- (g5);
        \draw(g5)-- (g6) -- (g7);
        \draw (h5)-- (h6) -- (h7);
      \draw[orange,ultra thick,double]  (i5)-- (i6) -- (i7);
    \draw[red,ultra thick] (a1) -- (b1)-- (c1);
        \draw[green,ultra thick] (c1)-- (d1)-- (e1);
   
     \draw (a2) -- (b2) -- (c2)-- (d2)--(e2);
     \draw [ violet,ultra thick](a3) -- (b3) -- (c3);
      \draw (c3)-- (d3)--(e3)--(f3)--(g3);
      \draw  [ blue,ultra thick](e3)--(f3)--(g3);
      \draw (c4) -- (d4)-- (e4)--(f4)--(g4) ;
            \draw[green,ultra thick] (c5) -- (d5)-- (e5);
           \draw (e5) --(f5)--(g5);%
             \draw [red,ultra thick](g5)--(h5)--(i5);
\draw (e6) -- (f6)--(g6)--(h6)--(i6);
\draw  [ blue,ultra thick](e7) -- (f7)--(g7);
\draw[violet,ultra thick](g7)--(h7)--(i7);
  \end{tikzpicture}
\end{center}
\vspace{-10pt}
\caption{The 2-faces of the inclusion diagram for $d=3$}\label{F:2Fincldiag}
\end{figure}
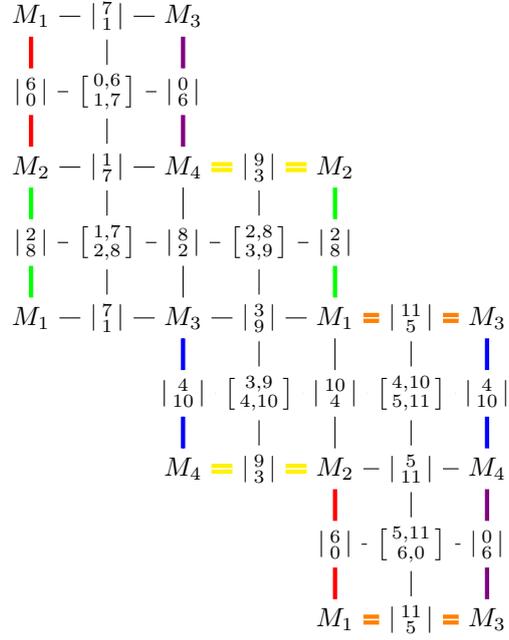

\begin{figure}[h]
\begin{center}
\begin{tikzpicture}[scale=0.6]
  \node (A1) at (-3.9,5.2) {$M_{4}$};
  \node (A2) at (-3.9,3.7) {$M_{2}$};
 \node (A3) at (-3.9,2.25) {$M_{1}$};
  \node (A4) at (-3.9,0.8) {$M_{3}$};
  \node (A5) at (-3.9,-0.65) {$M_{4}$};
  \node (A6) at (-3.9,-2.2) {$M_{2}$};
 \node (A7) at (-3.9,-3.7) {$M_{1}$};
  \node (A8) at (-3.9,-5.2) {$M_{3}$};
 
  \node[blue,thick] (a1) at (-2,4.5) {$\left|\begin{smallmatrix}1\\7\end{smallmatrix}\right|$};
   \node (a2) at (-2,3) {$\left|\begin{smallmatrix}2\\8\end{smallmatrix}\right|$};
     \node (a3) at (-2,1.5) {$\left|\begin{smallmatrix}3\\9\end{smallmatrix}\right|$};
       \node (a4) at (-2,0) {$\left|\begin{smallmatrix}4\\10\end{smallmatrix}\right|$};
               \node (a5) at (-2,-1.5) {$\left|\begin{smallmatrix}5\\11\end{smallmatrix}\right|$};
		       \node (a6) at (-2,-3) {$\left|\begin{smallmatrix}6\\0\end{smallmatrix}\right|$};
			       \node[red,thick]  (a7) at (-2,-4.5) {$\left|\begin{smallmatrix}7\\1\end{smallmatrix}\right|$};

    \draw (A1) -- (a1);
      \draw[green,thick] (A2) -- (a1);     
      \draw[green,thick] (A2) -- (a2); 
       \draw[red,thick] (A3) -- (a2);   
       \draw[red,thick] (A3) -- (a3); 
        \draw[brown,thick]  (A4) -- (a3); 
        \draw [brown,thick]  (A4) -- (a4); 
        \draw[blue, thick] (A5) -- (a4);
        \draw[blue, thick] (A5) -- (a5);
      \draw[green,thick] (A6) -- (a5);     
      \draw[green,thick] (A6) -- (a6); 
       \draw[red,thick] (A7) -- (a6);   
       \draw [red,thick](A7) -- (a7); 
       \draw (A8) -- (a7);

\node (b1) at (0,3.75) {$\left[\begin{smallmatrix}1,7\\2,8\end{smallmatrix}\right]$};
   \node (b2) at (0,2.25) {$\left[\begin{smallmatrix}2,8\\3,9\end{smallmatrix}\right]$};
     \node (b3) at (0,0.75) {$\left[\begin{smallmatrix}3,9\\4,10\end{smallmatrix}\right]$};
       \node (b4) at (0,-0.75) {$\left[\begin{smallmatrix}4,10\\5,11\end{smallmatrix}\right]$};
         \node (b5) at (0,-2.25) {$\left[\begin{smallmatrix}5,11\\6,0\end{smallmatrix}\right]$};
          \node (b6) at (0,-3.75) {$\left[\begin{smallmatrix}6,0\\7,1\end{smallmatrix}\right]$};

  \node (C1) at (3.9,5.2) {$M_{1}$};
  \node (C2) at (3.9,3.7) {$M_{3}$};
 \node (C3) at (3.9,2.25) {$M_{4}$};
  \node (C4) at (3.9,0.8) {$M_{2}$};
  \node (C5) at (3.9,-0.65) {$M_{1}$};
  \node (C6) at (3.9,-2.2) {$M_{3}$};
 \node (C7) at (3.9,-3.7) {$M_{4}$};
  \node (C8) at (3.9,-5.3) {$M_{2}$};
          
  \node[red,thick]  (c1) at (2,4.5) {$\left|\begin{smallmatrix}7\\1\end{smallmatrix}\right|$};
   \node (c2) at (2,3) {$\left|\begin{smallmatrix}8\\2\end{smallmatrix}\right|$};
     \node (c3) at (2,1.5) {$\left|\begin{smallmatrix}9\\3\end{smallmatrix}\right|$};
       \node (c4) at (2,0) {$\left|\begin{smallmatrix}10\\4\end{smallmatrix}\right|$};
               \node (c5) at (2,-1.5) {$\left|\begin{smallmatrix}11\\5\end{smallmatrix}\right|$};
		       \node (c6) at (2,-3) {$\left|\begin{smallmatrix}0\\6\end{smallmatrix}\right|$};
			       \node[blue,thick]  (c7) at (2,-4.5) {$\left|\begin{smallmatrix}1\\7\end{smallmatrix}\right|$};

          \draw (C1) -- (c1);
      \draw[brown,thick] (C2) -- (c1);     
      \draw[brown,thick]  (C2) -- (c2); 
       \draw[blue, thick] (C3) -- (c2);   
       \draw[blue, thick] (C3) -- (c3); 
        \draw [green,thick](C4) -- (c3); 
        \draw[green,thick] (C4) -- (c4); 
        \draw[red,thick] (C5) -- (c4);
        \draw[red,thick] (C5) -- (c5);
      \draw[brown,thick]  (C6) -- (c5);     
      \draw[brown,thick]  (C6) -- (c6); 
       \draw[blue, thick](C7) -- (c6); 
              \draw[blue, thick] (C7) -- (c7);  
                \draw (C8) -- (c7);  

 \draw[blue,thick] (a1) -- (b1);
  \draw[red,thick] (a2) -- (b1);
    \draw[green,thick] (a2) -- (b2);
     \draw[brown,thick]  (a3) -- (b2);
      \draw[red,thick] (a3) -- (b3);
      \draw[blue, thick] (a4) -- (b3);
      \draw [brown,thick] (a4) -- (b4);
      \draw[green,thick] (a5) -- (b4);
      \draw[blue, thick] (a5) -- (b5);
      \draw[red,thick] (a6) -- (b5);
      \draw[green,thick] (a6) -- (b6);
      \draw [brown,thick] (a7) -- (b6);

 \draw[red,thick](c1) -- (b1);
  \draw[blue, thick]  (c2) -- (b1);
    \draw[brown,thick]  (c2) -- (b2);
     \draw[green,thick] (c3) -- (b2);
      \draw[blue, thick](c3) -- (b3);
      \draw[red,thick] (c4) -- (b3);
      \draw[green,thick] (c4) -- (b4);
      \draw[brown,thick]  (c5) -- (b4);
      \draw[red,thick] (c5) -- (b5);
      \draw[blue, thick] (c6) -- (b5);
      \draw[brown,thick]  (c6) -- (b6);
      \draw[green,thick] (c7) -- (b6);
      
  \draw[red,very thick,dotted] (2.2,4.5)-- (4.5,4.5) -- (4.5,-5.8) -- (-2,-5.8) --(-2,-4.7);
\draw[blue,very thick,dotted] (-2.2,4.5)-- (-4.5,4.5) -- (-4.5,-6.2) -- (2,-6.2) --(2,-4.7);
 \end{tikzpicture}
\end{center}
\caption{Inclusion diagram for $d=3$}~\label{F: Incldiag3}
\end{figure}
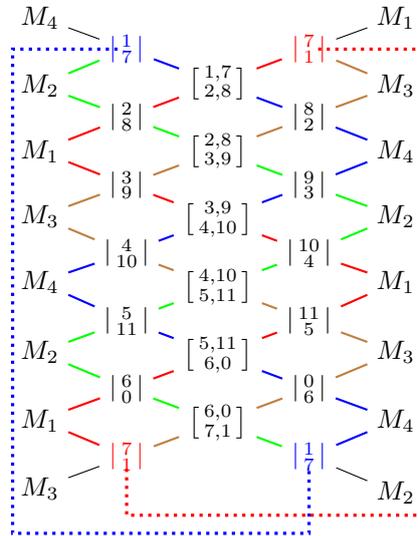

\section{The adjacent $Q$-pieces for $d=4$}~\label{A:tower4}
\input{Tower4.tex}

\vfill\eject


\begin{thebibliography}{99}

\bibitem{NAC}
N. A'Campo {\sl Signatures of monic polynomials} ArXiv:170205889v1.

\bibitem{Ar70}
V. J. Arnold, {\sl On some topological invariants of algebraic functions.} Trudy Mosk. Matem. Obshch. {\bf 21}, (1970), 27-46 (Russian).
  English transl. in Trans. Moscow Math. Soc. {\bf 21}, (1970), 30-52.

\bibitem{AVG1}
 V. I. Arnol'd, V. A. Vassiliev, V. V. Goryunov, and O. V. Lyashko, {\sl Singularities 1}, Contemporary Problems in Mathematics. Basic Directions [in Russian], Vol. 6, VINITI, Moscow (1988); English transl. in EMS, Springer-Verlag, Berlin - New York - Heidelberg.
 
 \bibitem{AVG2}
  V. I. Arnol'd, V. A. Vassiliev, V. V. Goryunov, and O. V. Lyashko. {\sl Singularities 2}, Contemporary Problems in Mathematics. Basic Directions [in Russian], Vol. 39, VINITI, Moscow (1989); English transl. in EMS, Springer-Verlag, Berlin - New York - Heidelberg.
  
  \bibitem{Art47}
   E. Artin, {\sl Theory of braids.} Ann. of Math. 48, No {\bf 1} (1947), 101-126.

\bibitem{Ba}
S. Barannikov. {\sl On the space of real polynomials without multiple critical values.} {\bf vol 26 }, (1992), Issue 2, 84-98 .


\bibitem{Bki}
N. Bourbaki. {\sl Lie groups Lie algebras.} Elements of Mathematics, Chapters 7-9 Springer Berlin (2005).

\bibitem{Bri71}
E. Brieskorn, {\sl Sur les groupes de tresses.} Sem. Bourbaki, {\bf 401}, novembre 1971 (Lecture Notes in Math., No {\bf 317}, (1973), 21-44).

\bibitem{Ce}
E. \v Cech, {\sl Th\'eorie g\'en\'erale de l'homologie quelconque.} Fund. Math.  {\bf 19}, (1932), 149-183.
  
\bibitem{CH} S. Chmutov, S. Duzin, J. Mostovoy {\it. Introduction to Vassiliev knot invariants} Cambridge univ. press, (2012).

\bibitem{C0}
N.C. Combe,  {\it On a new cell decomposition of a complement of the discriminant variety:
application to the cohomology of braid groups.} PhD Thesis (2018). 

\bibitem{C1} 
N.C. Combe {\it \v Cech cover of the complement of the discriminant variety}, arXiv: 1808.08411.

\bibitem{CoJ} 
N.C. Combe, V. Jug\'e {\it Counting bi-colored A'Campo forests.}  arXiv : 1702.07672.

\bibitem{Co1} N. Combe, {\it Geometric classification of real ternary octahedral quartics.} Discrete Computational Geometry, vol {\bf 60}, Issue 2, Springer (2018).

\bibitem{Co2}N. Combe, {\it On Coxeter algebraic varieties : the geometry of $CB_n$ quartics.} Math. Semesterberichte, (september 2018).

 \bibitem{CW72}
H. S. M. Coxeter, W. O.J Moser, {\it Generators and relations for discrete groups.} Third edition,  Springer-Verlag Berlin Heidelberg New York (1972).


\bibitem{DJS}
M. Davis, T. Januszkiewicz, R. Scott, {\sl The fundamental group of minimal blow-ups,}  arXiv: math/0203127 v2 (2002).

\bibitem{De}
P. Deligne, {\sl Les immeubles des groupes de tresses g\'en\'eralises.} Inv. math.  {\bf 17} (1972), 273-302.

\bibitem{DeM}
P. Deligne, D. Mumford {\sl The irreducibility of the space of curves of given genus} Inv. math.  {\bf 17} Publications Mathématiques de l'IHES,  Tome {\bf 36}  (1969),  p. 75-109.


\bibitem{Dev}
S. Devadoss, {\sl Tesselations of moduli spaces and the mosaic operad}, Homotopy Invariant Algebraic
Structures, Contemporary Mathematics {\bf 239} (1999), 91-114.

\bibitem{EHKR}  
P. Etingof, A. Henriques, J. Kamnitzer, EM. Rains.  {\sl The cohomology ring of the real locus of the moduli space of stable curves of genus 0 with marked points},Annals of mathematics, {\bf 171}, n 2, (2010).

\bibitem{FaN62}
E. Fadell, L. Neuwirth, {\sl Configurations spaces}. Math Scand, {\bf 10}, (1962), 111-118.

\bibitem{FoN62}
R. Fox, L. Neuwirth, {\sl The braid groups}. Math Scand, {\bf 10}, (1962), 119-126.

\bibitem{Fu}
D. B. Fuchs, {\it Classical varieties}, Contemporary Problems in Mathematics. Basic Directions [in Russian], Vol. 12, VINITI, Moscow (1986), pp. 253?314; English transl. in EMS, Springer-Verlag, Berlin - New York - Heidelberg.

\bibitem{FM}
 W. Fulton, R. MacPherson  {\it A Compactification of Configuration Spaces }, Annals of Mathematics
Second Series, {\bf Vol. 139}, No. 1 (Jan., 1994), pp. 183-225.

\bibitem{Gauss}
C. F. Gauss, {\sl Werke. } Band III. Georg Olm Verlag, Hildescheim, 1973. 


\bibitem{GhysA}
E. Ghys, {\sl A singular mathematic promenade}. ArXiv1612.06373v1, (2016).

\bibitem{Ghys}
E. Ghys, {\sl A singular mathematic promenade}. Lyon ENS \'editions, impr. 2017. 


\bibitem{Gro84}
A. Grothendieck. {\sl Esquisse d'un Programme}.  [archive] (1984).

\bibitem{Ka93v1}
M. Kapranov, {\sl Veronese curves and the Grothendieck-Knudsen moduli space $\overline{M}_{0,n}$}, Journal of Algebraic Geometry {\bf 2} (1993). 

\bibitem{Ka93}
M. Kapranov, {\sl The permutoassociahedron, MacLane's coherence theorem and asymptotic zones for the KZ equation} Journal of Pure an Applied Algebra {\bf 85} (1993) 119-142

\bibitem{Kee}
S. Keel, {\sl Intersection theory of moduli spaces of stable $n$-pointed curves},  Trans. Amer. Math Soc {\bf 108} (1992), 545-574.    

\bibitem{Kho96}
M. Khovanov { Real $K(\pi,1)$   arrangements from  finite root systems},  Math. Research Letters (1996), 261-274.                 

\bibitem{KSV}
T. Kimura, J. Stasheff, A. Voronov, {\sl Homology of moduli of curves and commutative homotopy
algebras} in The Gelfand Mathematical Seminars (1993-95), Birkh\"auser, (1996).

\bibitem{Knu}
F.F. Knudsen, {\sl The projectivity of the moduli space of stable curves II. The stacks $\overline{M}_{0,n}$},  Trans. Amer. Math Soc {\bf 108} (1992), 163-199.   

\bibitem{Kre72}
J. Kreweras, {\sl Sur les partitions non-crois\'ees d'un cycle}, Discrete Math  {\bf 1}, (1972), 333-350.

\bibitem{Mo} J. Mostovoy {\sl Introduction to Vassiliev knot invariants}, Cambridge, 2012.   


\bibitem{Sta}
J. Stasheff, {\sl  Homotopy associativity of H-spaces  I, II}  Trans. Amer. Math. Soc {\bf 108}  (1963), 275-292.       
  
\bibitem{Va1}
 V.A. Vassiliev. {\sl Topology of plane arrangements and their complements}. arXiv:1407.7236

\bibitem{Va2}
V.A. Vassiliev {\sl How to calculate homology groups of spaces of nonsingular algebraic projective hypersurfaces}	arXiv:1407.7229 

\bibitem{Va3}
V.A. Vassiliev { \sl A few problems on monodromy and discriminants} arXiv:1504.01997

\end{thebibliography}
\end{document}